\documentclass[11pt]{amsart}
\usepackage{fouriernc} 

\usepackage{thmtools}
\usepackage{thm-restate}

\usepackage{Definitions}    

\usepackage{PageSetup}      
\usepackage{Environments}   
\usepackage{array}
\usepackage{enumerate}

\usepackage{setspace}

\usepackage{xr}
\usepackage{afterpage}
\usepackage{morefloats}
\usepackage{mathtools}
\usepackage{lscape}
\usepackage{multirow}
\usepackage{enumitem}
\usepackage{graphicx}
\usepackage{wasysym}
\usepackage{scalerel}
\usepackage{stmaryrd}

\usepackage{tabularx}

\usepackage{soul}
\setul{}{1.5pt}

\theoremstyle{theorem}

\colorlet{mylight}{green!40!white}
\colorlet{mymed}{red!60!white}
\colorlet{mydark}{blue!80!white}
\colorlet{mylightfill}{green!20!white}
\colorlet{mymedfill}{red!30!white}
\colorlet{mydarkfill}{blue!40!white}

\colorlet{mydarkred}{red!80!black}
\colorlet{mydarkblue}{blue!70!black}
\colorlet{mydarkgreen}{green!60!black}

\usepackage[all]{xy}
\usepackage{tikz, tikz-3dplot, tikz-cd}
\usetikzlibrary{arrows.meta}
\tikzset{>={Latex[width=3mm,length=3mm]}}
\usetikzlibrary{decorations.markings}
\usetikzlibrary{decorations.pathreplacing,angles,quotes}

\usetikzlibrary{backgrounds}
\tdplotsetmaincoords{70}{120}
\usetikzlibrary{patterns}

\DeclareMathOperator{\tr}{tr}

\newcommand{\fib}[3]{{#1}^{\times _{#2}{#3}}}
\newcommand{\ful}[2]{\Psi^{#2}({#1})}
\newcommand{\fful}[3]{\Psi^{#2}_{#3}({#1})}
\newcommand{\xto}{\xrightarrow}

\newcommand{\fibra}{\mathrm{fib}}

\newcommand{\tiff}{if and only if }

\providecommand{\Ex}{\cE\mathrm{x}}
\DeclareMathOperator{\THH}{\textup{THH}}
\DeclareMathOperator{\TR}{\textup{TR}}

\def\hto{\xslashedrightarrow{}}
\def\calBi#1#2{\ensuremath{\mathscr{#1}\!/\!_{\mathbf{#2}}}}
\def\bicat#1#2{\ensuremath{\mathcal #1 \!/\!_{#2}}}
\newcommand{\Ft}{\ensuremath{F_{\mathrm{tr}}}}

\usepackage{tikz}
\usepackage{subcaption}

\newcommand{\R}{\mathbb{R}}
\newcommand{\Z}{\mathbb{Z}}
\newcommand{\ra}{\longrightarrow}
\newcommand{\Sph}{\mathbb{S}}

\DeclareMathOperator{\barsmash}{\overline{\wedge}}

\newcommand{\mc}{\mathcal}
\newcommand{\simar}{\overset\sim\to}

\DeclareMathOperator{\End}{\textup{End}}

\newcommand{\bcl}[3]{\left[\begin{gathered}{#1}\xto{#2}{#3}\end{gathered}\right]}
\newcommand{\bcr}[3]{\left[\begin{gathered}{#3}\xleftarrow{#2}{#1}\end{gathered}\right]}

\makeatletter
\def\slashedarrowfill@#1#2#3#4#5{%
  $\m@th\thickmuskip0mu\medmuskip\thickmuskip\thinmuskip\thickmuskip
   \relax#5#1\mkern-7mu%
   \cleaders\hbox{$#5\mkern-2mu#2\mkern-2mu$}\hfill
   \mathclap{#3}\mathclap{#2}%
   \cleaders\hbox{$#5\mkern-2mu#2\mkern-2mu$}\hfill
   \mkern-7mu#4$%
}
\def\rightslashedarrowfill@{%
  \slashedarrowfill@\relbar\relbar\mapstochar\rightarrow}
\newcommand\xslashedrightarrow[2][]{%
  \ext@arrow 0055{\rightslashedarrowfill@}{#1}{#2}}
\makeatother
\def\hto{\xslashedrightarrow{}}

\newcommand{\fa}{\ensuremath{\mathfrak{a}}}
\newcommand{\frl}{\ensuremath{\mathfrak{l}}}
\newcommand{\fr}{\ensuremath{\mathfrak{r}}}

\newcommand{\rdual}[1]{{#1}^\star}

\newcommand{\seqco}[2]{{#1}\circ \ldots \circ {#2}}

\usepackage{braket}

\newcommand{\sh}[1]{{\ensuremath{\hspace{1mm}\makebox[-1mm]{$\langle$}\makebox[0mm]{$\langle$}\hspace{1mm}{#1}\makebox[1mm]{$\rangle$}\makebox[0mm]{$\rangle$}}}}

\newcommand{\bigsh}[1]{{\ensuremath{\hspace{1mm}\makebox[-1mm]{$\big\langle$}\makebox[0mm]{$\big\langle$}\hspace{1mm}{#1}\makebox[1mm]{$\big\rangle$}\makebox[0mm]{$\big\rangle$}}}}
\newcommand{\Bigsh}[1]{{\ensuremath{\hspace{1mm}\makebox[-1mm]{$\Big\langle$}\makebox[0mm]{$\Big\langle$}\hspace{1mm}{#1}\makebox[1mm]{$\Big\rangle$}\makebox[0mm]{$\Big\rangle$}}}}

\newcommand{\odots}[1]{\odot}

\usetikzlibrary{shapes,snakes}
\tikzset{pb/.style={draw,regular polygon,regular polygon sides=3,inner sep=0pt,shape border rotate=180,font=\scriptsize}}
\tikzset{pbflat/.style={draw,regular polygon,regular polygon sides=3,shape border rotate=270,inner sep=1pt,font=\scriptsize, fill=white}}

\usepackage{xkeyval}
\makeatletter

\define@cmdkey      [PRE] {fam}     {A}  {}
\define@cmdkey      [PRE] {fam}     {B}  {}
\define@cmdkey      [PRE] {fam}     {C}  {}
\define@cmdkey      [PRE] {fam}     {D}  {}
\define@cmdkey      [PRE] {fam}     {f}  {}
\define@cmdkey      [PRE] {fam}     {g}  {}
\define@cmdkey      [PRE] {fam}     {h}  {}
\define@cmdkey      [PRE] {fam}     {k}  {}
\define@cmdkey      [PRE] {fam}     {E}  {}
\define@cmdkey      [PRE] {fam}     {F}  {}
\define@cmdkey      [PRE] {fam}     {G}  {}
\define@cmdkey      [PRE] {fam}     {H}  {}
\define@cmdkey      [PRE] {fam}     {phi}  {}
\define@cmdkey      [PRE] {fam}     {gamma}  {}
\define@cmdkey      [PRE] {fam}     {eta}  {}
\define@cmdkey      [PRE] {fam}     {kappa}  {}
\define@cmdkey      [PRE] {fam}     {a}  {}
\define@cmdkey      [PRE] {fam}     {b} {}
\define@cmdkey      [PRE] {fam}     {c} {}
\define@cmdkey      [PRE] {fam}     {d} {}
\define@cmdkey      [PRE] {fam}     {wi} {}
\define@cmdkey      [PRE] {fam}     {he} {}

\presetkeys         [PRE] {fam} {
	A = ,
	B = ,
	C = ,
	D = ,
	f = ,
	g = ,
	h = ,
	k = ,
	E = ,
	F = ,
	G = ,
	H = ,
	phi = ,
	gamma = ,
	eta = ,
	kappa = ,
	a = ,
	b = ,
	c = ,
	d = ,
	wi = ,
	he = ,
	}{}

\newcommand{\commutingcube}[1]{
	\setkeys[PRE]{fam}{#1}
\begin{tikzcd}[ampersand replacement=\&,column sep={\cmdPRE@fam@wi},row sep={\cmdPRE@fam@he}]
			\& \cmdPRE@fam@B \ar[ddd, near end, "\cmdPRE@fam@b"]\ar[rr, "\cmdPRE@fam@g"]	\&\& \cmdPRE@fam@D \ar[ddd, "\cmdPRE@fam@d"] 
				\\
	\cmdPRE@fam@A \ar[ddd, "\cmdPRE@fam@a"]\ar[rr,near start,  crossing over,  "\cmdPRE@fam@h"]\ar[ru, "\cmdPRE@fam@f"]	\&\& \cmdPRE@fam@C \ar[ru, "\cmdPRE@fam@k"]	
				\\ \\
			{}\& \cmdPRE@fam@F \ar[rr, near start, "\cmdPRE@fam@gamma"]		\&\& \cmdPRE@fam@H
				 \\
	\cmdPRE@fam@E \ar[rr, "\cmdPRE@fam@eta"]\ar[ru, "\cmdPRE@fam@phi"]		\&\& \cmdPRE@fam@G \ar[ru, "\cmdPRE@fam@kappa"]\arrow [from=uuu,  near start,crossing over,  "\cmdPRE@fam@c"]
\end{tikzcd}
}

\newcommand{\commutingsquare}[1]{
	\setkeys[PRE]{fam}{#1}
\begin{tikzcd}[ampersand replacement=\&,column sep={\cmdPRE@fam@wi},row sep={\cmdPRE@fam@he}]
					\& \cmdPRE@fam@B \ar[rr,"\cmdPRE@fam@g"]			\&	\& \cmdPRE@fam@D  \\
		\cmdPRE@fam@A \ar[rr,"\cmdPRE@fam@h"]\ar[ru,"\cmdPRE@fam@f"]	\&\& \cmdPRE@fam@C \ar[ru,"\cmdPRE@fam@k"]	\&
\end{tikzcd}
}

\newcommand{\triangularprism}[1]{
	\setkeys[PRE]{fam}{#1}
\begin{tikzcd}[ampersand replacement=\&,column sep={\cmdPRE@fam@wi},row sep={\cmdPRE@fam@he}]
		\& \cmdPRE@fam@B \ar[dd, near start, "\cmdPRE@fam@b"]\ar[rd,"\cmdPRE@fam@g"]
		\\
	\cmdPRE@fam@A \ar[dd,"\cmdPRE@fam@a"]\ar[rr, crossing over, near start,  "\cmdPRE@fam@h"]\ar[ru,"\cmdPRE@fam@f"] 
		\&\& \cmdPRE@fam@C \ar[dd,"\cmdPRE@fam@c"] 
		\\
		\& \cmdPRE@fam@F \ar[rd,"\cmdPRE@fam@gamma"] 
		\\
		\cmdPRE@fam@E \ar[rr,"\cmdPRE@fam@eta"]\ar[ru,"\cmdPRE@fam@phi"] \&\& \cmdPRE@fam@G
\end{tikzcd}
}

\newcommand{\triangularprismtwo}[1]{
	\setkeys[PRE]{fam}{#1}
	
\begin{tikzcd}[ampersand replacement=\&,column sep={\cmdPRE@fam@wi},row sep={\cmdPRE@fam@he}]
	\cmdPRE@fam@B \ar[dd,"\cmdPRE@fam@b"']\ar[rr,"\cmdPRE@fam@g"]\ar[rd,"\cmdPRE@fam@h"'] 
		\&\& \cmdPRE@fam@D \ar[dd,"\cmdPRE@fam@d"] 
		\\
	\& \cmdPRE@fam@C \ar[ru, "\cmdPRE@fam@k"']	
		 \\
	\cmdPRE@fam@F \ar[rr, near start, "\cmdPRE@fam@gamma"]\ar[rd,"\cmdPRE@fam@eta"'] 
		\&\& \cmdPRE@fam@H
		 \\
	\& \cmdPRE@fam@G \ar[ru,"\cmdPRE@fam@kappa"']\arrow [from =uu, near start,  crossing over, "\cmdPRE@fam@c"]
\end{tikzcd}
	
}

\makeatother

\hypersetup{
 pdfkeywords={latex starter,template},
 pdfauthor={Niles Johnson},
}

\usepackage{tikz-cd} 

\theoremstyle{theorem}
\newtheorem{thmx}{Theorem}

\subjclass[2010]{55M20, 55P42, 18D05, 55R70, 18D30, 55P91}

\keywords{periodic points, topological restriction homology, bicategorical trace, parameterized spectra, Nielsen theory, Reideister trace}

\title{Periodic points and topological restriction homology}
\author{Cary Malkiewich}  
\email{malkiewich@math.binghamton.edu}
\address{Binghamton University, PO Box 6000, Binghamton, NY 13902}
\author{Kate Ponto}
\email{kate.ponto@uky.edu}
\address{University of Kentucky, 719 Patterson Office Tower, Lexington, KY 40506}
\date{19th Jun, 2020}

\begin{document}

 \maketitle
 
{\centering \small \it Dedicated to Bruce Williams (1945-2018)\par}

\begin{abstract}
	We answer in the affirmative two conjectures made by Klein and Williams. First, in a range of dimensions, the equivariant Reidemeister trace defines a  complete obstruction to  removing $n$-periodic points from a self-map $f$. 
Second, this obstruction defines a class in topological restriction homology. 

We prove these results using duality and trace for bicategories. This allows for immediate generalizations, including a corresponding theorem for the fiberwise Reidemeister trace.
\end{abstract}
 
\setcounter{tocdepth}{1}
 \tableofcontents


\section{Introduction}

For a finite simplicial complex $X$ and a continuous map $f\colon X\to X$, the {\bf Lefschetz number} $L(f) \in \Z$ is a weighted sum of the fixed points of $f$. This invariant admits many generalizations. In this paper, we focus on generalizations that count the fixed points of $f^n$, or the {$n$-periodic points} of $f$.

Since it is a weighted sum, the Lefschetz number detects the presence of fixed points for any endomorphism in the homotopy class of $f$. However it does not give a sharp lower bound on the number of fixed points.  For that we need to refine the Lefschetz number to the \textbf{Reidemeister trace} $R(f)$.  This invariant takes values in the 0th homology group 
of the twisted free loop space of $f$,
\[\Lambda^{f} X\coloneqq \{ \ \gamma\in X^I \ | \ f(\gamma(1))=\gamma(0) \ \}.\]
If $X$ is a compact manifold of dimension at least 3, the Reidemeister trace is a complete obstruction to the removal of fixed points \cite{jiang,shi,wecken}. 

In this paper we compare several refinements of the Lefschetz number and Reidemeister trace for periodic points, the weakest of which are the Lefschetz number and Reidemeister trace for $f^n$. To build the others, we use Fuller's observation that the fixed points of the map
\[ \xymatrix @R=2pt {
	X\times \cdots \times X\ar[r]^{\ful{f}{n}} & X\times \cdots \times X \\
	(x_1,x_2,\ldots , x_n)\ar@{|->}[r] & (f(x_n), f(x_1), \ldots , f(x_{n-1}))
}\]
are precisely the periodic points of $f$ of period $n$ \cite{fuller,komiya,d:iterates}. 
If $f$ is homotopic to a map that has no $n$-periodic points, then $\ful{f}{n}$ is homotopic to a map with no fixed points.  

The map ${\ful{f}{n}} $
is equivariant with respect to the action of $C_n = \Z/n\Z$ that rotates coordinates. We can refine the observation above to say if
$f$ is homotopic to a map that has no $n$-periodic points, then $\ful{f}{n}$ is $C_n$-equivariantly homotopic to a map with no fixed points.  
Therefore the equivariant Reidemeister trace of $\ful fn$, which we also call the $n$th \textbf{Fuller trace}, is an obstruction to removing the $n$-periodic points from $f$. 
The Fuller trace is a map of equivariant spectra
	\[ R_{C_n}(\ful fn)\colon \Sph \to \Sigma^\infty_+ \Lambda^{\ful fn} X^n \]
or equivalently a map of spectra
	\[ R(\ful fn)^{C_n}\colon \Sph \to \left(\Sigma^\infty_+ \Lambda^{\ful fn} X^n\right)^{C_n}. \]
The following comparison theorem is the main result of the paper.
\begin{thm}[Theorems \ref{thm:main_first_half} and \ref{thm:main_second_half}]\label{thm:simple_main_thm}  
	Let $X$ be a finitely dominated space. Then the following diagram commutes up to homotopy for each $k \mid n$.
		\[\xymatrix@C=40pt@R=40pt{
 			&&\Sph \ar[dll]_-{R(\ful{f}n)^{C_n}} 
			\ar[dl]^-{R(\ful{f}k)^{C_k}} 
			\ar[d]^-{R(\ful{f}k)} 
		\ar[rd]^{R(f^k)} 
		\\
		(\Sigma^\infty_+ \Lambda^{\ful{f}n} X^n)^{C_n} \ar[r]^-{R} 
		&(\Sigma^\infty_+ \Lambda^{\ful{f}k} X^k)^{C_k} 
 			\ar[r]^-{F}
		&	\Sigma^\infty_+ \Lambda^{\ful{f}k} X^k \ar[r]^-{\simeq} 
		&		\Sigma^\infty_+ \Lambda^{f^k} X
		}\]
\end{thm}

\noindent
Here ``finitely dominated'' means that $X$ is a retract up to homotopy of a finite CW complex. This is essentially the most general case in which $R(f)$ is defined. The maps $R$ and $F$ are the natural analogs of the ``restriction'' and ``Frobenius'' maps from the theory of topological Hochschild homology, defined as in \cite[\S 2.5]{madsen_survey}. The homotopy equivalence equivalence at the bottom-right is given by the maps
\begin{align}\label{eq:forget_equiv}
	\begin{split}
	\{ \ \gamma_1,\ldots,\gamma_k\in X^I \ | \ f(\gamma_i(1))=\gamma_{i+1}(0) \ \}
	&\longrightarrow 
	\{ \ \gamma\in X^I \ | \ f^k(\gamma(1))=\gamma(0) \ \} \\
	(\gamma_1,\ldots,\gamma_k)
	& \longmapsto 
	f^{k-1}(\gamma_2) \cdot f^{k-2}(\gamma_3) \cdot \ldots \cdot f(\gamma_k) \cdot \gamma_1 \\
	(\gamma,c_{f(\gamma(1))},c_{f^2(\gamma(1))},\ldots,c_{f^{k-1}(\gamma(1))})
	& \longmapsfrom  \gamma
	\end{split}
\end{align}
where $c$ denotes the constant path at $x$.

\cref{thm:simple_main_thm} gives the following answer to a conjecture of Klein and Williams \cite{kw:equiv}.

\begin{cor}\label{cor:kw_solved}
	The Reidemeister traces $\{R(f^k)\}$ can be recovered from the Fuller trace $R(\ful{f}n)^{C_n}$.  
	The vanishing of $R(\ful{f}n)^{C_n}$ implies the vanishing of $R(f^k)$ for all $k | n$.
\end{cor}

\noindent
When combined with the main result of \cite{jezierski}, this implies

\begin{cor}\label{cor:converse}
	If $X$ is a compact manifold of dimension at least 3, the Fuller trace $R(\ful{f}n)^{C_n}$ vanishes in the homotopy category of spectra \tiff $f$ is homotopic to a map with no $n$-periodic points.
\end{cor}

\noindent
In other words, for high-dimensional manifolds the Fuller trace is a complete obstruction to the removal of $n$-periodic points.

Though our motivation for \cref{thm:simple_main_thm} comes mainly from dynamics, it also has important implications for algebraic $K$-theory. These implications can be succinctly expressed
by the slogan ``topological restriction homology ($\TR$) is the most natural home for periodic-point invariants".

More precisely, recall that the \textbf{topological restriction homology} $\TR(A)$ of a ring spectrum $A$ is the homotopy limit of $\THH(A)^{C_n}$ along restriction maps 
\[R\colon \THH(A)^{C_n} \to \THH(A)^{C_k}\] for $k \mid n$, see \cite{bhm,madsen_survey} for details. If $A = \Sigma^\infty_+ \Omega X$ for a connected CW complex $X$, $\THH(A) \simeq \Sigma^\infty_+ \Lambda X$, the suspension spectrum of the free loop space of $X$. 
The topological restriction homology of this ring, denoted $\TR(X)$, is the homotopy limit of the fixed point spectra $(\Sigma^\infty_+ \Lambda X)^{C_n}$ along the restriction maps 
\[R: (\Sigma^\infty_+ \Lambda X)^{C_n} \to (\Sigma^\infty_+ \Lambda X)^{C_k}.\] Concretely, the restriction map takes each $C_n$-equivariant map $S^V \to S^V \sma (\Lambda X)_+$ to its $C_{n/k}$-fixed points:
\[ \xymatrix{ S^{V^{C_{n/k}}} \ar[r] & S^{V^{C_{n/k}}} \sma (\Lambda X)^{C_{n/k}}_+ \ar@{<->}[r]^-\cong & S^{V^{C_{n/k}}} \sma (\Lambda X)_+ } \]
If $X$ is equipped with an endomorphism $f$, we define the {\bf twisted topological restriction homology} of $X$ as
\[ TR(X,f) \coloneqq \underset{n,R}\holim\, (\Sigma^\infty_+ \Lambda^{\ful{f}n} X^n)^{C_n} \]
where the restriction maps $R$ are the maps of \cref{thm:simple_main_thm}. Concretely, they take each $C_n$-equivariant map $S^V \to S^V \sma (\Lambda^{\ful{f}n} X^n)_+$ to its $C_{n/k}$-fixed points:
\[ \xymatrix{ S^{V^{C_{n/k}}} \ar[r] & S^{V^{C_{n/k}}} \sma (\Lambda^{\ful{f}n} X^n)^{C_{n/k}}_+ \ar@{<->}[r]^-\cong & S^{V^{C_{n/k}}} \sma (\Lambda^{\ful{f}k} X^k)_+ } \]
The justification for the name is that $\TR(X;f)$ agrees with a more general definition of $\TR(A;M)$ for any ring spectrum $A$ and bimodule $M$, defined in a similar way to Lindenstrauss and McCarthy's $W$-theory \cite{LM12} for ordinary rings. Details of this construction will appear in \cite{clmpz}.

By the tom Dieck splitting theorem, each of the spectra in the homotopy limit system for $\TR(X;f)$ splits as a finite product of homotopy orbit spectra
\begin{equation}\label{intro_splitting_1}
(\Sigma^\infty_+ \Lambda^{\ful{f}n} X^n)^{C_n} \simeq \prod_{k | n} (\Sigma^\infty_+ \Lambda^{\ful{f}k} X^k)_{hC_k},
\end{equation}
and the restriction map $R$ simply projects onto a subset of the factors. It follows that the homotopy limit $\TR(X;f)$ is an infinite product of homotopy orbit spectra
\begin{equation}\label{intro_splitting_2}
	\TR(X;f) \simeq \prod_{n \geq 1} (\Sigma^\infty_+ \Lambda^{\ful{f}n} X^n)_{hC_n}.
\end{equation}
Therefore, to define a class in $\pi_0\TR(X;f)$, it is enough to give a class in $\pi_0$ of the spectrum \eqref{intro_splitting_1} for each $n \geq 1$, agreeing along the restriction maps.

\cref{thm:simple_main_thm} says that the Fuller traces $R(\ful{f}n)^{C_n}$ give such a collection of classes. In particular, the commutativity of the left-hand triangle implies they agree along the restriction maps. Therefore they define a class in $\pi_0\TR(X;f)$ that we might call the ``infinite Fuller trace'' $R(\ful f\infty)^{C_\infty}$. Concretely, this is the element of the product \eqref{intro_splitting_2} whose $n$th term is recovered from $R(\ful{f}n)^{C_n}$ using the tom Dieck splitting \eqref{intro_splitting_1}. By \cref{cor:converse}, the infinite Fuller trace is a complete obstruction to removing $n$-periodic points for any value of $n$.  This is the precise interpretation of the slogan, ``periodic point invariants most naturally live in $\TR$.'' 

This slogan has been articulated before. Klein, McCarthy, Williams, and others have remarked that one should be able to construct a trace map from endomorphism $K$-theory $K(\End_{A}(M))$ to $\TR(A;M)$ for any ring spectrum, as in \cite{LM12}. Furthermore,  there should be a class $[f]\in K(\End_A(M))$ whose image in  $\pi_0 \TR(X,f)$ 
recovers the Reidemeister traces $R(f^n)$ of all the composites. Earlier results in this spirit can be found in \cite{grayson,gn,Iwashita,luck}, but this particular result will be developed in \cite{clmpz}. 
Granting this, this defines a periodic point invariant in $\pi_0 \TR(X,f)$ without reference to the  Fuller construction. 
From this point of view, the additional insight provided by \cref{thm:simple_main_thm} is that this class can be explicitly described as the trace of the Fuller map. 

\subsection*{Fiberwise invariants}
Following \cite{dp,gn,nicas,p:thesis,p:coincidences},  we interpret the Lefschetz number and Reidemeister trace as stable homotopy classes of maps rather than numbers. One of the primary advantages of this approach is that it allows for easy generalizations to the fiberwise and equivariant settings. Using this perspective, the following result has the same proof as its classical analog.

\begin{thm}\label{thm:fiberwise_main_thm}
	The variants of \cref{thm:simple_main_thm,cor:kw_solved} for a family of fiberwise endomorphisms $f\colon E \to E$ over $B$ also hold, provided $E \to B$ is a fibration with finitely dominated fiber.
\end{thm}

On the other hand, the fiberwise version of \cref{cor:converse} is the following conjecture. It will require a very different set of techniques, and we plan to take it up in future work.
\begin{conj}
	The fiberwise Fuller trace $R_B(\ful{f}{n})^{C_n}$ is the complete obstruction to the removal of $n$-periodic points from a family of endomorphisms $f\colon E \to E$ over $B$, when $B$ is a finite-dimensional cell complex and $E \to B$ is a smooth closed manifold bundle whose fiber $M$ has dimension at least $3 + \dim B$.
\end{conj}
Note that the special case of $n=1$ is proven in \cite[Cor 10.5]{kw}.

For higher values of $n$, we could have instead formulated the conjecture using the collection of Reidemeister traces $\{ R_B(f^k) : k \mid n \}$, but we expect that version of the conjecture to be false. In other words, we expect that the Reidemeister traces of the iterates do not form a complete obstruction to removing $n$-periodic points from bundles, in contrast to the case of a single endomorphism \cite{jezierski}. The reason for our expectation is that the product over $k \mid n$ of the maps in the bottom row of \cref{thm:simple_main_thm} is injective on $\pi_0$, but fails to be injective above $\pi_0$. As a result, once we start measuring the higher homotopy groups by looking at families of endomorphisms, we might find a Fuller trace that lies in the kernel, so that the corresponding Reidemeister traces are all zero. 
A counterexample of the following form would help settle this question.
\begin{conj}
	There is a family of endomorphisms $f\colon E \to E$ over some base $B$ for which $R_{B}(f)$ and $R_{B}(f^2)$ are zero, but the fiberwise Fuller trace $R_{B}(\ful f2)^{C_2}$ is nonzero.
\end{conj}

\subsection*{Organization}
We first give a short proof of \cref{thm:simple_main_thm} in the special case where $X$ is a compact ENR. We then proceed with the general case. The proof 
splits into two pieces, and these proofs are the first two parts of this paper.  In \cref{part:unwinding} we prove that the right triangle commutes using the string diagram 
calculus developed in \cite{mp2}.  In \cref{part:varying} we prove the left two triangles commute, by extending certain functors on the category of equivariant spectra to shadow functors on the bicategory of equivariant parameterized spectra. In \cref{part:fiberwise} we prove \cref{thm:fiberwise_main_thm}.

\subsection*{Acknowledgments} The authors are pleased to acknowledge contributions to this project that emerged from enjoyable conversations with Manuel Araujo, Jonathan Campbell, Ross Geoghegan, Niles Johnson, Inbar Klang, John Lind, Randy McCarthy, and Mike Shulman.
They are indebted to John Klein and Bruce Williams for asking the questions that motivated this work. The first author thanks the Max Planck Institute in Bonn for their hospitality while the majority of this paper was written.  The second author was partially supported by a Simons Collaboration Grant and NSF grant DMS-1810779.


\section{The case of a single smooth manifold or compact ENR}

There are two approaches to proving \cref{thm:simple_main_thm}.  The first is  a more classical and geometrically motivated path starting from an explicit descriptions of the Reidemeister trace.  This builds on ideas both explicit and implicit in  \cite{cj,dold:index,dold:transfer}, and a complete description of this version of the Reidemiester trace can be found in \cite{spectra_notes}.  Alternatively, there is a more formal and category theoretic approach that follows the understanding of fixed point invariants as traces in symmetric monoidal categories or bicategories. \cite{dp,p:thesis}.  

These approaches both require significant effort to implement but the work in each case is very different.  
We have chosen to follow the second approach
since we find it does not require the same level of  outsourcing to papers such as  \cite{ms,spectra_notes}.
Despite this preference, we find the geometric approach provides very useful intuition. To benefit from these insights we first sketch the alternative proof of \cref{thm:simple_main_thm} for a single manifold or compact ENR.  

Let $X$ be a compact topological space, with a topological embedding $i\colon X \to V$ into an open subset $V \subseteq \R^N$ and a retract $p\colon V \to X$, making $X$ into a compact ENR.
Choose $\epsilon>0$ so that the closed $\epsilon$-tube about $X$ is completely contained in $V$.
The Reidemeister trace $R(f)$ is the map of spectra obtained by formal de-suspension of the following map of spaces.
\begin{eqnarray*}
	S^N & \longrightarrow & S^N_\epsilon \sma (\Lambda^f X)_+ \\
	v & \longmapsto & \left\{
	\begin{array}{rl}
		\left(v - f(p(v))\right) \sma \gamma_{f(p(v)),v} & \textup{if } v \in V \textup{ and } \|v - f(p(v))\| \leq \epsilon \\
		{*} & \textup{otherwise}
	\end{array}\right.
\end{eqnarray*}
Here $S^N_\epsilon$ is a sphere of radius $\epsilon$, obtained by quotienting the complement of an open ball of radius $\epsilon$ in $\R^N$:
\[ S^N_\epsilon = \R^N/(\R^N - B_\epsilon) \cong \overline{B_\epsilon}/\partial \overline{B_\epsilon}. \]
The path $\gamma_{f(p(v)),v)}$ is defined by the formula
\[ \gamma_{f(p(v)),v}(t) = p[(1-t)f(p(v)) + tv]. \]
The condition on $\epsilon$ guarantees that $p$ is defined on the line segment from $f(p(v))$ to $v$, so that this path is well-defined. 
See  \cite[\S 7.7]{spectra_notes} for a discussion of how this description of the Reidemeister trace arises from the more categorical descriptions later in this paper.

Note that the homotopy class of this map does not depend on the choice of $\epsilon$.

The $C_n$-space $X^{\times n}$ becomes a $C_n$-equivariant ENR using the product embedding $i^{\times n}$ into the $C_n$-representation $\R^{nN} = \textup{Ind}^{C_n} \R^N$, and the product projection $p^{\times n}$. The Fuller trace $R_{C_n}(\ful fn)$ is given by the equivariant version of the above map, de-suspended by the $C_n$-representation $\R^{nN}$  \cite[\S 9.5]{spectra_notes}. It is a map 
\[S^{nN}  \longrightarrow  S^{N}_\epsilon \sma \ldots \sma S^{N}_\epsilon \sma (\Lambda^{\ful{f}n} X^n)_+\]
and  tor tuples $(v_1,\ldots ,v_n)$  where $v_i \in V$ and $\|v_{i+1} - f(p(v_i))\| \leq \epsilon$ for every $i$, it is given by 
\begin{eqnarray*}
	(v_1,\ldots,v_n) & \longmapsto & \left(v_1 - f(p(v_n))\right) \sma \left(v_2 - f(p(v_1))\right) \sma \left(v_3 - f(p(v_2))\right) \sma \ldots \\
	&& \sma \left( \gamma_{f(p(v_n)),v_1}, \gamma_{f(p(v_1)),v_2}, \gamma_{f(p(v_2)),v_3}, \ldots \right)
\end{eqnarray*}
Everywhere else it is zero.

To prove that the left-hand triangle of \cref{thm:simple_main_thm} commutes, it is enough to observe that taking $C_{n/k}$-fixed points of this map replaces the $n$ by $k$. The middle triangle of \cref{thm:simple_main_thm} commutes since forgetting the $C_n$ action, we have the  formula for the non-equivariant Reidemeister trace of $\ful{f}n$.
The right-hand triangle, on the other hand, does not follow from such a simple observation. We have to show that if we take the above formula, then apply the equivalence $\Lambda^{\ful{f}k} X^k \simar \Lambda^{f^k} X$ \eqref{eq:forget_equiv}, the map we get is homotopic to the formula for $R(f^n)$.

Applying the equivalence in \eqref{eq:forget_equiv} to the path in $\Lambda^{\ful{f}k} X^k$, gives the path
\[ f^{n-1}(\gamma_{f(p(v_1)),v_2}) \cdot \ldots \cdot f(\gamma_{f(p(v_{n-1})),v_n}) \cdot \gamma_{f(p(v_n)),v_1}. \]
We now change this path by a homotopy. As observed above, replacing $\epsilon$ by $\delta<\epsilon$ does not  change the Reidemeister trace  in the homotopy category. Since $f$ is a continuous function on a compact space it is uniformly continuous, and therefore there is a $\delta>0$ so that when every $v_i$ is within $\delta$ of $X$, 
the diameter of the paths $\gamma_{f(p(v_i)),v_{i+1}}$ and their images under $f$, $f^2$, $\ldots$, and $f^{n-1}$ are less than $\frac{\epsilon}{2n}$. (We measure all of these diameters as subsets of $\R^n$.) 
Then the sum of $n$ of these diameters is less than  $\frac{\epsilon}{2}$. This is small enough to guarantee that if we compose $n$ such paths together, the straight-line homotopy in $\R^n$ between their composite and $\gamma_{f(p(v_n)),v_1}$ lies entirely in $V$, and can therefore be projected to $X$. This gives a continuous homotopy of paths in $X$ rel endpoints
\[ f^{n-1}(\gamma_{f(p(v_1)),v_2}) \cdot f^{n-2}(\gamma_{f(p(v_2)),v_3}) \cdot \ldots \cdot f(\gamma_{f(p(v_{n-1})),v_n}) \cdot \gamma_{f(p(v_n)),v_1}
\quad \sim \quad
\gamma_{f^n(p(v_1)),v_1}. \]
Therefore our original formula is homotopic to:
\begin{eqnarray*}
	S^{nN} & \longrightarrow & S^{N}_\epsilon \sma \ldots \sma S^{N}_\epsilon \sma (\Lambda^{\ful{f}n} X^n)_+ \\
	(v_1,\ldots,v_n) & \longmapsto & v_1 - f(p(v_n)) \sma v_2 - f(p(v_1)) \sma v_3 - f(p(v_2)) \sma \ldots \sma \gamma_{f^n(p(v_1)),v_1}.
\end{eqnarray*}

The path now matches the path we would get for $R(f^n)$, but the sphere coordinates are different, so we apply a homotopy to those next. In the $(i+1)$st coordinate of the output, we apply a homotopy of the form
\[ v_{i+1} - f(p(v_i)) \quad \sim \quad v_{i+1} - f^2(p(v_{i-1})) \quad \sim \quad \ldots \quad \sim \quad v_{i+1} - f^i(p(v_1)) \]
by dragging the second term along the path $f(\gamma_{f(p(v_{i-1})),v_i})$, then the path $f(\gamma_{f^2(p(v_{i-2})),v_{i-1}})$, and so on. Note that before we start this homotopy, our map is supported on the region where the distance from each $v_i$ to $X$ is less than or equal to $\delta$, and throughout the homotopy, the boundary of this region is sent to the basepoint. In other words, if $d(v_{i+1},X) \geq \delta$ then throughout the homotopy the size of the sphere coordinate is always $\geq \delta$, because every path we use is a path contained in $X$. This guarantees that we get a well-defined homotopy of maps on all of $S^{nN}$. Performing this for each $1 \leq i \leq n$ gives a homotopic map with the formula
\begin{eqnarray*}
	S^{nN} & \longrightarrow & S^{N}_\delta \sma \ldots \sma S^{N}_\delta \sma (\Lambda^{\ful{f}n} X^n)_+ \\
	(v_1,\ldots,v_n) & \longmapsto & v_1 - f^n(p(v_1)) \sma v_2 - f(p(v_1)) \sma v_3 - f^2(p(v_1)) \sma \ldots \sma \gamma_{f^n(p(v_1)),v_1}.
\end{eqnarray*}
This is almost the formula for $R(f^n)$ using the embedding $(i,f \circ i,f^2 \circ i,\ldots,f^{n-1} \circ i)\colon X \to \R^{nN}$, except that $f^n$ is only being applied to the first coordinate. So for the final step, we examine the above formula and observe that it still makes sense if we relax the assumptions and allow $v_i$ to be any point in $\R^n$ when $i \geq 2$. With this change we can then remove the $f^k(p(v_1))$ term from the second through $n$th coordinates by a homotopy, arriving at
\begin{eqnarray*}
	S^{nN} & \longrightarrow & S^{N}_\delta \sma \ldots \sma S^{N}_\delta \sma (\Lambda^{\ful{f}n} X^n)_+ \\
	(v_1,\ldots,v_n) & \longmapsto & v_1 - f^n(p(v_1)) \sma v_2 \sma v_3 \sma \ldots \sma \gamma_{f^n(p(v_1)),v_1}
\end{eqnarray*}
This agrees with the formula for $R(f^n)$ using the embedding $(i,0,0,\ldots,0)\colon X \to \R^{nN}$. Equivalently, it is the $(n-1)N$-fold suspension of the formula for $R(f^n)$ using $i\colon X \to \R^N$. This concludes the proof that the third triangle commutes in the homotopy category.


\part{Unwinding the Fuller trace}\label{part:unwinding}

In this part we give a very general proof that the last triangle of \cref{thm:simple_main_thm} commutes: 
\begin{thmx}
\label{thm:main_first_half}
	For any finitely dominated space $X$, the following diagram commutes up to homotopy.
	\[\xymatrix@C=40pt@R=40pt{
 		\Sph
		\ar[d]_-{R(\ful{f}k)} 
		\ar[rd]^{R(f^k)} 
		\\
		\Sigma^\infty_+ \Lambda^{\ful{f}k} X^k \ar[r]^-{\simeq} 
		&\Sigma^\infty_+ \Lambda^{f^k} X
	}\]
\end{thmx}

The argument is formal and based on the observation from \cite{p:thesis} that $R(f)$ is a bicategorical trace. To motivate this argument, we first describe the analogous argument for symmetric monoidal categories in \cref{sec:smc}. We then recall how to define fixed point invariants using bicategories in \cref{sec:bicategories,sec:traces_in_bicat}, and finally prove the bicategorical version of the argument in \cref{sec:fuller_bicat,sec:base_change}. 
In this part, we black-box all of the needed properties of parametrized spectra. 

\subsubsection*{Remark on string diagrams} 
 Even in simple cases, conventional notation choices obfuscate some of the central ideas in this paper.  In an attempt to make these ideas more visible,  we use the string diagrams calculus of \cite{jsv}.   As shown in \cite{jsv}, string diagram calculations are a rigorous alternative to traditional diagram chasing in symmetric monoidal categories and string diagrams manipulations can be translated into more conventional diagrams.  (The corresponding result for bicategories with shadows can be found in \cite{ps:bicat}.)
Together 
\cite{jsv,ps:bicat} put all string diagram manipulations in this paper on rigorous footing.  
The one exception is \cref{fig:multitrace_cartoon} which should be regarded as motivation.

The building blocks for the symmetric monoidal  string diagram  calculus are the first four figures in  \cref{fig:string-smc}. They are ``Poincar\'e dual'' to the usual graphical representation of symmetric monoidal categories.

Finally, there are no string diagram calculations after \cref{lem:nthpower} since this would require the development of a new calculus and that is beyond the scope of this paper.

\section{Traces and multitraces in symmetric monoidal categories}\label{sec:smc}

In this section we  consider the special case of \cref{thm:main_first_half} in a symmetric monoidal category $\sC$ with monoidal product $\otimes$ and unit object $U$. 
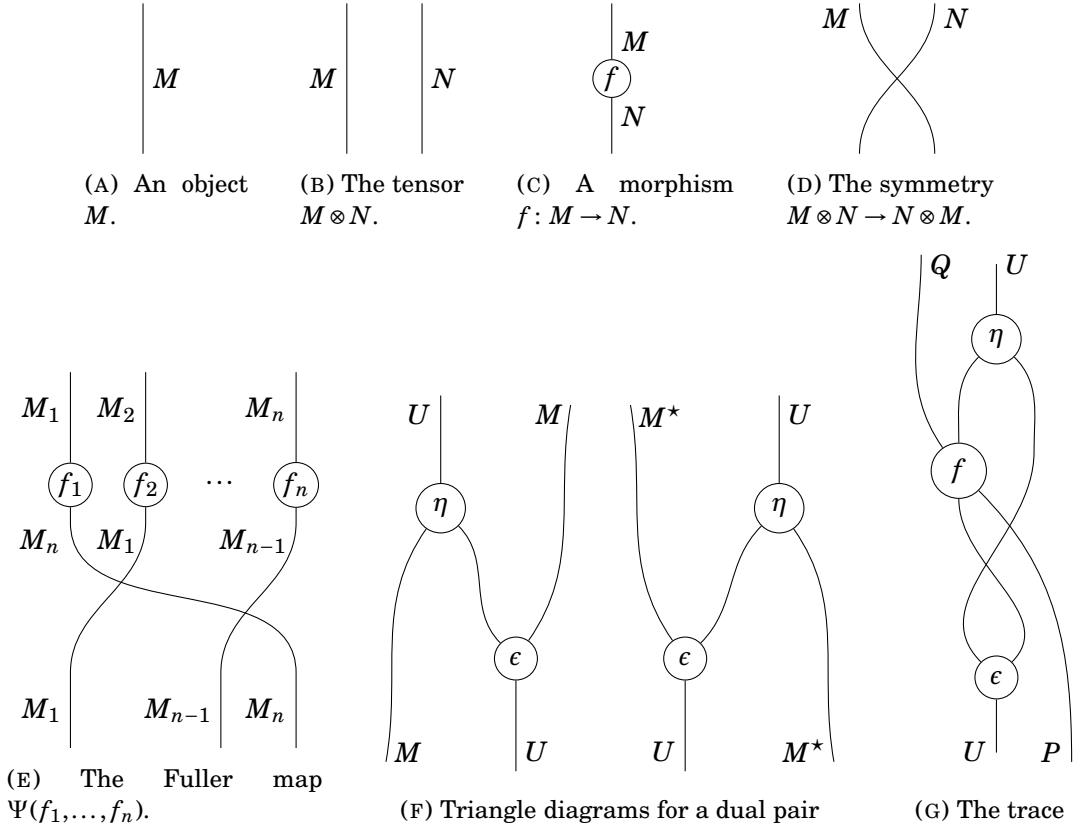
\begin{figure}[t!]
\centering
\begin{subfigure}[t]{.15\linewidth}
\centering
   \begin{tikzpicture}
       \draw(1,2)--(1, 0);
       \node[anchor= west] at (1,1){$M$};
     \end{tikzpicture}
\caption{An object $M$.}
\end{subfigure}
  \hspace{.5cm}
\begin{subfigure}[t]{.15\linewidth}
\centering
   \begin{tikzpicture}
      \draw(1,2)--(1,0);
      \draw(2,2)--(2,0);
      \node[anchor= east]at (1,1){$M$};
      \node[anchor= west]at (2,1){$N$};
     \end{tikzpicture}
\caption{The tensor $M\otimes N$.}
\end{subfigure}
\hspace{.5cm}
\begin{subfigure}[t]{.2\linewidth}
\centering
   \begin{tikzpicture}
\node[fill=white,draw,circle,inner sep=1pt] (f) at (1,1) {$f$};
       \draw (1,2) -- (f);
       \draw (f) -- (1,0);
       \node[anchor=west] at(1, 1.5){$M$};
       \node[anchor=west] at(1, .5){$N$};
     \end{tikzpicture}
\caption{A morphism $f\colon M\to N$.}
\end{subfigure}
  \hspace{.5cm}
\begin{subfigure}[t]{.20\linewidth}
\centering
   \begin{tikzpicture}
      \draw(1,2) to [out =-90, in =90](2,0);
      \draw(2,2) to [out =-90, in =90](1,0);
      \node[anchor= east]at (1,1.8){$M$};
      \node[anchor= west]at (2,1.8){$N$};
     \end{tikzpicture}
\caption{The symmetry \\ $M\otimes N\to N\otimes M$.}
\end{subfigure}

\begin{subfigure}[b]{.29\linewidth}
\centering
   \begin{tikzpicture}
       \node[fill=white,draw,circle,inner sep=1pt] (f) at (1,.5) {$f_1$};
\node[fill=white,draw,circle,inner sep=1pt] (f1) at (2,.5) {$f_{2}$};
  \node[fill=white,draw,circle,inner sep=1pt] (f3) at (4,.5) {$f_n$};
       \draw[-] (1,2) -- (f);
       \draw[-](f) -- (1,0) to [out =-90, in =90] (4,-2)--(4,-3);
  \draw[-] (2,2) -- (f1);
  \draw(f1) -- (2,0) to [out =-90, in =90] (1,-2)--(1,-3);
  \draw[-] (4,2) -- (f3);
  \draw[-](f3) -- (4,0) to [out =-90, in =90] (3,-2)--(3,-3);
  \node[anchor=east] at(1, 1.5){$M_1$};
  \node[anchor=east] at(2, 1.5){$M_2$};
  \node[anchor=east] at(4, 1.5){$M_n$};
  \node[anchor=east] at(1, -.25){$M_{n}$};
  \node[anchor=east] at(2, -.25){$M_1$};
  \node[anchor=east] at(4, -.25){$M_{n-1}$};
  \node[anchor=east] at(4, -2.5){$M_{n}$};
  \node[anchor=east] at(1, -2.5){$M_1$};
  \node[anchor=east] at(3, -2.5){$M_{n-1}$};
  \node at(3, .5){$\cdots$};
     \end{tikzpicture}
\caption{The Fuller map $\ful {f_1,\ldots ,f_n}{}$.}
   \label{fig:string-smc-fuller}
\end{subfigure}
\hspace{.03\linewidth}
\begin{subfigure}[b]{.44\linewidth}
\centering
   \begin{tikzpicture}
 \node at (-.5,-4)(s){};
 \node at (-.5,-4)[above right]{$M$};
 \node at (1.25,-4)[above right]{$U$};
 \node at (.25,-.5)(n1)[circle,fill = white, draw]{$\eta$};
 \node at (1.25,-2.5)(e1)[circle,fill = white, draw]{$\epsilon$};
 \node at (2,1)(f){};
 \node at (2,1)[below left]{$M$};
 \node at (.25,1)[below left]{$U$};

  \draw (s) to [out=80, in =-125] (n1) ;
  \draw (n1)--(.25, 1);
  \draw (n1) to[out=-45, in =135]  (e1);
  \draw (e1) -- (1.25, -4);
  \draw (e1) to [out = 55, in =-100] (f) ;
  
 \node at (5.5,-4)(s2){};
 \node at (5.5,-4)[above left]{$\rdual{M}$};
 \node at (3.5,-4)[above left]{$U$};
 \node at (4.75,-.5)(n2)[circle,fill = white, draw]{$\eta$};
 \node at (3.5,-2.5)(e2)[circle,fill = white, draw]{$\epsilon$};
 \node at (2.75,1)(f2){};
 \node at (2.75,1)[below right]{$\rdual{M}$};
 \node at (4.75,1)[below right]{$U$};

  \draw (s2) to [out=100, in =-55] (n2) ;
  \draw (n2)--(4.75, 1);
  \draw (n2) to[out=-135, in =45]  (e2);
  \draw (e2) -- (3.5, -4);
  \draw (e2) to [out = 125, in =-80] (f2) ;
  \end{tikzpicture}
   \caption{Triangle diagrams for a dual pair }
   \label{fig:triangles-smc}
\end{subfigure}
\hspace{.03\linewidth}
\begin{subfigure}[b]{.17\linewidth}
\centering
\begin{tikzpicture}

 \node at (2,-5.75)(p){};
 \node at (2,-5.75)[above left]{$P$};
 \node at (1,-5.75)[above left]{$U$};
 \node at (.5,-1.75)(f)[circle,fill = white, draw]{$f$};
 \node at (1,0)(n1)[circle,fill = white, draw]{$\eta$};
 \node at (1,-4.5)(e1)[circle,fill = white, draw]{$\epsilon$};
 \node at (0,1.25)(q){};
 \node at (0,1.25)[below right]{$Q$};
 \node at (1,1.25)[below right]{$U$};

  \draw (q) to [out=-90, in =125] (f) ;
  \draw (n1)--(1, 1);
  \draw (n1) to [out =-135, in =90] (f) ;
  \draw (f) to[out=-90, in =45]  (e1);
  \draw (e1) -- (1, -5.5);
  \draw (f) to [out = -45, in =90] (p) ;
  \draw (n1) to[out =-45, in =90] (1.5, -1.5);
  \draw (1.5, -1.5) to [out =-90,  in =135] (e1);

\end{tikzpicture}
\caption{The  trace}
\label{fig:smc_trace}
\end{subfigure}
\caption{String diagrams for a symmetric monoidal category. \\We view the strings as oriented from top to bottom.
}
   \label{fig:string-smc}

\end{figure}
Recall that:
\begin{itemize}
\item An object $M$ of $\sC$ is {\bf dualizable} if there is an object $\rdual{M}$ of $\sC$ and morphisms
	\[\eta\colon U\to M\otimes \rdual{M}\text{  and  }\epsilon \colon \rdual{M}\otimes M\to U\]
so that the composites in \cref{fig:triangles-smc} are identity maps.  These are the {\bf triangle identities}.
\item If $M$ is dualizable, the {\bf trace} of a morphism $f\colon P\otimes M\to M\otimes Q$ is the composite in \cref{fig:smc_trace}.  If $P$ and $Q$ are units,  $f\colon M\to M$.
\end{itemize}
We then define the {\bf Fuller construction} of an $n$-tuple of maps $f_i\colon M_i\to M_{i-1}$ to be the composite 
\begin{equation}
\label{eq:ful_construction}
\ful {f_1,\ldots ,f_n}{}\colon M_1\otimes \cdots \otimes M_n\xto{f_1\otimes \cdots \otimes f_n}M_n\otimes M_1\otimes \cdots \otimes M_{n-1}\xto{\gamma}M_1\otimes \cdots \otimes M_n\end{equation}
This is illustrated by the string diagram in \cref{fig:string-smc-fuller}. When all the $M_i$ and $f_i$ are equal, this is the $n$th Fuller map $\ful fn$ described in the introduction.

\begin{thm}[Symmetric monoidal version of \cref{thm:main_first_half}]\label{thm:smc_compare_traces}  
	For an $n$-tuple of dualizable objects $\{ M_i \}_{i=1}^n$ in a symmetric monoidal category $(\sC, \otimes, U)$ and maps $f_i\colon M_i\to M_{i-1}$
	\[\tr(\ful {f_1,\ldots ,f_n}{})= \tr(f_1\circ f_2\circ \cdots \circ f_n)\]
	as maps $U \to U$ in $\sC$.
\end{thm}
\begin{proof}
The trace of the Fuller construction 
	\[ \ful{f_1,\ldots, f_n}{}\colon \bigotimes M_i \to \bigotimes M_i  \] of maps $f_i\colon M_i \to M_{i-1}$ 
  is depicted in figure \cref{fig:n_to_fuller_1}. Symmetry isomorphisms transform this trace to  \cref{fig:n_to_fuller_4}. Canceling as in \cref{fig:triangles-smc} transforms this to  \cref{fig:n_to_fuller_6}.  This completes the proof of this theorem using string diagrams.

Alternatively, a diagram chase shows that
 if $X\otimes Y$ and $Z$ are dualizable and 
$g\colon X\otimes Y\to X\otimes Z$ and $f\colon Z\to Y$,
the trace of 
\[X\otimes Y\otimes Z\xto{g\otimes f} X\otimes Z\otimes Y\xto{\id\otimes \gamma}X\otimes Y\otimes Z\]
is the trace of 
\[X\otimes Y\xto{g} X\otimes Z\xto{\id\otimes f} X\otimes Y.\]
Since the map in \eqref{eq:ful_construction} is the composite
\begin{align*}
M_1\otimes \cdots \otimes M_{n-2}\otimes M_{n-1}\otimes M_n
&\xto{f_1\otimes \cdots \otimes f_{n-1}\otimes f_n}M_n\otimes M_1\otimes \cdots M_{n-2}\otimes M_{n-1}
\\
&\xto{\gamma_{1,n-2}\otimes \id}M_1\otimes \cdots M_{n-2}\otimes M_n\otimes M_{n-1}
\\
&\xto{\id^{\otimes n-2}\otimes \gamma}M_1\otimes \cdots M_{n-2}\otimes M_{n-1}\otimes M_n
\end{align*}
the trace of $\ful{f_1,\ldots, f_n}{}$ is the trace of $\ful{(f_n\circ f_1),\ldots, f_{n-1}}{}$.
Then the result follows by induction and the cyclic invariance of the trace.
\end{proof}

\tdplotsetmaincoords{60}{120}
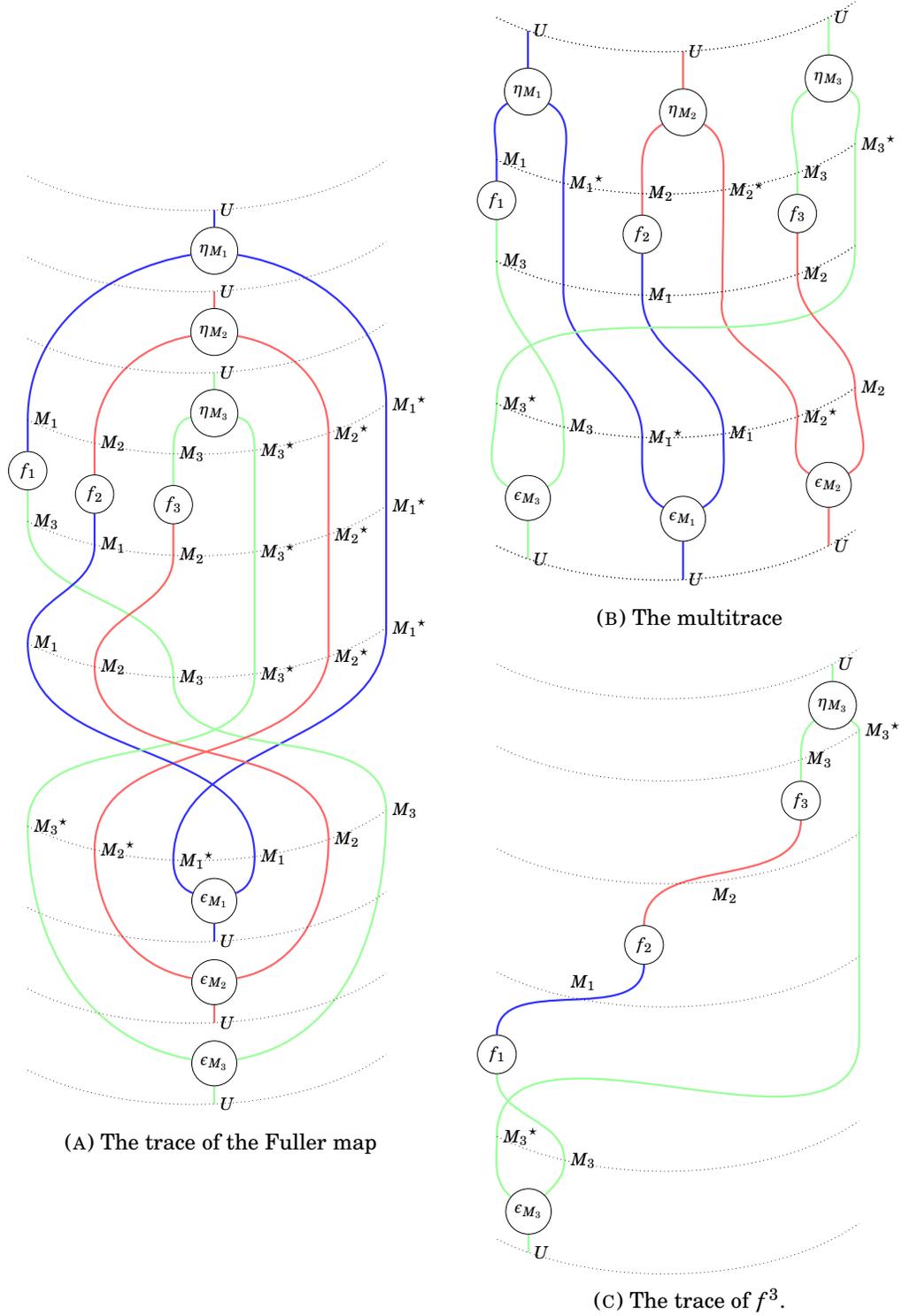
\begin{figure}[!htb]
    \centering
    \begin{tabular}[t]{cc}

\begin{subfigure}{.45\linewidth}
\resizebox{\textwidth}{!}
{
\centering
\begin{tikzpicture}[tdplot_main_coords, scale =.5]
   \pgfmathsetmacro{\RA}{10}
   \pgfmathsetmacro{\RB}{10}
   \pgfmathsetmacro{\RC}{10}

   \pgfmathsetmacro{\LA}{21}
   \pgfmathsetmacro{\LAA}{17}
   \pgfmathsetmacro{\LAB}{13}
   \pgfmathsetmacro{\LAC}{19}
   \pgfmathsetmacro{\LAD}{15}
   \pgfmathsetmacro{\LAE}{11}
   \pgfmathsetmacro{\LB}{9}
   \pgfmathsetmacro{\LBB}{6.5}
   \pgfmathsetmacro{\LC}{4}
   \pgfmathsetmacro{\LD}{-2}
   \pgfmathsetmacro{\LG}{-11}
   \pgfmathsetmacro{\LH}{-18}
   \pgfmathsetmacro{\LHH}{-22}
   \pgfmathsetmacro{\LI}{-13}

   \pgfmathsetmacro{\LIA}{-17}
   \pgfmathsetmacro{\LIB}{-21}
   \pgfmathsetmacro{\LIC}{-15}
   \pgfmathsetmacro{\LID}{-19}
   \pgfmathsetmacro{\LIE}{-23}

\pgfmathsetmacro{\BA}{-15}
\pgfmathsetmacro{\BAA}{\BA+50}
\pgfmathsetmacro{\BB}{\BA+20}
\pgfmathsetmacro{\BBA}{\BA+50}
\pgfmathsetmacro{\BC}{\BA+40}
\pgfmathsetmacro{\BCA}{\BA+50}
\pgfmathsetmacro{\BD}{\BA+60}
\pgfmathsetmacro{\BDA}{\BA+50}
\pgfmathsetmacro{\BE}{\BA+80}
\pgfmathsetmacro{\BEA}{\BA+50}
\pgfmathsetmacro{\BF}{\BA+100}
\pgfmathsetmacro{\BFA}{\BA+50}

   \draw [dotted,domain=\BA:\BF] plot ({\RA*cos(\x)}, {\RA*sin(\x)},\LA);
   \draw [dotted,domain=\BA:\BF] plot ({\RB*cos(\x)}, {\RB*sin(\x)},\LAA);
   \draw [dotted,domain=\BA:\BF] plot ({\RC*cos(\x)}, {\RC*sin(\x)},\LAB);

   \draw [dotted,domain=\BA:\BF] plot ({\RA*cos(\x)}, {\RA*sin(\x)},\LB);

   \draw [dotted,domain=\BA:\BF] plot ({\RA*cos(\x)}, {\RA*sin(\x)},\LC);

   \draw [dotted,domain=\BA:\BF] plot ({\RA*cos(\x)}, {\RA*sin(\x)},\LD);

   \draw [dotted,domain=\BA:\BF] plot ({\RA*cos(\x)}, {\RA*sin(\x)},\LG);

   \draw [dotted,domain=\BA:\BF] plot ({\RA*cos(\x)}, {\RA*sin(\x)},\LIC);
   \draw [dotted,domain=\BA:\BF] plot ({\RB*cos(\x)}, {\RB*sin(\x)},\LID);
   \draw [dotted,domain=\BA:\BF] plot ({\RC*cos(\x)}, {\RC*sin(\x)},\LIE);

\coordinate (ua1) at  ({\RA*cos(\BAA)}, {\RA*sin(\BAA)},\LA);
\node at (ua1)[right] {$U$};
  \coordinate (ua2) at ({\RB*cos(\BBA)}, {\RB*sin(\BBA)},\LAA) ;
\node at (ua2) [right] {$U$};
\coordinate (ua3) at  ({\RC*cos(\BCA)}, {\RC*sin(\BCA)},\LAB);
\node at (ua3) [right] {$U$};

  \node at ({\RA*cos(\BAA)}, {\RA*sin(\BAA)},\LAC) (nm1) [circle, draw, fill = white] {$\eta_{M_1}$};
  \node at ({\RB*cos(\BBA)}, {\RB*sin(\BBA)},\LAD) (nm2) [circle, draw, fill = white] {$\eta_{M_2}$};
  \node at ( {\RC*cos(\BCA)}, {\RC*sin(\BCA)},\LAE) (nm3) [circle, draw, fill = white] {$\eta_{M_3}$};

\coordinate (m21) at  ({\RA*cos(\BA)}, {\RA*sin(\BA)}, \LB);
\node at  (m21) [right] {$M_1$};
\coordinate (m21d) at  ({\RA*cos(\BF)}, {\RA*sin(\BF)}, \LB);
\node at  (m21d) [right] {$\rdual{M_1}$};
\coordinate (m22) at  ({\RB*cos(\BB)}, {\RB*sin(\BB)}, \LB);
\node at  (m22) [right] {$M_2$};
\coordinate (m22d) at  ({\RB*cos(\BE)}, {\RB*sin(\BE)}, \LB);
\node at  (m22d) [right] {$\rdual{M_2}$};
  \coordinate (m23) at  ({\RC*cos(\BC)}, {\RC*sin(\BC)}, \LB);
\node at (m23) [right] {$M_3$};
  \coordinate (m23d) at  ({\RC*cos(\BD)}, {\RC*sin(\BD)}, \LB) ;
\node at (m23d) [right] {$\rdual{M_3}$};

  \node at ({\RA*cos(\BA)}, {\RA*sin(\BA)}, \LBB) (f21) [circle, draw, fill = white] {$f_1$};
  \node at ({\RB*cos(\BB)}, {\RB*sin(\BB)}, \LBB) (f22) [circle, draw, fill = white] {$f_2$};
  \node at  ({\RC*cos(\BC)}, {\RC*sin(\BC)}, \LBB) (f23) [circle, draw, fill = white] {$f_3$};

  \coordinate (m31)  at ({\RA*cos(\BA)}, {\RA*sin(\BA)}, \LC);
\node at  (m31) [right] {$M_3$};
  \coordinate (m31d)  at ({\RA*cos(\BF)}, {\RA*sin(\BF)}, \LC);
\node at  (m31d) [right] {$\rdual{M_1}$};
  \coordinate (m32)  at ({\RB*cos(\BB)}, {\RB*sin(\BB)}, \LC);
\node at  (m32) [right] {$M_1$};
  \coordinate (m32d)  at ({\RB*cos(\BE)}, {\RB*sin(\BE)}, \LC);
\node at  (m32d) [right] {$\rdual{M_2}$};
  \coordinate (m33)  at ({\RC*cos(\BC)}, {\RC*sin(\BC)}, \LC);
\node at (m33) [right] {$M_2$};
  \coordinate (m33d) at ({\RC*cos(\BD)}, {\RC*sin(\BD)}, \LC);
\node at (m33d) [right] {$\rdual{M_3}$};

  \coordinate (m41) at ({\RC*cos(\BC)}, {\RC*sin(\BC)}, \LD);
\node at (m41) [right] {$M_3$};
  \coordinate (m41d) at ({\RA*cos(\BF)}, {\RA*sin(\BF)}, \LD);
\node at (m41d) [right] {$\rdual{M_1}$};
  \coordinate (m42)  at ({\RA*cos(\BA)}, {\RA*sin(\BA)}, \LD);
\node at (m42) [right] {$M_1$};
  \coordinate (m42d)  at ({\RB*cos(\BE)}, {\RB*sin(\BE)}, \LD);
\node at (m42d) [right] {$\rdual{M_2}$};
  \coordinate (m43)  at ({\RB*cos(\BB)}, {\RB*sin(\BB)}, \LD);
\node at  (m43) [right] {$M_2$};
  \coordinate (m43d) at ({\RC*cos(\BD)}, {\RC*sin(\BD)}, \LD);
\node at (m43d) [right] {$\rdual{M_3}$};

  \coordinate (m71) at ({\RC*cos(\BF)}, {\RC*sin(\BF)}, \LG);
\node at  (m71) [right] {$M_3$};
  \coordinate (m71d) at ({\RA*cos(\BC)}, {\RA*sin(\BC)} ,\LG);
\node at  (m71d) [right] {$\rdual{M_1}$};
  \coordinate (m72) at ({\RA*cos(\BD)}, {\RA*sin(\BD)}, \LG);
\node at (m72) [right] {$M_1$};
  \coordinate (m72d)  at ({\RB*cos(\BB)}, {\RB*sin(\BB)}, \LG);
\node at (m72d) [right] {$\rdual{M_2}$};
  \coordinate (m73)  at  ({\RB*cos(\BE)}, {\RB*sin(\BE)}, \LG);
\node at (m73) [right] {$M_2$};
  \coordinate (m73d)  at  ({\RC*cos(\BA)}, {\RC*sin(\BA)}, \LG);
\node at (m73d) [right] {$\rdual{M_3}$};

  \node at ({\RA*cos(\BFA)}, {\RA*sin(\BFA)}, \LI) (e91) [circle, draw, fill = white] {$\epsilon_{M_1}$};
  \node at ({\RB*cos(\BEA)}, {\RB*sin(\BEA)}, \LIA) (e92) [circle, draw, fill = white] {$\epsilon_{M_2}$};
  \node at ({\RC*cos(\BDA)}, {\RC*sin(\BDA)}, \LIB) (e93) [circle, draw, fill = white] {$\epsilon_{M_3}$};

  \coordinate (u91) at ({\RA*cos(\BFA)}, {\RA*sin(\BFA)}, \LIC);
\node at (u91) [right] {$U$};
  \coordinate (u92)  at ({\RB*cos(\BEA)}, {\RB*sin(\BEA)}, \LID);
\node at  (u92) [right] {$U$};
  \coordinate (u93) at ({\RC*cos(\BDA)}, {\RC*sin(\BDA)}, \LIE);
\node at  (u93) [right] {$U$};

 \begin{scope}[on background layer]

  \draw [mydark, very thick] (e91) to [out = 160, in = -90]
(m71d) to [out =90, in =-90] 
(m41d) -- (m31d)  to [out=90,in =-90]  (m21d) to [out = 90, in = -10] (nm1); 
  \draw [mydark, very thick] (nm1) to [out = 190, in =90] (m21) --(f21); 
  \draw [mylight, very thick] (f21) to [out=-90,in =90]  (m31)  to [out=-90,in =90]  (m41) to[in=90, out = -90]   
(m71) 
to [out=-90, in =10] (e93);
  \draw [mymed, very thick ]  (e92) to [out = 170, in = -90]  
 (m72d) to [out =90, in =-100] 
(m42d) -- (m32d)  to [out =90, in =-90]  (m22d) to [out = 90, in =-10] (nm2);  
\draw [mymed, very thick] (nm2) to [out =-170, in =90]  (m22) -- (f22);
\draw [mydark, very thick] (f22) to [out =-90, in =90]  (m32)  to [out =-90, in =90]  (m42)  to [out = -90,in =90]  
(m72)  to [out =-90, in =20]  
(e91);
  \draw [mylight, very thick ](e93) to [out = 170, in = -90]   
(m73d)to [out =90, in =-90]  
(m43d) -- (m33d) -- (m23d) to [out = 90, in = -10]  (nm3);
\draw [mylight, very thick](nm3) to [out =-170, in =90] (m23) -- (f23);
\draw [mymed, very thick] (f23) to [out =-90, in =90]  (m33) to [out =-90, in =90]  (m43) to [out =-90, in =90]
 (m73)
 to [out =-90, in =10]  (e92);

  \draw[mydark, very thick] (ua1) -- (nm1);
  \draw[mymed, very thick] (ua2) -- (nm2);
  \draw[mylight, very thick ] (ua3) -- (nm3);
  \draw[mydark, very thick] (u91) -- (e91);
  \draw[mymed, very thick] (u92) -- (e92);
  \draw[mylight, very thick] (u93) -- (e93);
    \end{scope}
     \end{tikzpicture}
     }
  \caption{The trace of the Fuller map} \label{fig:n_to_fuller_1}
\end{subfigure}
    &
        \begin{tabular}{c}
        \smallskip
\begin{subfigure}[b]{.45\linewidth}
\resizebox{\textwidth}{!}{
\centering
\begin{tikzpicture}[tdplot_main_coords, scale =.5]

   \pgfmathsetmacro{\RA}{10}
   \pgfmathsetmacro{\RB}{10}
   \pgfmathsetmacro{\RC}{10}
   
   \pgfmathsetmacro{\LA}{21}
   \pgfmathsetmacro{\LAA}{21}
   \pgfmathsetmacro{\LAB}{21}
   \pgfmathsetmacro{\LAC}{18}
   \pgfmathsetmacro{\LAD}{18}
   \pgfmathsetmacro{\LAE}{18}
   \pgfmathsetmacro{\LB}{14}
   \pgfmathsetmacro{\LBB}{12}
   \pgfmathsetmacro{\LC}{9}
   \pgfmathsetmacro{\LD}{6}
   \pgfmathsetmacro{\LG}{2}
   \pgfmathsetmacro{\LH}{0}
   \pgfmathsetmacro{\LHH}{-22}
   \pgfmathsetmacro{\LI}{-2}

   \pgfmathsetmacro{\LIA}{-2}
   \pgfmathsetmacro{\LIB}{-2}
   \pgfmathsetmacro{\LIC}{-5}
   \pgfmathsetmacro{\LID}{-5}
   \pgfmathsetmacro{\LIE}{-5}

\pgfmathsetmacro{\BA}{-15}
\pgfmathsetmacro{\BAA}{\BA+10}
\pgfmathsetmacro{\BB}{\BA+20}
\pgfmathsetmacro{\BBA}{\BA+30}
\pgfmathsetmacro{\BC}{\BA+40}
\pgfmathsetmacro{\BCA}{\BA+50}
\pgfmathsetmacro{\BD}{\BA+60}
\pgfmathsetmacro{\BDA}{\BA+70}
\pgfmathsetmacro{\BE}{\BA+80}
\pgfmathsetmacro{\BEA}{\BA+90}
\pgfmathsetmacro{\BF}{\BA+100}
\pgfmathsetmacro{\BFA}{\BA+30}

   \draw [dotted,domain=\BA:\BF] plot ({\RA*cos(\x)}, {\RA*sin(\x)},\LA);
   \draw [dotted,domain=\BA:\BF] plot ({\RB*cos(\x)}, {\RB*sin(\x)},\LAA);
   \draw [dotted,domain=\BA:\BF] plot ({\RC*cos(\x)}, {\RC*sin(\x)},\LAB);

   \draw [dotted,domain=\BA:\BF] plot ({\RA*cos(\x)}, {\RA*sin(\x)},\LB);
   \draw [dotted,domain=\BA:\BF] plot ({\RB*cos(\x)}, {\RB*sin(\x)},\LB);
   \draw [dotted,domain=\BA:\BF] plot ({\RC*cos(\x)}, {\RC*sin(\x)},\LB);

   \draw [dotted,domain=\BA:\BF] plot ({\RA*cos(\x)}, {\RA*sin(\x)},\LC);
   \draw [dotted,domain=\BA:\BF] plot ({\RB*cos(\x)}, {\RB*sin(\x)},\LC);
   \draw [dotted,domain=\BA:\BF] plot ({\RC*cos(\x)}, {\RC*sin(\x)},\LC);

   \draw [dotted,domain=\BA:\BF] plot ({\RA*cos(\x)}, {\RA*sin(\x)},\LG);
   \draw [dotted,domain=\BA:\BF] plot ({\RB*cos(\x)}, {\RB*sin(\x)},\LG);
   \draw [dotted,domain=\BA:\BF] plot ({\RC*cos(\x)}, {\RC*sin(\x)},\LG);
   
   \draw [dotted,domain=\BA:\BF] plot ({\RA*cos(\x)}, {\RA*sin(\x)},\LIC);
   \draw [dotted,domain=\BA:\BF] plot ({\RB*cos(\x)}, {\RB*sin(\x)},\LID);
   \draw [dotted,domain=\BA:\BF] plot ({\RC*cos(\x)}, {\RC*sin(\x)},\LIE);

  \coordinate (ua1)  at ({\RA*cos(\BAA)}, {\RA*sin(\BAA)},\LA);
\node at  (ua1) [right] {$U$};
  \coordinate (ua2) at ({\RB*cos(\BCA)}, {\RB*sin(\BCA)},\LAA);
\node at  (ua2) [right] {$U$};
  \coordinate (ua3) at ({\RC*cos(\BEA)}, {\RC*sin(\BEA)},\LAB);
\node at  (ua3) [right] {$U$};

  \node at ({\RA*cos(\BAA)}, {\RA*sin(\BAA)},\LAC) (nm1) [circle, draw, fill = white] {$\eta_{M_1}$};
  \node at ({\RB*cos(\BCA)}, {\RB*sin(\BCA)},\LAD) (nm2) [circle, draw, fill = white] {$\eta_{M_2}$};
  \node at ( {\RC*cos(\BEA)}, {\RC*sin(\BEA)},\LAE) (nm3) [circle, draw, fill = white] {$\eta_{M_3}$};

  \coordinate (m21) at ({\RA*cos(\BA)}, {\RA*sin(\BA)}, \LB);
\node at  (m21) [right] {$M_1$};
  \coordinate (m21d) at ({\RA*cos(\BB)}, {\RA*sin(\BB)}, \LB);
\node at  (m21d) [right] {$\rdual{M_1}$};
  \coordinate (m22) at ({\RB*cos(\BC)}, {\RB*sin(\BC)}, \LB);
\node at  (m22) [right] {$M_2$};
  \coordinate (m22d) at  ({\RB*cos(\BD)}, {\RB*sin(\BD)}, \LB);
\node at  (m22d) [right] {$\rdual{M_2}$};
  \coordinate (m23) at  ({\RC*cos(\BE)}, {\RC*sin(\BE)}, \LB);
\node at  (m23) [right] {$M_3$};
  \coordinate (m23d) at  ({\RC*cos(\BF)}, {\RC*sin(\BF)}, \LB);
\node at  (m23d) [right] {$\rdual{M_3}$};

  \node at ({\RA*cos(\BA)}, {\RA*sin(\BA)}, \LBB) (f21) [circle, draw, fill = white] {$f_1$};
  \node at ({\RB*cos(\BC)}, {\RB*sin(\BC)}, \LBB) (f22) [circle, draw, fill = white] {$f_2$};
  \node at  ({\RC*cos(\BE)}, {\RC*sin(\BE)}, \LBB) (f23) [circle, draw, fill = white] {$f_3$};

  \coordinate (m31) at ({\RA*cos(\BA)}, {\RA*sin(\BA)}, \LC);
\node at  (m31) [right] {$M_3$};
  \node at ({\RA*cos(\BB)}, {\RA*sin(\BB)}, \LC) (m31d){};
  \coordinate (m32) at ({\RB*cos(\BC)}, {\RB*sin(\BC)}, \LC);
\node at  (m32) [right] {$M_1$};
  \node at ({\RB*cos(\BD)}, {\RB*sin(\BD)}, \LC) (m32d){};
  \coordinate (m33) at ({\RC*cos(\BE)}, {\RC*sin(\BE)}, \LC);
\node at (m33) [right] {$M_2$};
  \node at ({\RC*cos(\BF)}, {\RC*sin(\BF)}, \LC)(m33d){};

  \coordinate (m71) at ({\RC*cos(\BB)}, {\RC*sin(\BB)}, \LG);
\node at (m71) [right] {$M_3$};
  \coordinate (m71d) at ({\RA*cos(\BC)}, {\RA*sin(\BC)} ,\LG);
\node at  (m71d) [right] {$\rdual{M_1}$};
  \coordinate (m72) at ({\RA*cos(\BD)}, {\RA*sin(\BD)}, \LG);
\node at (m72) [right] {$M_1$};
  \coordinate(m72d) at ({\RB*cos(\BE)}, {\RB*sin(\BE)}, \LG);
\node at  (m72d) [right] {$\rdual{M_2}$};
  \coordinate (m73) at  ({\RB*cos(\BF)}, {\RB*sin(\BF)}, \LG);
 \node at(m73) [right] {$M_2$};
  \coordinate (m73d) at  ({\RC*cos(\BA)}, {\RC*sin(\BA)}, \LG);
\node at (m73d) [right] {$\rdual{M_3}$};

  \node at ({\RA*cos(\BCA)}, {\RA*sin(\BCA)}, \LI) (e91) [circle, draw, fill = white] {$\epsilon_{M_1}$};
  \node at ({\RB*cos(\BEA)}, {\RB*sin(\BEA)}, \LIA) (e92) [circle, draw, fill = white] {$\epsilon_{M_2}$};
  \node at ({\RC*cos(\BAA)}, {\RC*sin(\BAA)}, \LIB) (e93) [circle, draw, fill = white] {$\epsilon_{M_3}$};

  \coordinate (u91) at ({\RA*cos(\BCA)}, {\RA*sin(\BCA)}, \LIC);
\node at  (u91) [right] {$U$};
  \coordinate (u92) at ({\RB*cos(\BEA)}, {\RB*sin(\BEA)}, \LID);
\node at  (u92) [right] {$U$};
  \coordinate (u93) at ({\RC*cos(\BAA)}, {\RC*sin(\BAA)}, \LIE);
\node at  (u93) [right] {$U$};

 \begin{scope}[on background layer]

  \draw [mydark, very thick] (e91) to [out = 150, in = -90]
(m71d) to [out =90, in =-90] 
(m31d)  to [out=90,in =-90]  (m21d) to [out = 90, in = -30] (nm1); 
  \draw [mydark, very thick] (nm1) to [out = -150, in =90] (m21) --(f21);
 \draw [mylight, very thick] (f21)  to [out=-90,in =90]  (m31)  to [out=-90,in =90] 
(m71) 
to [out=-90, in =30] (e93);
  \draw [mymed, very thick ]  (e92) to [out = 150, in = -90]  
 (m72d) to [out =90, in =-100]  
(m32d)  to [out =90, in =-90]  (m22d) to [out = 90, in =-30] (nm2);  
\draw [mymed, very thick] (nm2) to [out =-150, in =90]  (m22) -- (f22);
 \draw [mydark, very thick] (f22) to [out =-90, in =90]  (m32)  to [out =-90, in =90]  
(m72)  to [out =-90, in =30]  
(e91);
  \draw [mylight, very thick ](e93) to [out = 150, in = -90]  
(m73d)to [out =90, in =-90]  
(m33d) -- (m23d) to [out = 90, in = -30]  (nm3);
\draw [mylight, very thick](nm3) to [out =-150, in =90] (m23) -- (f23);
 \draw [mymed, very thick] (f23) to [out =-90, in =90]  (m33) to [out =-90, in =90]  
 (m73)
 to [out =-90, in =30]  
(e92);

  \draw[mydark, very thick] (ua1) -- (nm1);
  \draw[mymed, very thick] (ua2) -- (nm2);
  \draw[mylight, very thick ] (ua3) -- (nm3);
  \draw[mydark, very thick] (u91) -- (e91);
  \draw[mymed, very thick] (u92) -- (e92);
  \draw[mylight, very thick] (u93) -- (e93);

    \end{scope}

     \end{tikzpicture}
     }
  \caption{The multitrace} \label{fig:n_to_fuller_4}
\end{subfigure}
	\\
 \begin{subfigure}[t]{.45\linewidth}
\resizebox{\textwidth}{!}{
\centering
\begin{tikzpicture}[tdplot_main_coords, scale=.5]

   \pgfmathsetmacro{\RA}{10}
   \pgfmathsetmacro{\RB}{10}
   \pgfmathsetmacro{\RC}{10}

   \pgfmathsetmacro{\LAB}{25}
   \pgfmathsetmacro{\LAE}{23}
   \pgfmathsetmacro{\LBC}{21}
   \pgfmathsetmacro{\LDB}{19}
   \pgfmathsetmacro{\LCB}{16}
   \pgfmathsetmacro{\LDA}{13}
   \pgfmathsetmacro{\LGA}{10}
   \pgfmathsetmacro{\LCC}{5}
   \pgfmathsetmacro{\LD}{6}
   \pgfmathsetmacro{\LGC}{2}
   \pgfmathsetmacro{\LIB}{-1}
   \pgfmathsetmacro{\LIE}{-3}

\pgfmathsetmacro{\BA}{-15}
\pgfmathsetmacro{\BAA}{\BA+10}
\pgfmathsetmacro{\BB}{\BA+20}
\pgfmathsetmacro{\BBA}{\BA+30}
\pgfmathsetmacro{\BC}{\BA+40}
\pgfmathsetmacro{\BCA}{\BA+50}
\pgfmathsetmacro{\BD}{\BA+60}
\pgfmathsetmacro{\BDA}{\BA+70}
\pgfmathsetmacro{\BE}{\BA+80}
\pgfmathsetmacro{\BEA}{\BA+90}
\pgfmathsetmacro{\BF}{\BA+100}
\pgfmathsetmacro{\BFA}{\BA+30}

   \pgfmathsetmacro{\AA}{-20}
   \pgfmathsetmacro{\AB}{-30}
   \pgfmathsetmacro{\AC}{-5}
   \pgfmathsetmacro{\AD}{-20}
   \pgfmathsetmacro{\ADD}{-7}
   \pgfmathsetmacro{\AE}{5}
    \pgfmathsetmacro{\AEE}{12}
   \pgfmathsetmacro{\AF}{20}
   \pgfmathsetmacro{\AFF}{30}
   \pgfmathsetmacro{\AG}{40}
      \pgfmathsetmacro{\AGG}{50}
   \pgfmathsetmacro{\AH}{60}
   \pgfmathsetmacro{\AHH}{75}
   \pgfmathsetmacro{\AI}{90}
   \pgfmathsetmacro{\AII}{100}
   \pgfmathsetmacro{\AJ}{110}

   \draw [dotted,domain=\BA:\BF] plot ({\RC*cos(\x)}, {\RC*sin(\x)},\LAB);

   \draw [dotted,domain=\BA:\BF] plot ({\RB*cos(\x)}, {\RB*sin(\x)},\LBC);

   \draw [dotted,domain=\BA:\BF] plot ({\RA*cos(\x)}, {\RA*sin(\x)},\LCB);
   
   \draw [dotted,domain=\BA:\BF] plot ({\RA*cos(\x)}, {\RA*sin(\x)},\LGC);
   \draw [dotted,domain=\BA:\BF] plot ({\RA*cos(\x)}, {\RA*sin(\x)},\LGA);
   \draw [dotted,domain=\BA:\BF] plot ({\RC*cos(\x)}, {\RC*sin(\x)},\LIE);

  \coordinate (ua3) at ({\RC*cos(\BEA)}, {\RC*sin(\BEA)},\LAB);
\node at  (ua3) [right] {$U$};

  \node at ( {\RC*cos(\BEA)}, {\RC*sin(\BEA)},\LAE) (nm3) [circle, draw, fill = white] {$\eta_{M_3}$};

  \coordinate (m23) at  ({\RC*cos(\BE)}, {\RC*sin(\BE)}, \LBC);
\node at  (m23) [right] {$M_3$};
  \coordinate (m23d) at  ({\RC*cos(\BF)}, {\RC*sin(\BF)}, \LBC);
\node at  (m23d) [right] {$\rdual{M_3}$};

  \node at ({\RA*cos(\BA)}, {\RA*sin(\BA)}, \LD) (f21) [circle, draw, fill = white] {$f_1$};
  \node at ({\RB*cos(\BC)}, {\RB*sin(\BC)}, \LDA) (f22) [circle, draw, fill = white] {$f_2$};
  \node at  ({\RC*cos(\BE)}, {\RC*sin(\BE)}, \LDB) (f23) [circle, draw, fill = white] {$f_3$};

  \coordinate (m31) at ({\RC*cos(\BA)}, {\RC*sin(\BA)}, \LCC);
  \coordinate (m33) at ({\RB*cos(\AG)}, {\RB*sin(\AG)}, \LCB);
\node at (m33) [below right] {$M_2$};

  \coordinate (m63d) at  ({\RC*cos(\BF)}, {\RC*sin(\BF)}, \LD);
  
  \coordinate (m71) at ({\RC*cos(\BB)}, {\RC*sin(\BB)}, \LGC);
\node at  (m71) [right] {$M_3$};
  \coordinate (m72) at ({\RA*cos(\BB)}, {\RA*sin(\BB)}, \LGA);
\node at (m72) [above right] {$M_1$};
  \coordinate (m73d) at  ({\RC*cos(\BA)}, {\RC*sin(\BA)}, \LGC);
\node at (m73d) [right] {$\rdual{M_3}$};

  \node at ({\RC*cos(\BAA)}, {\RC*sin(\BAA)}, \LIB) (e93) [circle, draw, fill = white] {$\epsilon_{M_3}$};

  \coordinate (u93)  at ({\RC*cos(\BAA)}, {\RC*sin(\BAA)}, \LIE);
\node at  (u93) [right] {$U$};

 \begin{scope}[on background layer]
 \draw [mydark, very thick] (f21) to[out =90, in =-90] (f22);
 \draw [mylight, very thick] (f21) to [out=-90,in =90]  (m31)  to [out=-90,in =90]  
(m71);
 \draw [mylight, very thick] (m71) to [out=-90, in =45] (e93);
 \draw [mymed, very thick] (f22) to[out =90, in =-90]  (f23);
  \draw [mylight, very thick ](e93) to [out = 140, in = -90]   
(m73d)to [out =90, in =-90]   (m63d) -- 
(m23d) to [out = 90, in = -40]  (nm3);
\draw [mylight, very thick](nm3) to [out =-140, in =90] (m23) -- (f23);
  \draw[mylight, very thick ] (ua3) -- (nm3);
  \draw[mylight, very thick] (u93) -- (e93);

    \end{scope}

\end{tikzpicture}}
 \caption{The trace of $f^3$.} \label{fig:n_to_fuller_6}
 \end{subfigure}
        \end{tabular}
    \end{tabular}
\caption{Untwisting the Fuller trace
}
   \label{fig:n_to_fuller_smc_2}
\end{figure}

\begin{example} If $X$ is a finite or finitely dominated complex and $f\colon X \to X$, 
	the {\bf Lefschetz number} $L(f)$ is the trace of
	\[ \Sigma^\infty_+ f\colon \Sigma^\infty_+ X \to \Sigma^\infty_+ X \]
	in the stable homotopy category. This trace is a self-map of the sphere spectrum $\Sph = \Sigma^\infty_+ *$. The above theorem implies that $L(\ful{f}k) = L(f^k)$, which is the Lefschetz version of \cref{thm:main_first_half}, see also \cite[4.4]{fuller_67}. 
\end{example}

Our proof of \cref{thm:main_first_half} will essentially be a generalization of the above proof. In the more general setting of a shadowed bicategory, we will re-arrange the Fuller trace into a map as in \cref{fig:n_to_fuller_4} that we call the {\bf multitrace}, and then re-arrange the multitrace into the trace of the composite. The latter step does not require us to re-order objects, so we can do it in any shadowed bicategory. The former does require us to re-order the objects, so we need to ask for more structure beyond that of a bicategory. We will show that this step can be performed anytime we have a ``shadowed $n$-Fuller structure,'' defined in \S\ref{sec:fuller_bicat}. In \cite{mp2} we show that parametrized spectra have such a structure.

\begin{example}
The multitrace defined by \cref{fig:n_to_fuller_4} coincides with the multitrace from \cite[\S 4]{schlichtkrull}, cf. \cite[(2.6.4),(2.6.5)]{madsen_survey}. To recall it explicitly, let $A_0,\ldots , A_k$ be $n\times n$ matrices with coefficients in a field, and $V$ an $n$ dimensional vector space with basis $\{e_i\}$. The coevaluation map of $V$ is given by linearly extending $1\mapsto \sum e_i\otimes e_i^*$, and the evaluation map is given by linearly extending $(\phi,v)\mapsto \phi(v)$.  

The image of $(1,\ldots 1)$ under the multitrace of $(A_0,\ldots ,A_k)$ is then
	\[ \sum_{0\leq i_1,\ldots, i_k\leq n} a^0_{i_ki_0}\otimes a^1_{i_0i_1}\otimes \ldots \otimes a^k_{i_{k-1}i_k}\]
where $a_{j,l}^m$ is the $(j,l)$ entry of $A_m$.
\end{example}

\section{Bicategories and shadows}\label{sec:bicategories}

A \textbf{bicategory} $\sB$ consists of the following data:
  \begin{itemize}
  \item A collection of \emph{objects} or \emph{0-cells} $R,S,T,\dots$.
  \item For each pair of objects, a category $\sB(R,S)$.
  \item For each object, a \emph{unit} $U_R\in \sB(R,R)$.
  \item For each triple of objects, a \emph{composition} functor
    \[\odot\colon \sB(R,S)\times\sB(S,T)\to \sB(R,T).\]
  \item Associator and unit isomorphisms
      \begin{align*}
      \fa\colon M\odot (N\odot P) &\xto{\sim} (M\odot N)\odot P\\
      \frl\colon U_R \odot M &\xto{\sim} M\\
      \fr\colon M\odot U_S &\xto{\sim} M
      \end{align*}
      satisfying the same coherence axioms as for a monoidal category.
  \end{itemize}
The objects of $\sB(R,S)$ are called \emph{1-cells} and the morphisms are \emph{2-cells}. We think of these as ``monoidal categories with many objects'' and the operation $\odot$ as a tensor product. The coherence theorem for bicategories \cite{power} allows us to tensor a string of several 1-cells in a well-defined way up to canonical isomorphism, hence we often omit parentheses from expressions such as $M \odot N \odot P$.

A \textbf{shadow functor} on a bicategory $\sB$ is a 1-category $\bT$, a functor $\bigsh{-}\colon \sB(R,R) \to \bT$ for each 0-cell $R$, and natural isomorphisms
\[\theta\colon \sh{M\odot N}\xto{\sim} \sh{N\odot M}\]
satisfying following two coherence conditions.
\[ \resizebox{\textwidth}{!}{\xymatrix{\bigsh{(M\odot N)\odot P} \ar[r]^\theta \ar[d]_{\sh{\fa}} &
	\bigsh{P \odot (M\odot N)} \ar[r]^{\sh{\fa}} &
	\bigsh{(P\odot M) \odot N}\\
	\bigsh{M\odot (N\odot P)} \ar[r]^\theta & \bigsh{(N\odot P)
		\odot M} \ar[r]^{\sh{\fa}} & \bigsh{N\odot (P\odot
		M)}\ar[u]_\theta }
\quad
\xymatrix{\bigsh{M\odot U_R} \ar[r]^\theta \ar[dr]_{\sh{\fr}} &
	\bigsh{U_R\odot M} \ar[d]^{\sh{\frl}} \ar[r]^\theta &
	\bigsh{M\odot U_R} \ar[dl]^{\sh{\fr}}\\
	&\bigsh{M}}
}\]
This makes $\sB$ into a \textbf{bicategory with shadow}.

The point of a shadowed bicategory is that the 1-cells can be tensored along a circle. Given 1-cells $M_i \in \sB(R_{i-1},R_i)$, indices taken mod $n$, define their \textbf{circular product} by
\[ \sh{M_1,\ldots,M_n} := \sh{(\ldots((M_1 \odot M_2) \odot M_3)\ldots \odot M_n)}. \]
The following allows us to work with such products without worrying about parenthesization.
\begin{thm}[Coherence for shadowed bicategories {\cite[Theorem 9.12
]{mp2}}]
	If a functor is naturally isomorphic to the circular product by a composition of isomorphisms $\fa, \frl, \fr, \theta$, then there is only one such isomorphism.
\end{thm}

\begin{example}[Examples of bicategories and shadows]\label{ex:bicategories}\label{ex:shadows}\hfill 
\begin{enumerate}
	\item If $\sC$ is a monoidal category, it is also a bicategory with one object. If $\sC$ is a symmetric monoidal category, this is a shadowed bicategory in which $\bT = \sC$, the shadow functor is $\id\colon \sC \to \sC$, and $\theta$ is the symmetry isomorphism in $\sC$.

	\item There is a bicategory of bimodules and homomorphisms where the 0-cells are rings $R$, the 1-cells are bimodules $_R M_S$ and the 2-cells are bimodule homomorphisms $_R M_S \to \,_R N_S$.  The composition functor $\odot$ is the tensor product
	and the unit $U_R$ is $R$ as a bimodule over itself. 
There is a shadow functor 
that assigns $_R M_R$ to the quotient $M/\langle rm-mr\rangle$. 
	This can be generalized by taking the 1-cells to be chain complexes and the 2-cells to be maps in the derived category.

	\item There is a point-set bicategory of parameterized spaces $\bicat U\bS$. The 0-cells are spaces $A$ and $\bicat U\bS(A,B)$ is the category of spaces $X \to A \times B$. The composition of $X \to A \times B$ and $Y \to B \times C$ is the fiber product $X \times_B Y$ and the unit $U_B$ is the diagonal $B \to B \times B$. The shadow sends $X \to B \times B$ to the pullback along the diagonal $B \to B \times B$, which gives an unbased space.

\item There is a homotopy bicategory of spaces $\ho(\bicat U\bS)$. It is obtained by inverting the weak homotopy equivalences in each of the categories $\bicat U\bS(A,B)$, and replacing $\odot$ and $\sh{}$ by their right-derived functors. In particular, the shadow of the 1-cell in $\bicat U\bS(X,X)$ given by $X \xto{(f,\id)} X \times X$ is the twisted free loop space $\Lambda^f X$.

	\item
There is a bicategory $\Ex$ of parameterized spectra \cite[Ch. 17]{ms}.  Its 0-cells are spaces $A$, and the category $\Ex(A,B)$ is the homotopy category of spectra parametrized by the product space $A\times B$. Each parametrized space $X \to A \times B$ in $\ho(\bicat U\bS(A,B))$ has a suspension spectrum $\Sigma^\infty_{+(A \times B)} X$ in $\Ex(A,B)$. 
The shadow functor from $\Ex$ to spectra agrees with the one in $\ho(\bicat U\bS)$ along the suspension spectrum functor. 

\item The last three examples admit generalizations $G\bicat U{G\bS}, \ho G\bicat U{G\bS}, G\Ex$ by allowing the action of a finite group $G$. When forming the homotopy category $\ho G\bicat U{G\bS}$, we invert those maps of $G$-spaces that are equivalences on the $H$-fixed points for every $H \leq G$.
\end{enumerate}
\end{example}

A \textbf{pseudofunctor} is a homomorphism of bicategories $F\colon \sC \to \sD$. It consists of the following data.
\begin{itemize}
	\item A function $\ob \sC \to \ob \sD$ of 0-cells, denoted by $F$.
	\item A functor $\sC(R,S) \to \sD(F(R),F(S))$ for each pair of 0-cells in $\sC$, denoted by $F$.
	\item Natural isomorphisms
	\begin{align*}
	m\colon F(M) \odot F(N) &\xto{\sim} F(M \odot N) \\
	i\colon U_{F(R)} &\xto{\sim} F(U_R)
	\end{align*}
	satisfying the same coherence axioms as for a strong monoidal functor.
\end{itemize}
A \textbf{strong shadow functor} is a homomorphism of shadowed bicategories $F\colon (\sC,\bT_{\sC}) \to (\sD,\bT_{\sD})$. It consists of a pseudofunctor and the following additional data.
\begin{itemize}
	\item A functor of shadow categories $F_{\tr}\colon \bT_{\sC} \to \bT_{\sD}$.
	\item Natural isomorphisms $s\colon \sh{F(M)} \xto{\sim} \Ft \sh{M}$ such that
  \[\xymatrix{\sh{F(M)\odot F(N)} \ar[r]^\theta\ar[d]_{\sh{m}} & \sh{F(N)\odot F(M)} \ar[d]^{\sh{m}}\\
    \sh{F(M\odot N)}\ar[d]_s & \sh{F(N\odot M)} \ar[d]^s\\
    \Ft\,\sh{M\odot N} \ar[r]_{\Ft(\theta)} & \Ft\,\sh{N\odot M}.}
  \]
	commutes whenever it makes sense.
\end{itemize}
If $F,G \colon \sC \to \sD$ are strong shadow functors that are the same function of 0-cells, an \textbf{isomorphism of strong shadow functors} from $F$ to $G$ consists of natural isomorphisms $F \cong G$ and $F_{\tr} \cong G_{\tr}$ that commute with $m$, $i$, and $s$. We will often implicitly work with these functors up to isomorphism.

\section{Duality and trace for bicategories}\label{sec:traces_in_bicat}

A 1-cell $M\colon R\hto S$ in a bicategory is \textbf{right dualizable}, or \textbf{dualizable over $S$}, if there is a 1-cell $\rdual{M}\colon S\hto R$, and coevaluation and evaluation 2-cells
	\[\eta\colon U_R \to M\odot \rdual{M}\text{ and  }\epsilon\colon \rdual{M}\odot M\to U_S\] 
satisfying the triangle identities.
We say that $(M,\rdual{M})$ is a \textbf{dual pair}, that $\rdual{M}$ is \textbf{left dualizable} or \textbf{dualizable over $S$}.

\begin{example}[Dualizable objects]\label{ex:dualizable_objects}\hfill 
	\begin{enumerate}
		\item An object $M$ is dualizable in the symmetric monoidal category $\sC$ \tiff it is right (or left) dualizable in the bicategory associated to $\sC$.

		\item If $A$ and $R$ are rings, a bimodule $_A M_R$ is right dualizable precisely when it is finitely generated and projective as a right $R$-module, in which case the dual is $\Hom_R(M,R)$. Of course, $M$ left dualizable when it is finitely generated and projective as a left $A$-module. 
	\end{enumerate}
\end{example}

\begin{example}\label{ex:base_change}
The graph $A \xto{\id_A,f}  A \times B$ of a map $f\colon A\to B$ of unbased spaces 
defines  1-cells in $\ho(\bicat U\bS)(A,B)$ and $\ho(\bicat U\bS)(B,A)$. Taking suspension spectra gives two different 1-cells in $\Ex$, which we call the {\bf base-change} 1-cells associated to $f$:
\[ \bcl AfB\coloneqq \Sigma^\infty_{+(A \times B)} A, \qquad \bcr AfB\coloneqq \Sigma^\infty_{+(B \times A)} A. \] 
\begin{itemize}
	\item The 1-cell $\bcl AfB$ is always right dualizable \cite[17.3.1]{ms}. 
	\item If $p\colon E\to B$ is a perfect fibration, i.e. a fibration whose fibers are finitely dominated, then $\bcr EpB$ is right dualizable \cite[4.7]{PS:mult}. When there is a $G$-action, the same is true if the fiber is equivariantly finitely dominated, meaning it is a retract in the homotopy category of $G$-spaces of a finite $G$-CW complex \cite[4.3]{PS:mult}\cite[18.2.1]{ms}. 
\end{itemize}
\end{example}

Base change objects define a pseudofunctor $\bS\to \ho(\bicat U\bS)$.  In particular, 
there are coherent isomorphisms
\begin{align}\label{eq:compositions_of_base_change_in_spectra}
m_{[]}\colon \bcr{B}{g}{C} \odot \bcr{A}{f}{B} &\xto{\sim} \bcr{A}{g \circ f}{C}, \qquad
i_{[]}\colon U_A \xto{\sim} \bcr{A}{\id}{A}.
\end{align}
for each pair of composable maps $A \overset{f}\to B \overset{g}\to C$ and each 0-cell $A$.
The same discussion applies with $G$-equivariant spaces as well.

Let $\sB$ be a bicategory with a shadow functor to $\bT$, and $M$ a right dualizable 1-cell of $\sB$.
  The {\bf trace} of a 2-cell $f\colon Q\odot M\rightarrow M\odot P$ is the morphism in $\bT$ is the composite:
  \[ \xymatrix{
  	\sh{Q} \ar[r]^-{\eta} & \sh{Q, M, \rdual{M}} \ar[r]^-{f} & \sh{M, P, \rdual{M}} \ar[r]^-{\epsilon} & \sh{P}.
  } \]
When $\sB$ comes from a symmetric monoidal category, this is the trace as defined in \cref{sec:smc}.

As in the symmetric monoidal case, it is helpful to visualize these traces using the string diagram calculus for bicategories from \cite{ps:bicat}. (The rigor of this approach is established in the appendix of that paper.  As in the symmertic monoidal case any string diagram can be translated into a conventional commutative diagram.)  We represent 0-cells by 2-dimensional regions, 1-cells by strings, and 2-cells by vertices; see \cref{fig:string-bicat}. 
Pasting pictures together corresponds to horizontal (tensoring) and vertical composition of the resulting expressions in the bicategory.
\begin{figure}[tb]
 \begin{subfigure}[b]{.15\linewidth}
\centering
      \begin{tikzpicture}
       \fill[mydarkfill] (0,0) rectangle (2,2);
       \node[anchor= north west,blue] at (0,2) {$R$};
     \end{tikzpicture}
\caption{Object $R$}
\end{subfigure}
  \hspace{.5cm}
\begin{subfigure}[b]{.2\linewidth}
\centering
   \begin{tikzpicture}
       \fill[mydarkfill] (0,0) rectangle (1,2);
       \fill[mylightfill] (1,0) rectangle (2,2);
       \node[anchor=north west,blue] at (0,2) {$R$};
       \node[anchor=north east,green!70!black] at (2,2) {$S$};
       \draw[-](1,2)--(1, 0);
       \node[anchor= west] at (1,1){$M$};
     \end{tikzpicture}
\caption{1-cell $R\xrightarrow{M} S$}
\end{subfigure}
  \hspace{.5cm}
\begin{subfigure}[b]{.25\linewidth}
\centering
   \begin{tikzpicture}
      \fill[mydarkfill] (0,0) rectangle (1,2);
      \fill[mylightfill] (1,0) rectangle (2,2);
      \fill[mymedfill] (2,0) rectangle (3,2);
      \node[anchor=north west,blue] at (0,2) {$R$};
      \node[anchor=north,green!70!black] at (1.5,2) {$S$};
      \node[anchor=north east,red] at (3,2) {$T$};
      \draw[-](1,2)--(1,0);
      \draw[-](2,2)--(2,0);
      \node[anchor= east]at (1,1){$M$};
      \node[anchor= west]at (2,1){$N$};
     \end{tikzpicture}
\caption{Composite $M\odot N$}
\end{subfigure}
  \hspace{.5cm}
\begin{subfigure}[b]{.2\linewidth}
\centering
   \begin{tikzpicture}
       \fill[mydarkfill] (0,0) rectangle (1,2);
       \fill[mylightfill] (1,0) rectangle (2,2);
       \node[anchor=north west,blue] at (0,2) {$R$};
       \node[anchor=north east,green!70!black] at (2,2) {$S$};
       \node[fill=white,draw,circle,inner sep=1pt] (f) at (1,1) {$f$};
       \draw[-] (1,2) -- (f);
       \draw[-](f) -- (1,0);
       \node[anchor=west] at(1, 1.5){$M$};
       \node[anchor=west] at(1, .5){$N$};
     \end{tikzpicture}
\caption{2-cell 
}
\end{subfigure}
\begin{subfigure}[b]{.25\linewidth}
\centering
   \begin{tikzpicture}
      \fill[mydarkfill] (0,0) rectangle (1,2);
      \fill[mylightfill] (1,0) rectangle (2,2);
      \fill[mymedfill] (2,0) rectangle (3,2);
      \node[anchor=north west,blue] at (0,2) {$R$};
      \node[anchor=north,green!70!black] at (1.5,2) {$S$};
      \node[anchor=north east,red] at (3,2) {$T$};
      \draw[-](1,2)--(1,0);
      \draw[-](2,2)--(2,0);
      \node[anchor= east]at (1,1){$M$};
      \node[anchor= west]at (2,1){$N$};
\filldraw [pattern color = black!60!white, pattern = dots, draw =white]  (.5,0) rectangle (2.5,2);
     \end{tikzpicture}
\caption{$F(M\odot N)$}
\end{subfigure}
  \hspace{.5cm}
\begin{subfigure}[b]{.25\linewidth}
\centering
   \begin{tikzpicture}
      \fill[mydarkfill] (0,0) rectangle (1,2);
      \fill[mylightfill] (1,0) rectangle (2,2);
      \fill[mymedfill] (2,0) rectangle (3,2);
      \node[anchor=north west,blue] at (0,2) {$R$};
      \node[anchor=north,green!70!black] at (1.5,2) {$S$};
      \node[anchor=north east,red] at (3,2) {$T$};
      \draw[-](1,2)--(1,0);
      \draw[-](2,2)--(2,0);
      \node[anchor= east]at (1,1){$M$};
      \node[anchor= west]at (2,1){$N$};
\filldraw [pattern color = black!60!white, pattern = dots, draw =white]  (.5,0) rectangle (1.25,2);
\filldraw [pattern color = black!60!white, pattern = dots, draw =white]  (1.75,0) rectangle (2.5,2);
     \end{tikzpicture}
\caption{ $F(M)\odot F(N)$}
\end{subfigure}
  \hspace{.5cm}
\begin{subfigure}[b]{.25\linewidth}
\centering
\begin{tikzpicture}
\filldraw[blue!60!white] (0,0) ellipse (1.5 and 0.35);
\filldraw[mydarkfill] (-1.5,0) -- (-1.5,-1.5) arc (180:360:1.5 and 0.35) -- (1.5,0) arc (0:180:1.5 and -0.35);
\draw[-] (0,-.35) --(0, -1.85);
\node[anchor=west] at(0, -1){$M$};
\end{tikzpicture}
\caption{$\sh{M}$}\label{cyl:shadow}
\end{subfigure}
   \caption{String diagrams for bicategories}
   \label{fig:string-bicat}
\end{figure}
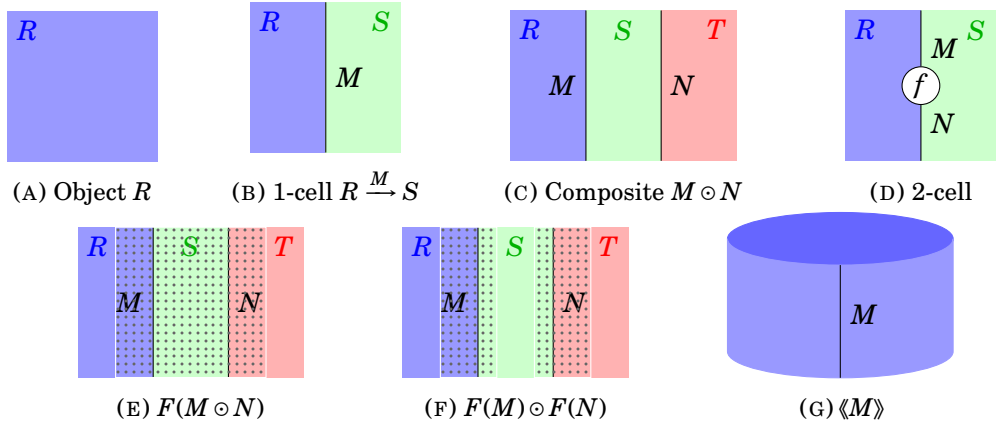
To extend this visual language to bicategories with shadow, we represent the shadow by closing a planar string diagram into a cylindrical string diagram. See \cref{cyl:shadow}. See \cref{bitrace} for the string diagram for the bicategorical trace, cf. the earlier string diagram in \cref{fig:smc_trace}.
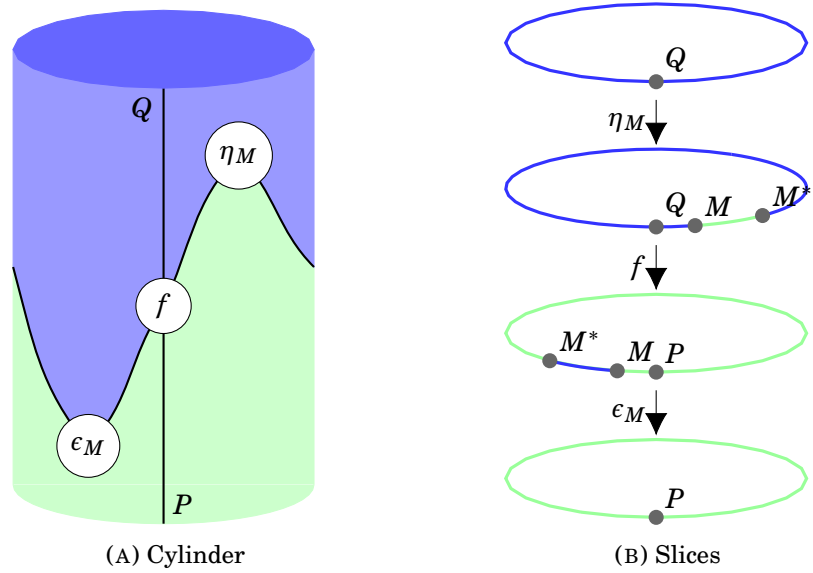
\begin{figure}
\tdplotsetmaincoords{75}{120}
\centering
    \begin{subfigure}[b]{0.30\textwidth}
{
\begin{tikzpicture}[tdplot_main_coords, scale =.5]
   \pgfmathsetmacro{\RA}{4}

   \pgfmathsetmacro{\LA}{6}
   \pgfmathsetmacro{\LB}{4}
   \pgfmathsetmacro{\LC}{0}
   \pgfmathsetmacro{\LD}{-4}
   \pgfmathsetmacro{\LE}{-6}

\pgfmathsetmacro{\AA}{-25}
\pgfmathsetmacro{\AB}{\AA+15}
\pgfmathsetmacro{\AC}{\AA+25}
\pgfmathsetmacro{\AD}{\AA+25}
\pgfmathsetmacro{\AE}{\AA+55}
\pgfmathsetmacro{\AF}{\AA+85}
\pgfmathsetmacro{\AG}{\AA+85}
\pgfmathsetmacro{\AH}{\AA+100}
\pgfmathsetmacro{\AI}{\AA+115}
\pgfmathsetmacro{\AJ}{120}
\pgfmathsetmacro{\AK}{\AA-35}

   \filldraw [blue!60!white,domain=0:360] plot ({\RA*cos(\x)}, {\RA*sin(\x)},\LA);

\coordinate (tl) at ({\RA*cos(\AK)}, {\RA*sin(\AK)}, \LA);
\coordinate (tr) at ({\RA*cos(\AJ)}, {\RA*sin(\AJ)}, \LA);
\coordinate (bl) at ({\RA*cos(\AK)}, {\RA*sin(\AK)}, \LE);
\coordinate (br) at ({\RA*cos(\AJ)}, {\RA*sin(\AJ)}, \LE);

\coordinate (s1) at  ({\RA*cos(\AK)}, {\RA*sin(\AK)}, \LC);
\coordinate (s2) at  ({\RA*cos(\AJ)}, {\RA*sin(\AJ)}, \LC);

\coordinate (n2) at  ({\RA*cos(\AF)}, {\RA*sin(\AF)}, \LB);
\node at (n2)[circle, draw , fill =white] {$\eta_{M}$};

\coordinate (q2) at ({\RA*cos(\AE)}, {\RA*sin(\AE)}, \LA);
\node at (q2)[below left] {$Q$};

\coordinate (f2) at ({\RA*cos(\AE)}, {\RA*sin(\AE)}, \LC);
\node at (f2)[circle, draw , fill =white] {$f$};

\coordinate (e2) at ({\RA*cos(\AD)}, {\RA*sin(\AD)}, \LD);
\node at (e2)[circle, draw , fill =white] {$\epsilon_{M}$};

\coordinate (p2) at  ({\RA*cos(\AE)}, {\RA*sin(\AE)}, \LE);
\node at (p2)[above right] {$P$};

 \begin{scope}[on background layer]
\filldraw[mylightfill] (bl)-- (s1)to [out = -75, in =130]
(e2)to  [out= 45, in = -125]  (f2)to [out= 65, in = -135]  (n2)
to  [out= -45, in = 135] (s2) -- (br) [domain=\AK:\AJ] plot ({\RA*cos(\x)}, {\RA*sin(\x)},\LE);

\filldraw[mydarkfill] (tl)-- (s1)to [out = -75, in =130] 
 (e2)to  [out= 45, in = -125]  (f2)to [out= 65, in = -135]  (n2)
to  [out= -45, in = 135] (s2) --  (tr) [domain=\AK:\AJ] plot ({\RA*cos(\x)}, {\RA*sin(\x)},\LA);

\draw[thick]  (s1)to [out = -75, in =130] 
(e2)to  [out= 45, in = -125]  
(f2)to [out= 65, in = -135]  (n2)
to  [out= -45, in = 135] (s2) ;

\draw[thick] (q2)--(f2)-- (p2);
\end{scope}

\end{tikzpicture}
} \caption{Cylinder}

\end{subfigure}
\hspace{2cm}
    \begin{subfigure}[b]{0.30\textwidth}
\tdplotsetmaincoords{75}{120}
\centering
{
\begin{tikzpicture}[tdplot_main_coords, scale =.5,roundnode/.style={circle, fill=black!60, inner sep=0pt, minimum size=2mm}]
   \pgfmathsetmacro{\RA}{4}

   \pgfmathsetmacro{\LA}{6}
   \pgfmathsetmacro{\LB}{4}
   \pgfmathsetmacro{\LBA}{2}
   \pgfmathsetmacro{\LC}{0}
   \pgfmathsetmacro{\LCA}{-2}
   \pgfmathsetmacro{\LD}{-4}
   \pgfmathsetmacro{\LE}{-6}

\pgfmathsetmacro{\AA}{-25}
\pgfmathsetmacro{\AB}{\AA+10}
\pgfmathsetmacro{\AC}{\AA+25}
\pgfmathsetmacro{\AD}{\AA+40}
\pgfmathsetmacro{\AE}{\AA+55}
\pgfmathsetmacro{\AEA}{\AA+70}
\pgfmathsetmacro{\AF}{\AA+85}
\pgfmathsetmacro{\AG}{\AA+100}
\pgfmathsetmacro{\AH}{\AA+100}
\pgfmathsetmacro{\AI}{\AA+115}
\pgfmathsetmacro{\AJ}{120}
\pgfmathsetmacro{\AK}{\AA-35}

   \draw [very thick, mydark,domain=0:360] plot ({\RA*cos(\x)}, {\RA*sin(\x)},\LA);
	\coordinate (q1) at ({\RA*cos(\AE)}, {\RA*sin(\AE)},\LA); 
		\node[roundnode] at (q1){};
		\node at (q1)[above right] {$Q$};

   \draw [very thick, mydark,domain=-180:\AEA] plot ({\RA*cos(\x)}, {\RA*sin(\x)},\LBA);
   \draw [very thick, mydark,domain=\AG:180] plot ({\RA*cos(\x)}, {\RA*sin(\x)},\LBA);
   \draw [very thick, mylight,domain=\AG:\AEA] plot ({\RA*cos(\x)}, {\RA*sin(\x)},\LBA);
	\coordinate (q2) at  ({\RA*cos(\AE)}, {\RA*sin(\AE)},\LBA){}; 
		\node[roundnode] at (q2){};
		\node at (q2)[above right] {$Q$};
	\coordinate (m1) at ({\RA*cos(\AEA)}, {\RA*sin(\AEA)},\LBA){}; 
		\node[roundnode] at (m1){};
		\node at (m1)[above right] {$M$};
	\coordinate (dm1) at ({\RA*cos(\AG)}, {\RA*sin(\AG)},\LBA){};
		\node[roundnode] at (dm1){};
		\node at (dm1)[above right] {$M^*$};

   \draw [very thick, mylight,domain=0:360] plot ({\RA*cos(\x)}, {\RA*sin(\x)},\LCA);
   \draw [very thick, mydark,domain=\AB:\AD] plot ({\RA*cos(\x)}, {\RA*sin(\x)},\LCA);
	\coordinate (p1) at  ({\RA*cos(\AE)}, {\RA*sin(\AE)},\LCA){}; 
		\node[roundnode] at (p1){};
		\node at (p1)[above right] {$P$};
	\coordinate (m2) at ({\RA*cos(\AD)}, {\RA*sin(\AD)},\LCA){}; 
		\node[roundnode] at (m2){};
		\node at (m2)[above right] {$M$};
	\coordinate (dm2) at ({\RA*cos(\AB)}, {\RA*sin(\AB)},\LCA){};
		\node[roundnode] at (dm2){};
		\node at (dm2)[above right] {$M^*$};

   \draw [very thick, mylight,domain=0:360] plot ({\RA*cos(\x)}, {\RA*sin(\x)},\LE);
\coordinate (p2) at  ({\RA*cos(\AE)}, {\RA*sin(\AE)}, \LE);
\node[roundnode] at (p2){};
\node at (p2)[above right] {$P$};

\draw [->] ({\RA*cos(\AE)}, {\RA*sin(\AE)},\LA-.5) -- ({\RA*cos(\AE)}, {\RA*sin(\AE)},\LBA+2.25)node[midway,left ]{$\eta_{M}$};

\draw [->] ({\RA*cos(\AE)}, {\RA*sin(\AE)},\LBA-.5) -- ({\RA*cos(\AE)}, {\RA*sin(\AE)},\LCA+2.25)node[midway,left ]{$f$};

\draw [->] ({\RA*cos(\AE)}, {\RA*sin(\AE)},\LCA-.5) -- ({\RA*cos(\AE)}, {\RA*sin(\AE)},\LE+2.25)node[midway,left ]{$\epsilon_{M}$};

\end{tikzpicture}
} \caption{Slices}

\end{subfigure}
\caption{The bicategorical trace}\label{bitrace}
\end{figure}

\begin{example}
Suppose that $X$ is a finite or finitely dominated complex and $f\colon X \to X$.
\begin{itemize}

	\item The {\bf Reidemeister trace} $R(f)$ is the trace of the canonical isomorphism
		\[ \bcr X{p}* \xto{\sim} \bcr X{p}* \odot \bcr XfX. \]	
	This trace is a map in the homotopy category
		\[ R(f)\colon \Sph \simeq \Sigma^\infty_+ \Bigsh{\bcr {*}{}{*}} \to \Sigma^\infty_+ \Bigsh{\bcr XfX} \simeq \Sigma^\infty_+ \Lambda^f X \]
	(cf \cite[Appendix A]{cp}), which can be regarded as an element of $H_0(\Lambda^f X)$.
	
	\item The {\bf $n$th Fuller trace} $R_{C_n}(\ful fn)$ is the trace of the canonical isomorphism
		\[ \bcr {X^n}{p}* \xto{\sim} \bcr {X^n}{p}* \odot \bcr {X^n}{\ful fn}{X^n}. \]
	It is a map in the $C_n$-equivariant homotopy category
		\[ R_{C_n}(\ful fn)\colon \Sph \simeq \Sigma^\infty_+ \Bigsh{\bcr {*}{}{*}} \to \Sigma^\infty_+ \Bigsh{\bcr {X^n}{\ful fn}{X^n}} \simeq \Sigma^\infty_+ \Lambda^{\ful fn} X^n. \]
\end{itemize}

These are not the standard definitions of Lefschetz number or the Reidemeister trace.  The definition here for the Lefschetz number is shown to agree with more classical descriptions in \cite{dp}.  This description of the Reidemeister trace is compared to more classical versions in \cite{p:coincidences} and to the description in \cite{kw} in \cite[6.3.2]{p:thesis}. The fact that different constructions of $\Ex$ give the same base-change isomorphisms is handled carefully in \cite{spectra_notes}, so we refrain from commenting on it here.
\end{example}

We end this section by recalling a fundamental functoriality result for the trace.
\begin{thm}\label{bicategory_map_preserves_traces}\cite[8.3]{ps:bicat}
  Let $F\colon \sB\rightarrow \sC$ be a strong shadow functor and suppose $M\in \sB(R, S)$ is right dualizable.
  \begin{enumerate}
  \item Then $F(M)$ is right dualizable with dual $F(\rdual{M})$.\label{item:bfpt1}
  \item  For any $f\colon Q\odot M\rightarrow M\odot P$, the following square commutes:\label{item:bfpt2}
    \[\xymatrix@C=1.5in{\sh{F(Q)}\ar[r]^{\tr(m_{M,P}^{-1}\circ F(f)\circ m_{Q,M})}
      \ar[d]_s &\sh{F(P)}\ar[d]^s\\
      \Ft\,\sh{Q}\ar[r]^{\Ft(\tr(f))}&\Ft\,\sh{P}.}\]
  \end{enumerate}
\end{thm}

\section{The multitrace for bicategories and Fuller bicategories}\label{sec:fuller_bicat}
Now that we have defined the Reidemeister trace for $f^n$ and for $\ful{f}n$, we may begin the formal work of relating them together.

Suppose in a shadowed bicategory $\sB$ we select right dualizable 1-cells $M_i\in \sB(A_i,B_i)$, 1-cells $Q_i\in \sB(A_{i-1}, A_i)$, $P_i\in \sB(B_{i-1}, B_i)$ (subscripts taken mod $n$), and 2-cells \[\phi_i\colon Q_i\odot M_i\to M_{i-1}\odot P_i.\]
Then we define the ``composite'' $\seqco{\phi_1}{\phi_n}$ to be the composite of the 2-cells
 \begin{align*}
 Q_1\odot \ldots \odot Q_n \odot M_n\xto{\id^{n-1}\odot \phi_n} Q_1\odot \ldots \odot Q_{n-1} \odot M_{n-1}\odot P_n\xto{\id^{n-2}\odot \phi_{n-1}\odot \id}
 \\ Q_1\odot \ldots \odot Q_{n-2} \odot M_{n-2}\odot P_{n-1}\odot P_n \to \ldots \to M_n\odot P_1\odot \ldots \odot P_n.
 \end{align*}
If the modules $Q_i$ and $P_i$ are all units, this is canonically isomorphic to the composite of the maps $\phi_i$.
 
On the other hand, we define the {\bf multitrace} of the maps $\phi_i$, denoted $\tr(\phi_1,\ldots ,\phi_n)$, as the composite in $\bT$:
  \[ \xymatrix@C=60pt@R15pt{
 	\sh{Q_1,\ldots,Q_n}
		\ar[r]^-{\sh{\id,\eta_1,\id,\ldots,\id,\eta_n}} 
	& \sh{Q_1,M_1,\rdual{M}_1,Q_2,M_2,\rdual{M}_2,\ldots,M_n,\rdual{M}_n} 
		\ar[d]^-{\sh{\phi_1,\id,\ldots,\id,\phi_n,\id}} 
	\\
 	\sh{P_1,\ldots,P_n} 
	& \sh{M_n,P_1,\rdual{M}_1,M_1,P_2,\rdual{M}_2,\ldots,P_n,\rdual{M}_n}
		\ar[l]_-{\sh{\id,\epsilon_1,\id, \ldots,\id, \epsilon_n}}	
  } \]
  
\begin{thm}[Step 1 of \cref{thm:main_first_half}]\label{lem:nthpower} 
The multitrace equals the trace of the composite,
\[\tr(\phi_1,\ldots,\phi_n) = \tr(\seqco{\phi_1}{\phi_n}),\]
as maps $\sh{Q_1,\ldots,Q_n} \to \sh{P_1,\ldots,P_n}$.
 \end{thm}

\begin{proof}
Using the string diagram calculus of \cite{ps:bicat}, \cref{Comparison_n_traces} provides a full proof.

Alternatively,
if $X$ is dualizable,  the composite
\begin{equation}\label{eq:full_2}
A\odot C\odot Y\xto{\id\odot \eta \odot \id^2}A\odot X\odot \rdual{X}\odot C\odot Y\xto{f\odot \id\odot g}Y\odot B\odot \rdual{X}\odot X\odot D\xto{\id^2\odot \epsilon\odot \id}Y\odot B\odot D
\end{equation}
for 2-cells $f\colon A\odot X\to Y\odot B$ and $g\colon C\odot Y\to X\odot D$
is\[
A\odot C\odot Y\xto{\id \odot g}
A\odot X\odot D \xto{\id\odot \eta\odot \id^2} 
A\odot X\odot \rdual{X}\odot X\odot D\xto{\id^2\odot \epsilon\odot \id}
A\odot X\odot D\xto{f\odot \id} Y\odot B\odot D
\]
Canceling the center evaluation and coevaluation, \eqref{eq:full_2} is the composite 
\begin{equation}\label{eq:full_2_composed}
A\odot C\odot Y\xto{\id \odot g}
A\odot X\odot D\xto{f\odot \id} Y\odot B\odot D
\end{equation}
If $Y$ is also dualizable
the multitrace of $f$ and $g$ is the trace of \eqref{eq:full_2} and so the multitrace of $f$ and $g$ is the trace of \eqref{eq:full_2_composed}.
Then the theorem follows by induction.
\end{proof}\usetikzlibrary{patterns}
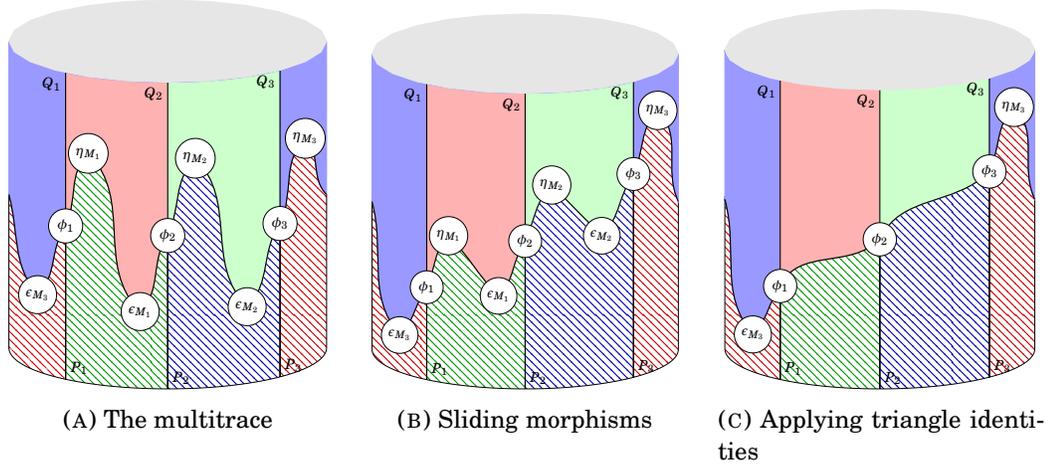
\begin{figure}
\tdplotsetmaincoords{75}{120}
\begin{subfigure}[t]{0.3\textwidth}
\centering
\resizebox{1\textwidth}{!}
{
\begin{tikzpicture}[tdplot_main_coords, scale =.5]
   \pgfmathsetmacro{\RA}{8}

   \pgfmathsetmacro{\LA}{8}
   \pgfmathsetmacro{\LB}{4}
   \pgfmathsetmacro{\LC}{0}
   \pgfmathsetmacro{\LD}{-4}
   \pgfmathsetmacro{\LE}{-8}

\pgfmathsetmacro{\AA}{-25}
\pgfmathsetmacro{\AB}{\AA+15}
\pgfmathsetmacro{\AC}{\AA+25}
\pgfmathsetmacro{\AD}{\AA+45}
\pgfmathsetmacro{\AE}{\AA+55}
\pgfmathsetmacro{\AF}{\AA+65}
\pgfmathsetmacro{\AG}{\AA+85}
\pgfmathsetmacro{\AH}{\AA+100}
\pgfmathsetmacro{\AI}{\AA+115}
\pgfmathsetmacro{\AJ}{120}
\pgfmathsetmacro{\AK}{\AA-35}

   \filldraw [black!10!white,domain=0:360] plot ({\RA*cos(\x)}, {\RA*sin(\x)},\LA);

\coordinate (tl) at ({\RA*cos(\AK)}, {\RA*sin(\AK)}, \LA);
\coordinate (tr) at ({\RA*cos(\AJ)}, {\RA*sin(\AJ)}, \LA);
\coordinate (bl) at ({\RA*cos(\AK)}, {\RA*sin(\AK)}, \LE);
\coordinate (br) at ({\RA*cos(\AJ)}, {\RA*sin(\AJ)}, \LE);

\coordinate (s1) at  ({\RA*cos(\AK)}, {\RA*sin(\AK)}, \LC);
\coordinate (s2) at  ({\RA*cos(\AJ)}, {\RA*sin(\AJ)}, \LC);

\coordinate (n1) at  ({\RA*cos(\AC)}, {\RA*sin(\AC)}, \LB);
\node at (n1)[circle, draw , fill =white] {$\eta_{M_1}$};
\coordinate (n2) at  ({\RA*cos(\AF)}, {\RA*sin(\AF)}, \LB);
\node at (n2)[circle, draw , fill =white] {$\eta_{M_2}$};
\coordinate (n3) at  ({\RA*cos(\AI)}, {\RA*sin(\AI)}, \LB);
\node at (n3)[circle, draw , fill =white] {$\eta_{M_3}$};

\coordinate (q1) at ({\RA*cos(\AB)}, {\RA*sin(\AB)}, \LA);
\node at (q1)[below left] {$Q_1$};
\coordinate (q2) at ({\RA*cos(\AE)}, {\RA*sin(\AE)}, \LA);
\node at (q2)[below left] {$Q_2$};
\coordinate (q3) at ({\RA*cos(\AH)}, {\RA*sin(\AH)}, \LA);
\node at (q3)[below left] {$Q_3$};

\coordinate (f1) at ({\RA*cos(\AB)}, {\RA*sin(\AB)}, \LC);
\node at (f1)[circle, draw , fill =white] {$\phi_1$};
\coordinate (f2) at ({\RA*cos(\AE)}, {\RA*sin(\AE)}, \LC);
\node at (f2)[circle, draw , fill =white] {$\phi_2$};
\coordinate (f3) at ({\RA*cos(\AH)}, {\RA*sin(\AH)}, \LC);
\node at (f3)[circle, draw , fill =white] {$\phi_3$};

\coordinate (e1) at ({\RA*cos(\AA)}, {\RA*sin(\AA)}, \LD);
\node at (e1)[circle, draw , fill =white] {$\epsilon_{M_3}$};
\coordinate (e2) at ({\RA*cos(\AD)}, {\RA*sin(\AD)}, \LD);
\node at (e2)[circle, draw , fill =white] {$\epsilon_{M_1}$};
\coordinate (e3) at ({\RA*cos(\AG)}, {\RA*sin(\AG)}, \LD);
\node at (e3)[circle, draw , fill =white] {$\epsilon_{M_2}$};

\coordinate (p1) at ({\RA*cos(\AB)}, {\RA*sin(\AB)}, \LE);
\node at (p1)[above right] {$P_1$};
\coordinate (p2) at  ({\RA*cos(\AE)}, {\RA*sin(\AE)}, \LE);
\node at (p2)[above right] {$P_2$};
\coordinate (p3) at  ({\RA*cos(\AH)}, {\RA*sin(\AH)}, \LE);
\node at (p3)[above right] {$P_3$};

 \begin{scope}[on background layer]
\filldraw[pattern =north west lines, pattern color = mydarkred] (bl)-- (s1)to [out = -75, in =130] (e1) to [out= 45, in = -125]  (f1)--(p1)
[domain=\AK:\AB] plot ({\RA*cos(\x)}, {\RA*sin(\x)},\LE);
\filldraw[pattern =north west lines, pattern color = mydarkgreen]
(p1)to (f1) to  [out= 65, in = -135]  (n1) 
to  [out= -45, in = 135]  (e2)to  [out= 45, in = -125]  (f2) to (p2)[domain=\AE:\AB] plot ({\RA*cos(\x)}, {\RA*sin(\x)},\LE);
\filldraw[pattern =north west lines, pattern color = mydarkblue](p2)to (f2) to [out= 65, in = -135]  (n2)
to [out= -45, in = 135]  (e3)to  [out= 45, in = -125] (f3) to (p3)[domain=\AH:\AE] plot ({\RA*cos(\x)}, {\RA*sin(\x)},\LE);
\filldraw[pattern =north west lines, pattern color = mydarkred] (p3)to (f3) to [out= 65, in = -135]   (n3)
to  [out= -45, in = 135] (s2) -- (br) [domain=\AJ:\AH] plot ({\RA*cos(\x)}, {\RA*sin(\x)},\LE);
\filldraw[ mydarkfill] (tl)-- (s1)to [out = -75, in =130] (e1) to [out= 45, in = -125]  (f1) to (q1) [domain=\AK:\AB] plot ({\RA*cos(\x)}, {\RA*sin(\x)},\LA);
\filldraw[mymedfill]  (q1)to (f1) to  [out= 65, in = -135]  (n1) 
to  [out= -45, in = 135]  (e2)to  [out= 45, in = -125]  (f2) to (q2)[domain=\AE:\AB] plot ({\RA*cos(\x)}, {\RA*sin(\x)},\LA);
\filldraw[mylightfill] (q2)to (f2) to [out= 65, in = -135]  (n2)
to [out= -45, in = 135]  (e3)to  [out= 45, in = -125] (f3) to (q3)[domain=\AH:\AE] plot ({\RA*cos(\x)}, {\RA*sin(\x)},\LA);
\filldraw[mydarkfill] (q3) to (f3) to [out= 65, in = -135]   (n3)
to  [out= -45, in = 135] (s2) --  (tr) [domain=\AK:\AH] plot ({\RA*cos(\x)}, {\RA*sin(\x)},\LA);

\draw[thick]  (s1)to [out = -75, in =130] (e1) to [out= 45, in = -125]  (f1)to  [out= 65, in = -135]  (n1) 
to  [out= -45, in = 135]  (e2)to  [out= 45, in = -125]  (f2)to [out= 65, in = -135]  (n2)
to [out= -45, in = 135]  (e3)to  [out= 45, in = -125] (f3) to [out= 65, in = -135]   (n3)
to  [out= -45, in = 135] (s2) ;

\draw[thick] (q1)--(f1)-- (p1);
\draw[thick] (q2)--(f2)-- (p2);
\draw[thick] (q3)--(f3)-- (p3);
\end{scope}

\end{tikzpicture}
} \caption{The multitrace}\label{ntrace}
\end{subfigure}
\hspace{.2cm}
\begin{subfigure}[t]{0.3\textwidth}
\resizebox{1\textwidth}{!}
{

\begin{tikzpicture}[tdplot_main_coords, scale =.5]
   \pgfmathsetmacro{\RA}{8}

   \pgfmathsetmacro{\LA}{8}
   \pgfmathsetmacro{\LBA}{0}
   \pgfmathsetmacro{\LBB}{3}
   \pgfmathsetmacro{\LBC}{6}
   \pgfmathsetmacro{\LCA}{-3}
   \pgfmathsetmacro{\LCB}{0}
   \pgfmathsetmacro{\LCC}{3}
   \pgfmathsetmacro{\LDA}{-6}
   \pgfmathsetmacro{\LDB}{-3}
   \pgfmathsetmacro{\LDC}{0}
   \pgfmathsetmacro{\LE}{-8}

\pgfmathsetmacro{\AA}{-25}
\pgfmathsetmacro{\AB}{\AA+15}
\pgfmathsetmacro{\AC}{\AA+25}
\pgfmathsetmacro{\AD}{\AA+45}
\pgfmathsetmacro{\AE}{\AA+55}
\pgfmathsetmacro{\AF}{\AA+65}
\pgfmathsetmacro{\AG}{\AA+85}
\pgfmathsetmacro{\AH}{\AA+100}
\pgfmathsetmacro{\AI}{\AA+115}
\pgfmathsetmacro{\AJ}{120}
\pgfmathsetmacro{\AK}{\AA-35}

   \filldraw [black!10!white,domain=0:360] plot ({\RA*cos(\x)}, {\RA*sin(\x)},\LA);

\coordinate (tl) at ({\RA*cos(\AK)}, {\RA*sin(\AK)}, \LA);
\coordinate (tr) at ({\RA*cos(\AJ)}, {\RA*sin(\AJ)}, \LA);
\coordinate (bl) at ({\RA*cos(\AK)}, {\RA*sin(\AK)}, \LE);
\coordinate (br) at ({\RA*cos(\AJ)}, {\RA*sin(\AJ)}, \LE);

\coordinate (s1) at  ({\RA*cos(\AK)}, {\RA*sin(\AK)}, \LCB);
\coordinate (s2) at  ({\RA*cos(\AJ)}, {\RA*sin(\AJ)}, \LCB);

\coordinate (n1) at  ({\RA*cos(\AC)}, {\RA*sin(\AC)}, \LBA);
\node at (n1)[circle, draw , fill =white] {$\eta_{M_1}$};
\coordinate (n2) at  ({\RA*cos(\AF)}, {\RA*sin(\AF)}, \LBB);
\node at (n2)[circle, draw , fill =white] {$\eta_{M_2}$};
\coordinate (n3) at  ({\RA*cos(\AI)}, {\RA*sin(\AI)}, \LBC);
\node at (n3)[circle, draw , fill =white] {$\eta_{M_3}$};

\coordinate (q1) at ({\RA*cos(\AB)}, {\RA*sin(\AB)}, \LA);
\node at (q1)[below left] {$Q_1$};
\coordinate (q2) at ({\RA*cos(\AE)}, {\RA*sin(\AE)}, \LA);
\node at (q2)[below left] {$Q_2$};
\coordinate (q3) at ({\RA*cos(\AH)}, {\RA*sin(\AH)}, \LA);
\node at (q3)[below left] {$Q_3$};

\coordinate (f1) at ({\RA*cos(\AB)}, {\RA*sin(\AB)}, \LCA);
\node at (f1)[circle, draw , fill =white] {$\phi_1$};
\coordinate (f2) at ({\RA*cos(\AE)}, {\RA*sin(\AE)}, \LCB);
\node at (f2)[circle, draw , fill =white] {$\phi_2$};
\coordinate (f3) at ({\RA*cos(\AH)}, {\RA*sin(\AH)}, \LCC);
\node at (f3)[circle, draw , fill =white] {$\phi_3$};

\coordinate (e1) at ({\RA*cos(\AA)}, {\RA*sin(\AA)}, \LDA);
\node at (e1)[circle, draw , fill =white] {$\epsilon_{M_3}$};
\coordinate (e2) at ({\RA*cos(\AD)}, {\RA*sin(\AD)}, \LDB);
\node at (e2)[circle, draw , fill =white] {$\epsilon_{M_1}$};
\coordinate (e3) at ({\RA*cos(\AG)}, {\RA*sin(\AG)}, \LDC);
\node at (e3)[circle, draw , fill =white] {$\epsilon_{M_2}$};

\coordinate (p1) at ({\RA*cos(\AB)}, {\RA*sin(\AB)}, \LE);
\node at (p1)[above right] {$P_1$};
\coordinate (p2) at  ({\RA*cos(\AE)}, {\RA*sin(\AE)}, \LE);
\node at (p2)[above right] {$P_2$};
\coordinate (p3) at  ({\RA*cos(\AH)}, {\RA*sin(\AH)}, \LE);
\node at (p3)[above right] {$P_3$};

 \begin{scope}[on background layer]
\filldraw[pattern =north west lines, pattern color = mydarkred] (bl)-- (s1)to [out = -75, in =130] (e1) to [out= 45, in = -125]  (f1)--(p1) [domain=\AB:\AK] plot ({\RA*cos(\x)}, {\RA*sin(\x)},\LE);
\filldraw[pattern =north west lines, pattern color = mydarkgreen] (p1)--(f1)to  [out= 65, in = -135]  (n1) 
to  [out= -45, in = 135]  (e2)to  [out= 45, in = -125]  (f2)--(p2)[domain=\AE:\AB] plot ({\RA*cos(\x)}, {\RA*sin(\x)},\LE);
\filldraw[pattern =north west lines, pattern color = mydarkblue](p2)--(f2)to [out= 65, in = -135]  (n2)
to [out= -45, in = 135]  (e3)to  [out= 45, in = -125] (f3) --(p3)[domain=\AH:\AE] plot ({\RA*cos(\x)}, {\RA*sin(\x)},\LE);
\filldraw[pattern =north west lines, pattern color = mydarkred](p3)--(f3)to [out= 65, in = -135]   (n3)
to  [out= -45, in = 115] (s2) -- (br) [domain=\AJ:\AH] plot ({\RA*cos(\x)}, {\RA*sin(\x)},\LE);
\filldraw[mydarkfill] (tl)-- (s1)to [out = -75, in =130] (e1) to [out= 45, in = -125]  (f1)--(q1)
[domain=\AB:\AJ] plot ({\RA*cos(\x)}, {\RA*sin(\x)},\LA);
\filldraw[mymedfill] (q1)--(f1)to  [out= 65, in = -135]  (n1) 
to  [out= -45, in = 135]  (e2)to  [out= 45, in = -125]  (f2)--(q2)[domain=\AE:\AB] plot ({\RA*cos(\x)}, {\RA*sin(\x)},\LA);
\filldraw[mylightfill](q2)--(f2)to [out= 65, in = -135]  (n2)
to [out= -45, in = 135]  (e3)to  [out= 45, in = -125] (f3)--(q3)[domain=\AH:\AE] plot ({\RA*cos(\x)}, {\RA*sin(\x)},\LA);
\filldraw[mydarkfill](q3)--(f3) to [out= 65, in = -135]   (n3)
to  [out= -45, in = 115] (s2) --  (tr) [domain=\AK:\AH] plot ({\RA*cos(\x)}, {\RA*sin(\x)},\LA);

\draw[thick]  (s1)to [out = -75, in =130] (e1) to [out= 45, in = -125]  (f1)to  [out= 65, in = -135]  (n1) 
to  [out= -45, in = 135]  (e2)to  [out= 45, in = -125]  (f2)to [out= 65, in = -135]  (n2)
to [out= -45, in = 135]  (e3)to  [out= 45, in = -125] (f3) to [out= 65, in = -135]   (n3)
to  [out= -45, in = 115] (s2) ;

\draw[thick] (q1)--(f1)-- (p1);
\draw[thick] (q2)--(f2)-- (p2);
\draw[thick] (q3)--(f3)-- (p3);
\end{scope}

\end{tikzpicture}
}
\caption{Sliding morphisms }
\end{subfigure}
\hspace{.2cm}
\begin{subfigure}[t]{0.3\textwidth}
\resizebox{1\textwidth}{!}
{\begin{tikzpicture}[tdplot_main_coords, scale =.5]
   \pgfmathsetmacro{\RA}{8}

   \pgfmathsetmacro{\LA}{8}
   \pgfmathsetmacro{\LBA}{0}
   \pgfmathsetmacro{\LBB}{3}
   \pgfmathsetmacro{\LBC}{6}
   \pgfmathsetmacro{\LCA}{-3}
   \pgfmathsetmacro{\LCB}{0}
   \pgfmathsetmacro{\LCC}{3}
   \pgfmathsetmacro{\LDA}{-6}
   \pgfmathsetmacro{\LDB}{-3}
   \pgfmathsetmacro{\LDC}{0}
   \pgfmathsetmacro{\LE}{-8}

\pgfmathsetmacro{\AA}{-25}
\pgfmathsetmacro{\AB}{\AA+15}
\pgfmathsetmacro{\AC}{\AA+25}
\pgfmathsetmacro{\AD}{\AA+45}
\pgfmathsetmacro{\AE}{\AA+55}
\pgfmathsetmacro{\AF}{\AA+65}
\pgfmathsetmacro{\AG}{\AA+85}
\pgfmathsetmacro{\AH}{\AA+100}
\pgfmathsetmacro{\AI}{\AA+115}
\pgfmathsetmacro{\AJ}{120}
\pgfmathsetmacro{\AK}{\AA-35}

   \filldraw [black!10!white,domain=0:360] plot ({\RA*cos(\x)}, {\RA*sin(\x)},\LA);

\coordinate (tl) at ({\RA*cos(\AK)}, {\RA*sin(\AK)}, \LA);
\coordinate (tr) at ({\RA*cos(\AJ)}, {\RA*sin(\AJ)}, \LA);
\coordinate (bl) at ({\RA*cos(\AK)}, {\RA*sin(\AK)}, \LE);
\coordinate (br) at ({\RA*cos(\AJ)}, {\RA*sin(\AJ)}, \LE);

\coordinate (s1) at  ({\RA*cos(\AK)}, {\RA*sin(\AK)}, \LCB);
\coordinate (s2) at  ({\RA*cos(\AJ)}, {\RA*sin(\AJ)}, \LCB);

\coordinate (n1) at  ({\RA*cos(\AC)}, {\RA*sin(\AC)}, \LBA);
\coordinate (n2) at  ({\RA*cos(\AF)}, {\RA*sin(\AF)}, \LBB);
\coordinate (n3) at  ({\RA*cos(\AI)}, {\RA*sin(\AI)}, \LBC);
\node at (n3)[circle, draw , fill =white] {$\eta_{M_3}$};

\coordinate (q1) at ({\RA*cos(\AB)}, {\RA*sin(\AB)}, \LA);
\node at (q1)[below left] {$Q_1$};
\coordinate (q2) at ({\RA*cos(\AE)}, {\RA*sin(\AE)}, \LA);
\node at (q2)[below left] {$Q_2$};
\coordinate (q3) at ({\RA*cos(\AH)}, {\RA*sin(\AH)}, \LA);
\node at (q3)[below left] {$Q_3$};

\coordinate (f1) at ({\RA*cos(\AB)}, {\RA*sin(\AB)}, \LCA);
\node at (f1)[circle, draw , fill =white] {$\phi_1$};
\coordinate (f2) at ({\RA*cos(\AE)}, {\RA*sin(\AE)}, \LCB);
\node at (f2)[circle, draw , fill =white] {$\phi_2$};
\coordinate (f3) at ({\RA*cos(\AH)}, {\RA*sin(\AH)}, \LCC);
\node at (f3)[circle, draw , fill =white] {$\phi_3$};

\coordinate (e1) at ({\RA*cos(\AA)}, {\RA*sin(\AA)}, \LDA);
\node at (e1)[circle, draw , fill =white] {$\epsilon_{M_3}$};
\coordinate (e2) at ({\RA*cos(\AD)}, {\RA*sin(\AD)}, \LDB);
\coordinate (e3) at ({\RA*cos(\AG)}, {\RA*sin(\AG)}, \LDC);

\coordinate (p1) at ({\RA*cos(\AB)}, {\RA*sin(\AB)}, \LE);
\node at (p1)[above right] {$P_1$};
\coordinate (p2) at  ({\RA*cos(\AE)}, {\RA*sin(\AE)}, \LE);
\node at (p2)[above right] {$P_2$};
\coordinate (p3) at  ({\RA*cos(\AH)}, {\RA*sin(\AH)}, \LE);
\node at (p3)[above right] {$P_3$};

 \begin{scope}[on background layer]
\filldraw[pattern =north west lines, pattern color = mydarkred] (bl)-- (s1)to [out = -75, in =130] (e1) to [out= 45, in = -125] (f1)--(p1)[domain=\AB:\AK] plot ({\RA*cos(\x)}, {\RA*sin(\x)},\LE);
\filldraw[pattern =north west lines, pattern color = mydarkgreen] (p1)--(f1)to  [out= 65, in = -135]  (f2)--(p2)[domain=\AE:\AB] plot ({\RA*cos(\x)}, {\RA*sin(\x)},\LE);
\filldraw[pattern =north west lines, pattern color = mydarkblue](p2)--(f2)to [out= 65, in = -135] (f3)--(p3)[domain=\AH:\AE] plot ({\RA*cos(\x)}, {\RA*sin(\x)},\LE);
\filldraw[pattern =north west lines, pattern color = mydarkred](p3)--(f3) to [out= 65, in = -135]   (n3)to  [out= -45, in = 115] (s2) -- (br)[domain=\AJ:\AH] plot ({\RA*cos(\x)}, {\RA*sin(\x)},\LE);
\filldraw[mydarkfill] (tl)-- (s1)to [out = -75, in =130] (e1) to [out= 45, in = -125]  
(f1)--(q1)[domain=\AB:\AK] plot ({\RA*cos(\x)}, {\RA*sin(\x)},\LA);
\filldraw[mymedfill](q1)--(f1)to  [out= 65, in = -135]  
(f2)--(q2)[domain=\AE:\AB] plot ({\RA*cos(\x)}, {\RA*sin(\x)},\LA);
\filldraw[mylightfill](q2)--(f2)to [out= 65, in = -135] 
(f3)--(q3)[domain=\AH:\AE] plot ({\RA*cos(\x)}, {\RA*sin(\x)},\LA);
\filldraw[mydarkfill](q3)--(f3) to [out= 65, in = -135]   (n3) to  [out= -45, in = 115] 
(s2) --  (tr) [domain=\AJ:\AH] plot ({\RA*cos(\x)}, {\RA*sin(\x)},\LA);

\draw[thick]  (s1)to [out = -75, in =130] (e1) to [out= 45, in = -125]  (f1)to  [out= 65, in = -135]  
(f2)to [out= 65, in = -135] 
(f3) to [out= 65, in = -135]   (n3)
to  [out= -45, in = 115] (s2) ;

\draw[thick] (q1)--(f1)-- (p1);
\draw[thick] (q2)--(f2)-- (p2);
\draw[thick] (q3)--(f3)-- (p3);
\end{scope}

\end{tikzpicture}
}
\caption{Applying triangle identities}
\end{subfigure}

\caption{The multitrace is isomorphic to the trace of the composition.}
\label{Comparison_n_traces}
\end{figure}

Now we turn to the rest of the proof of \cref{thm:main_first_half}. As discussed in \cref{sec:smc}, we need to break free of the bicategory structure on $\sB$ and use some additional structure that can reorder tensored objects. Here we give an axiomatic description of this extra structure and use it to prove \cref{thm:main_first_half}.  The existence of examples of this structure (other than symmetric monoidal categories) is established in \cite{mp2}.

To motivate the following definitions, it is useful to think of the trace of 
 $\ful fn$ 
as $n$ nested circles, with an extra twist owing to the fact that $\ful fn$ rotates the factors around.  See 
\cref{fig:multitrace_cartoon}.
If we re-interpret this picture as a single circle winding around $n$ times, we get precisely the multitrace pictured in \cref{ntrace}. So we just need to formally understand the process of ``unwinding the coil'' in \cref{fig:multitrace_cartoon}, in other words lifting it to the $n$-fold cover of the circle.
\begin{figure}
\centering

\hspace{1cm}
\tdplotsetmaincoords{65}{100}
\centering
{
\begin{tikzpicture}[tdplot_main_coords, scale =.5,roundnode/.style={circle, fill=black!60, inner sep=0pt, minimum size=2mm}]
   \pgfmathsetmacro{\RA}{3}
   \pgfmathsetmacro{\RB}{5}
   \pgfmathsetmacro{\RC}{7}

   \pgfmathsetmacro{\vertshift}{7}
   \pgfmathsetmacro{\LA}{9}
   \pgfmathsetmacro{\LBA}{\LA-\vertshift}
   \pgfmathsetmacro{\LCA}{\LA-2*\vertshift}
   \pgfmathsetmacro{\LD}{\LA-3*\vertshift}
   \pgfmathsetmacro{\LE}{\LA-4*\vertshift}

\pgfmathsetmacro{\AA}{-45}
\pgfmathsetmacro{\AB}{\AA-30}
\pgfmathsetmacro{\AC}{\AA+20}
\pgfmathsetmacro{\AD}{\AA+15}
\pgfmathsetmacro{\AE}{\AA+55}
\pgfmathsetmacro{\AEA}{\AA+95}
\pgfmathsetmacro{\AF}{\AA+90}
\pgfmathsetmacro{\AG}{\AA+140}
\pgfmathsetmacro{\AH}{\AA+100}
\pgfmathsetmacro{\AI}{\AA+115}
\pgfmathsetmacro{\AJ}{120}
\pgfmathsetmacro{\AK}{\AA-35}
\pgfmathsetmacro{\Jct}{150}
\pgfmathsetmacro{\Jcta}{210}

   \draw [very thick, mydarkfill,domain=-\Jct:\AE] plot ({\RA*cos(\x)}, {\RA*sin(\x)},\LA);
   \draw [very thick, mymedfill,domain=\AE:\Jct] plot ({\RA*cos(\x)}, {\RA*sin(\x)},\LA);
   \draw [very thick, mymedfill,domain=-\Jct:\AE] plot ({\RB*cos(\x)}, {\RB*sin(\x)},\LA);
   \draw [very thick, mylightfill,domain=\AE:\Jct] plot ({\RB*cos(\x)}, {\RB*sin(\x)},\LA);
   \draw [very thick, mylightfill,domain=-\Jct:\AE] plot ({\RC*cos(\x)}, {\RC*sin(\x)},\LA);
   \draw [very thick, mydarkfill,domain=\AE:\Jct] plot ({\RC*cos(\x)}, {\RC*sin(\x)},\LA);
   \draw [very thick, mydarkfill] ({\RA*cos(-\Jct)}, {\RA*sin(-\Jct)},\LA) to  ({\RC*cos(-\Jcta)}, {\RC*sin(-\Jcta)},\LA);
   \draw [very thick, mymedfill] ({\RB*cos(-\Jct)}, {\RB*sin(-\Jct)},\LA) to  ({\RA*cos(-\Jcta)}, {\RA*sin(-\Jcta)},\LA);
   \draw [very thick, mylightfill] ({\RC*cos(-\Jct)}, {\RC*sin(-\Jct)},\LA) to  ({\RB*cos(-\Jcta)}, {\RB*sin(-\Jcta)},\LA);

	\coordinate (q1) at ({\RA*cos(\AE)}, {\RA*sin(\AE)},\LA); 
		\node[roundnode] at (q1){};
		\node at ([shift={(-.65,.65)}]q1) {$Q_1$};
	\coordinate (q1b) at ({\RB*cos(\AE)}, {\RB*sin(\AE)},\LA); 
		\node[roundnode] at (q1b){};
		\node at ([shift={(-.65,.65)}]q1b) {$Q_2$};
	\coordinate (q1c) at ({\RC*cos(\AE)}, {\RC*sin(\AE)},\LA); 
		\node[roundnode] at (q1c){};
		\node at ([shift={(-.65,.65)}]q1c){$Q_3$};

   \draw [very thick, mydarkfill,domain=-\Jct:\AE] plot ({\RA*cos(\x)}, {\RA*sin(\x)},\LBA);
   \draw [very thick, mymedfill,domain=\AE:\AEA] plot ({\RA*cos(\x)}, {\RA*sin(\x)},\LBA);
   \draw [very thick, dotted, mydarkgreen,domain=\AEA:\AG] plot ({\RA*cos(\x)}, {\RA*sin(\x)},\LBA);
   \draw [very thick, mymedfill,domain=\AG:\Jct] plot ({\RA*cos(\x)}, {\RA*sin(\x)},\LBA);
   \draw [very thick, mymedfill,domain=-\Jct:\AE] plot ({\RB*cos(\x)}, {\RB*sin(\x)},\LBA);
   \draw [very thick, mylightfill,domain=\AE:\AEA] plot ({\RB*cos(\x)}, {\RB*sin(\x)},\LBA);
   \draw [very thick, dotted, mydarkblue,domain=\AEA:\AG] plot ({\RB*cos(\x)}, {\RB*sin(\x)},\LBA);
   \draw [very thick, mylightfill,domain=\AG:\Jct] plot ({\RB*cos(\x)}, {\RB*sin(\x)},\LBA);
   \draw [very thick, mylightfill,domain=-\Jct:\AE] plot ({\RC*cos(\x)}, {\RC*sin(\x)},\LBA);
   \draw [very thick, mydarkfill,domain=\AE:\AEA] plot ({\RC*cos(\x)}, {\RC*sin(\x)},\LBA);
   \draw [very thick,  dotted,mydarkred,domain=\AEA:\AG] plot ({\RC*cos(\x)}, {\RC*sin(\x)},\LBA);
   \draw [very thick, mydarkfill,domain=\AG:\Jct] plot ({\RC*cos(\x)}, {\RC*sin(\x)},\LBA);
   \draw [very thick, mydarkfill] ({\RA*cos(-\Jct)}, {\RA*sin(-\Jct)},\LBA) to  ({\RC*cos(-\Jcta)}, {\RC*sin(-\Jcta)},\LBA);
   \draw [very thick, mymedfill] ({\RB*cos(-\Jct)}, {\RB*sin(-\Jct)},\LBA) to  ({\RA*cos(-\Jcta)}, {\RA*sin(-\Jcta)},\LBA);
   \draw [very thick, mylightfill] ({\RC*cos(-\Jct)}, {\RC*sin(-\Jct)},\LBA) to  ({\RB*cos(-\Jcta)}, {\RB*sin(-\Jcta)},\LBA);

	\coordinate (q2) at  ({\RA*cos(\AE)}, {\RA*sin(\AE)},\LBA){}; 
		\node[roundnode] at (q2){};
		\node at ([shift={(-.65,.65)}]q2) {$Q_1$};
	\coordinate (q2b) at  ({\RB*cos(\AE)}, {\RB*sin(\AE)},\LBA){}; 
		\node[roundnode] at (q2b){};
		\node at ([shift={(-.65,.65)}]q2b) {$Q_2$};
	\coordinate (q2c) at  ({\RC*cos(\AE)}, {\RC*sin(\AE)},\LBA){}; 
		\node[roundnode] at (q2c){};
		\node at ([shift={(-.65,.65)}]q2c){$Q_3$};
	\coordinate (m1) at ({\RA*cos(\AEA)}, {\RA*sin(\AEA)},\LBA){}; 
		\node[roundnode] at (m1){};
		\node at ([shift={(-.65,.65)}]m1){$M_1$};
	\coordinate (m1b) at ({\RB*cos(\AEA)}, {\RB*sin(\AEA)},\LBA){}; 
		\node[roundnode] at (m1b){};
		\node at ([shift={(-.65,.65)}]m1b){$M_2$};
	\coordinate (m1c) at ({\RC*cos(\AEA)}, {\RC*sin(\AEA)},\LBA){}; 
		\node[roundnode] at (m1c){};
		\node at ([shift={(-.65,.65)}]m1c) {$M_3$};
	\coordinate (dm1) at ({\RA*cos(\AG)}, {\RA*sin(\AG)},\LBA){};
		\node[roundnode] at (dm1){};
		\node at ([shift={(-.65,.65)}]dm1) {$M^*_1$};
	\coordinate (dm1b) at ({\RB*cos(\AG)}, {\RB*sin(\AG)},\LBA){};
		\node[roundnode] at (dm1b){};
		\node at ([shift={(-.65,.65)}]dm1b){$M^*_2$};
	\coordinate (dm1c) at ({\RC*cos(\AG)}, {\RC*sin(\AG)},\LBA){};
		\node[roundnode] at (dm1c){};
		\node at ([shift={(-.65,.65)}]dm1c){$M^*_3$};

   \draw [very thick, mydarkfill,domain=-\Jct:\AD] plot ({\RA*cos(\x)}, {\RA*sin(\x)},\LCA);
   \draw [very thick,  dotted, mydarkred,domain=\AE:\AD] plot ({\RA*cos(\x)}, {\RA*sin(\x)},\LCA);
   \draw [very thick, dotted, mydarkgreen,domain=\AE:\AG] plot ({\RA*cos(\x)}, {\RA*sin(\x)},\LCA);
   \draw [very thick, mymedfill,domain=\AG:\Jct] plot ({\RA*cos(\x)}, {\RA*sin(\x)},\LCA);
   \draw [very thick, mymedfill,domain=-\Jct:\AD] plot ({\RB*cos(\x)}, {\RB*sin(\x)},\LCA);
   \draw [very thick,  dotted,mydarkgreen,domain=\AE:\AD] plot ({\RB*cos(\x)}, {\RB*sin(\x)},\LCA);
   \draw [very thick,  dotted,mydarkblue,domain=\AE:\AG] plot ({\RB*cos(\x)}, {\RB*sin(\x)},\LCA);
   \draw [very thick, mylightfill,domain=\AG:\Jct] plot ({\RB*cos(\x)}, {\RB*sin(\x)},\LCA);
   \draw [very thick, mylightfill,domain=-\Jct:\AD] plot ({\RC*cos(\x)}, {\RC*sin(\x)},\LCA);
   \draw [very thick, dotted, mydarkblue,domain=\AD:\AE] plot ({\RC*cos(\x)}, {\RC*sin(\x)},\LCA);
  \draw [very thick,  dotted,mydarkred,domain=\AE:\AG] plot ({\RC*cos(\x)}, {\RC*sin(\x)},\LCA);
   \draw [very thick, mydarkfill,domain=\AG:\Jct] plot ({\RC*cos(\x)}, {\RC*sin(\x)},\LCA);
   \draw [very thick, mydarkfill] ({\RA*cos(-\Jct)}, {\RA*sin(-\Jct)},\LCA) to  ({\RC*cos(-\Jcta)}, {\RC*sin(-\Jcta)},\LCA);
   \draw [very thick, mymedfill] ({\RB*cos(-\Jct)}, {\RB*sin(-\Jct)},\LCA) to  ({\RA*cos(-\Jcta)}, {\RA*sin(-\Jcta)},\LCA);
   \draw [very thick, mylightfill] ({\RC*cos(-\Jct)}, {\RC*sin(-\Jct)},\LCA) to  ({\RB*cos(-\Jcta)}, {\RB*sin(-\Jcta)},\LCA);
	\coordinate (p1) at  ({\RA*cos(\AE)}, {\RA*sin(\AE)},\LCA){}; 
		\node[roundnode] at (p1){};
		\node at ([shift={(-.65,.65)}]p1) {$P_1$};
	\coordinate (p1b) at  ({\RB*cos(\AE)}, {\RB*sin(\AE)},\LCA){}; 
		\node[roundnode] at (p1b){};
		\node at ([shift={(-.65,.65)}]p1b) {$P_2$};
	\coordinate (p1c) at  ({\RC*cos(\AE)}, {\RC*sin(\AE)},\LCA){}; 
		\node[roundnode] at (p1c){};
		\node at ([shift={(-.65,.65)}]p1c){$P_3$};
	\coordinate (m2) at ({\RA*cos(\AD)}, {\RA*sin(\AD)},\LCA){}; 
		\node[roundnode] at (m2){};
		\node at ([shift={(-.65,.65)}]m2) {$M_3$};
	\coordinate (m2b) at ({\RB*cos(\AD)}, {\RB*sin(\AD)},\LCA){}; 
		\node[roundnode] at (m2b){};
		\node at ([shift={(-.65,.65)}]m2b){$M_1$};
	\coordinate (m2c) at ({\RC*cos(\AD)}, {\RC*sin(\AD)},\LCA){}; 
		\node[roundnode] at (m2c){};
		\node at ([shift={(-.65,.65)}]m2c) {$M_2$};
	\coordinate (dm2) at ({\RA*cos(\AG)}, {\RA*sin(\AG)},\LCA){};
		\node[roundnode] at (dm2){};
		\node at ([shift={(-.65,.65)}]dm2){$M^*_1$};
	\coordinate (dm2b) at ({\RB*cos(\AG)}, {\RB*sin(\AG)},\LCA){};
		\node[roundnode] at (dm2b){};
		\node at ([shift={(-.65,.65)}]dm2b) {$M^*_2$};
	\coordinate (dm2c) at ({\RC*cos(\AG)}, {\RC*sin(\AG)},\LCA){};
		\node[roundnode] at (dm2c){};
		\node at ([shift={(-.65,.65)}]dm2c) {$M^*_3$};

   \draw [very thick, mydarkfill,domain=\AB:\AD] plot ({\RA*cos(\x)}, {\RA*sin(\x)},\LD);
   \draw [very thick,  dotted,mydarkred,domain=\AE:\AD] plot ({\RA*cos(\x)}, {\RA*sin(\x)},\LD);
   \draw [very thick, dotted, mydarkgreen,domain=\AE:\Jct] plot ({\RA*cos(\x)}, {\RA*sin(\x)},\LD);
   \draw [very thick, dotted, mydarkred,domain=\AB:-\Jct] plot ({\RA*cos(\x)}, {\RA*sin(\x)},\LD);
   \draw [very thick, mymedfill,domain=\AB:\AD] plot ({\RB*cos(\x)}, {\RB*sin(\x)},\LD);
   \draw [very thick, dotted, mydarkgreen,domain=\AE:\AD] plot ({\RB*cos(\x)}, {\RB*sin(\x)},\LD);
   \draw [very thick, dotted, mydarkblue,domain=\AE:\Jct] plot ({\RB*cos(\x)}, {\RB*sin(\x)},\LD);
   \draw [very thick, dotted, mydarkgreen,domain=\AB:-\Jct] plot ({\RB*cos(\x)}, {\RB*sin(\x)},\LD);
   \draw [very thick, mylightfill,domain=\AB:\AD] plot ({\RC*cos(\x)}, {\RC*sin(\x)},\LD);
   \draw [very thick, dotted, mydarkblue,domain=\AD:\AE] plot ({\RC*cos(\x)}, {\RC*sin(\x)},\LD);
  \draw [very thick,  dotted,mydarkred,domain=\AE:\Jct] plot ({\RC*cos(\x)}, {\RC*sin(\x)},\LD);
   \draw [very thick, dotted, mydarkblue,domain=\AB:-\Jct] plot ({\RC*cos(\x)}, {\RC*sin(\x)},\LD);
   \draw [very thick, dotted, mydarkred] ({\RA*cos(-\Jct)}, {\RA*sin(-\Jct)},\LD) to  ({\RC*cos(-\Jcta)}, {\RC*sin(-\Jcta)},\LD);
   \draw [very thick, dotted, mydarkgreen] ({\RB*cos(-\Jct)}, {\RB*sin(-\Jct)},\LD) to  ({\RA*cos(-\Jcta)}, {\RA*sin(-\Jcta)},\LD);
   \draw [very thick, dotted, mydarkblue] ({\RC*cos(-\Jct)}, {\RC*sin(-\Jct)},\LD) to  ({\RB*cos(-\Jcta)}, {\RB*sin(-\Jcta)},\LD);

	\coordinate (p1) at  ({\RA*cos(\AE)}, {\RA*sin(\AE)},\LD){}; 
		\node[roundnode] at (p1){};
		\node at ([shift={(-.65,.65)}]p1){$P_1$};
	\coordinate (p1b) at  ({\RB*cos(\AE)}, {\RB*sin(\AE)},\LD){}; 
		\node[roundnode] at (p1b){};
		\node at ([shift={(-.65,.65)}]p1b) {$P_2$};
	\coordinate (p1c) at  ({\RC*cos(\AE)}, {\RC*sin(\AE)},\LD){}; 
		\node[roundnode] at (p1c){};
		\node at ([shift={(-.65,.65)}]p1c){$P_3$};
	\coordinate (m2) at ({\RA*cos(\AD)}, {\RA*sin(\AD)},\LD){}; 
		\node[roundnode] at (m2){};
		\node at ([shift={(-.65,.65)}]m2) {$M_3$};
	\coordinate (m2b) at ({\RB*cos(\AD)}, {\RB*sin(\AD)},\LD){}; 
		\node[roundnode] at (m2b){};
		\node at ([shift={(-.65,.65)}]m2b) {$M_1$};
	\coordinate (m2c) at ({\RC*cos(\AD)}, {\RC*sin(\AD)},\LD){}; 
		\node[roundnode] at (m2c){};
		\node at ([shift={(-.65,.65)}]m2c){$M_2$};
	\coordinate (dm2) at ({\RA*cos(\AB)}, {\RA*sin(\AB)},\LD){};
		\node[roundnode] at (dm2){};
		\node at ([shift={(-.65,.65)}]dm2){$M^*_3$};
	\coordinate (dm2b) at ({\RB*cos(\AB)}, {\RB*sin(\AB)},\LD){};
		\node[roundnode] at (dm2b){};
		\node at ([shift={(-.65,.65)}]dm2b){$M^*_1$};
	\coordinate (dm2c) at ({\RC*cos(\AB)}, {\RC*sin(\AB)},\LD){};
		\node[roundnode] at (dm2c){};
		\node at ([shift={(-.65,.65)}]dm2c){$M^*_2$};

   \draw [very thick, dotted, mydarkred,domain=-\Jct:\AE] plot ({\RA*cos(\x)}, {\RA*sin(\x)},\LE);
   \draw [very thick, dotted, mydarkgreen,domain=\AE:\Jct] plot ({\RA*cos(\x)}, {\RA*sin(\x)},\LE);
   \draw [very thick, dotted, mydarkgreen,domain=-\Jct:\AE] plot ({\RB*cos(\x)}, {\RB*sin(\x)},\LE);
   \draw [very thick, dotted, mydarkblue,domain=\AE:\Jct] plot ({\RB*cos(\x)}, {\RB*sin(\x)},\LE);
   \draw [very thick, dotted, mydarkblue,domain=-\Jct:\AE] plot ({\RC*cos(\x)}, {\RC*sin(\x)},\LE);
   \draw [very thick, dotted, mydarkred,domain=\AE:\Jct] plot ({\RC*cos(\x)}, {\RC*sin(\x)},\LE);
   \draw [very thick, dotted, mydarkred] ({\RA*cos(-\Jct)}, {\RA*sin(-\Jct)},\LE) to  ({\RC*cos(-\Jcta)}, {\RC*sin(-\Jcta)},\LE);
   \draw [very thick, dotted, mydarkgreen] ({\RB*cos(-\Jct)}, {\RB*sin(-\Jct)},\LE) to  ({\RA*cos(-\Jcta)}, {\RA*sin(-\Jcta)},\LE);
   \draw [very thick, dotted, mydarkblue] ({\RC*cos(-\Jct)}, {\RC*sin(-\Jct)},\LE) to  ({\RB*cos(-\Jcta)}, {\RB*sin(-\Jcta)},\LE);

\coordinate (p2) at  ({\RA*cos(\AE)}, {\RA*sin(\AE)}, \LE);
\node[roundnode] at (p2){};
\node at ([shift={(-.65,.65)}]p2) {$P_1$};
\coordinate (p2b) at  ({\RB*cos(\AE)}, {\RB*sin(\AE)}, \LE);
\node[roundnode] at (p2b){};
\node at ([shift={(-.65,.65)}]p2b) {$P_2$};
\coordinate (p2c) at  ({\RC*cos(\AE)}, {\RC*sin(\AE)}, \LE);
\node[roundnode] at (p2c){};
\node at ([shift={(-.65,.65)}]p2c){$P_3$};

\end{tikzpicture}
} \caption{The trace of $\ful f3$}\label{fig:multitrace_cartoon}

\end{figure}

A {\bf shadowed $n$-Fuller structure} on a bicategory with shadow $\sB$ consists of the following.
\begin{enumerate}
	\renewcommand{\theenumi}{\arabic{enumi}}
	\item\label{item:defn_boxtimes} A strong functor (pseudofunctor) of bicategories
	\[\boxtimes \colon \underbrace{\sB\times \ldots \times \sB}_n\to \sB. \]
	Here $\sB\times \ldots\times \sB$ is the bicategory whose 0-cells are tuples of 0-cells $\sB$ and
	\[(\sB\times \ldots\times \sB)((A_1,\ldots,A_n),(B_1,\ldots,B_n))\coloneqq \sB(A_1,B_1)\times \ldots \times \sB(A_n,B_n). \]
	The product, shadow, associator, and so on are all defined componentwise. 

	More explicitly, this is a function that assigns a 0-cell $\boxtimes A_i$ to every tuple of 0-cells $A_i$, functors 
	$\boxtimes\colon \prod \sB(A_i,B_i) \to \sB(\boxtimes A_i,\boxtimes B_i)$, and natural isomorphisms
	\begin{align*}
		m_{\boxtimes}\colon \left(\boxtimes M_i\right) \odot \left(\boxtimes N_i\right) &\cong \boxtimes (M_i \odot N_i) \\
		i_{\boxtimes}\colon U_{\boxtimes A_i} &\cong \boxtimes U_{A_i}
	\end{align*}
	satisfying the same coherence axioms as for a monoidal functor including
\eqref{eq:i_m},  \eqref{eq:m_box_naturality}, \eqref{eq:m_box_assoc}.

	\item\label{item:defn_vartheta} A pseudonatural transformation
	\[\vartheta\colon \boxtimes \circ \gamma \to \boxtimes\]
	where $\gamma$ is the strong functor $\sB\times \ldots \times \sB\to \sB\times \ldots \times \sB$ that permutes the leftmost $\sB$ to the right. 

	More explicitly,   for each $n$ tuple of objects $(A_1,\ldots ,A_n)$ in $\sB$ there is an object 
	$T_{A_i}\in \sB(  A_{2}\times \ldots \times A_n\times A_1, A_1\times \ldots \times A_n)$
	and natural  isomorphisms
	\[ \vartheta\colon T_{A_i}\odot (\boxtimes M_{i}) \overset\cong\to (\boxtimes M_{i+1})\odot T_{B_i}\]
	for all $M_i\in \sB(A_i,B_i)$ that are compatible with $m_{\boxtimes}$ and $i_{\boxtimes}$.\footnote{In fact, the compatibility with $i_\boxtimes$ is entirely optional, because our arguments below do not use it.}

	\item\label{item:defn_tau_fuller} A natural isomorphism
	\[ \tau\colon \sh{T_{A_{i-1}},\boxtimes Q_i} \overset\cong\to \sh{Q_1,\ldots,Q_n} \] 
	so that 
	\begin{equation}\label{eq:fuller_defn_diagram}\xymatrix@R=10pt{\sh{T_{A_{i-1}}, \boxtimes M_{i}, \boxtimes N_i }\ar[r]_-\sim\ar[d]^-\vartheta_-\sim&
		\sh{T_{A_{i-1}}, \boxtimes ( M_{i}\odot N_i) }\ar[r]^-\tau_-\sim
		&\sh{ M_1, N_1, M_2,\ldots ,M_n, N_n}
		\ar[dd]_-\sim^\theta
		\\
		\sh{\boxtimes M_{i+1}, T_{B_{i}}, \boxtimes N_i}\ar[d]^\theta
		\\
		\sh{T_{B_{i}}, \boxtimes N_i,\boxtimes M_{i+1}}\ar[r]_-\sim
		&\sh{T_{B_{i}}, \boxtimes (N_i \odot M_{i+1})}\ar[r]^-\tau_-\sim
		&\sh{ N_1, M_2, \ldots, M_n, N_n, M_1}
	}\end{equation}
	commutes for all $M_i\in \sB(A_{i-1},B_i)$ and $N_i\in \sB(B_i,A_i)$. 
\end{enumerate}

\begin{example}
	If $\sC$ is a symmetric monoidal category, it has a canonical $n$-Fuller structure in which $\boxtimes$ is the $n$-fold tensor product, $T_{A_i}$ is the unit, and the rest of the isomorphisms are the canonical ones that come from the coherence theorem for symmetric monoidal categories.
\end{example}

\begin{example}
	The bicategory $\Ex$ has a shadowed $n$-Fuller structure. This can be deduced from our foundational work in \cref{thm:spectra_smbf} below, and the formal work from \cite{mp2} summarized in \cref{thm:indexed_bicat} below.
\end{example}

The following statement is an immediate consequence of \cref{bicategory_map_preserves_traces}.

\begin{lem}\label{lem:external_tensor_of_duals}
	If $M_i\in \sB (A_i,B_i)$ are right dualizable with duals $N_i\in \sB (B_i,A_i)$
	and \[\boxtimes \colon \sB\times \ldots \times \sB\to \sB\] is a
	strong functor of bicategories 
	then  $\boxtimes M_i\in \sB (\prod A_i,\prod B_i)$ is dualizable with dual $\boxtimes N_i$. 
\end{lem}

For dualizable $M_i\in \sB (A_i,B_i)$  and $Q_i\in \sB (A_{i-1}, A_i)$ and $P_i\in \sB (B_{i-1}, B_i)$ (subscripts are taken mod $n$)  the {\bf abstract Fuller map} of
\[ \phi_i\colon Q_i\odots{A_i} M_i\to M_{i-1}\odots{B_{i-1}} P_i,\]
denoted   $\ful{\phi_1,\ldots ,\phi_n}{} \in \sB (\prod A_i, \prod B_i)$, is the composite 
\[\xymatrix@R=10pt{T_{A_{i-1}}\odot \boxtimes Q_i\odot \boxtimes M_i\ar[d]^{\id\odot m_{\boxtimes}}
	& T_{A_{i-1}}\odot \boxtimes M_{i-1}\odot \boxtimes P_i\ar[r]^-{\vartheta\odot \id}_-\sim\ar[d]^{\id\odot m_\boxtimes}
	& \boxtimes M_{i}\odot T_{B_{i-1}}\odot \boxtimes P_i
	\\
	T_{A_{i-1}}\odot \boxtimes (Q_i\odot  M_i)\ar[r]^-{\id\odot \boxtimes \phi_i}& T_{A_{i-1}}\odot \boxtimes (M_{i-1}\odot  P_i)
}\]
Essentially, it is $\boxtimes \phi_i$, but written in a form that allows us to use the dualizability of $\boxtimes M_i$ to take its trace.

\begin{thm}[Step 2 of \cref{thm:main_first_half}]\label{prop:fuller_equals_multitrace_axioms_version}
If $\sB$ is a bicategory with a shadowed Fuller structure, then for each tuple of maps $\phi_i$ as above the following diagram commutes.
\[\xymatrix@C=60pt{
        \sh{T_{A_i}\odot \boxtimes Q_i}	\ar[r]^-{\tr(\ful{\phi_1,\ldots ,\phi_n}{})}\ar[d]^-\tau_-\sim
	&\sh{T_{B_i}\odot \boxtimes P_i}\ar[d]^-\tau_-\sim
		\\
	\sh{Q_1,\ldots,Q_n} \ar[r]^-{\tr(\phi_1,\ldots, \phi_n)}
	& \sh{P_1,\ldots,P_n} 
}\]
\end{thm}

\begin{proof}  This is a modification of the compatibility between trace and shadow functors (\cref{bicategory_map_preserves_traces}).  
The required commutative diagram is \cref{fig:fuller_equals_multitrace}.  This is a very large diagram and so we have labeled the small regions so we can more easily indicate why they commute.  

The right column of commutative diagrams in \cref{fig:fuller_equals_multitrace} are mostly examples of  the naturality of $\tau$:
\begin{equation}\label{eq:tau_naturality}\xymatrix{\sh{T_{C_i}\odot \boxtimes X_i}\ar[r]^-\tau\ar[d]^{\sh{\id\odot \boxtimes f_i}}&\sh{X_1\odot\cdots \odot X_n}\ar[d]^-{\sh{f_1\odot \cdots \odot f_n}}
\\
\sh{T_{C_i}\odot \boxtimes Y_i}\ar[r]^-\tau&\sh{Y_1\odot\cdots \odot Y_n}
}\end{equation}
This diagram commutes for 2-cells $f_i\colon X_i\to Y_i$.   The remaining region in the right column is the assumed compatibility between $\theta$ and $\vartheta$.

Many of the left column regions  are the result of applying the functor $\sh{T_{C_i}\otimes - }$ to coherence axioms for $\boxtimes$.  These include: 
\begin{itemize}
\item naturality of $m_\boxtimes$
\begin{equation}\label{eq:m_box_naturality} 
\xymatrix{(\boxtimes X_i)\odot (\boxtimes Y_i) \ar[r]^-{m_\boxtimes}\ar[d]^{(\boxtimes f_i)\odot (\boxtimes g_i)}&\boxtimes (X_i\odot  Y_i)\ar[d]^{\boxtimes (f_i\odot g_i)}
\\
(\boxtimes W_i)\odot (\boxtimes Z_i) \ar[r]^-{m_\boxtimes}&\boxtimes (W_i\odot  Z_i)  
}\end{equation} 
\item associativity of $m_\boxtimes$
\begin{equation}\label{eq:m_box_assoc} \xymatrix{(\boxtimes X_i)\odot (\boxtimes Y_i)\odot (\boxtimes Z_i)\ar[r]^-{m_\boxtimes \odot \id}\ar[d]^{\id\odot m_\boxtimes}
&(\boxtimes (X_i\odot  Y_i))\odot (\boxtimes Z_i)\ar[d]^{m_\boxtimes}
\\
(\boxtimes X_i)\odot (\boxtimes (Y_i\odot Z_i))\ar[r]^-{m_\boxtimes}
&\boxtimes (X_i\odot Y_i\odot Z_i)
}\end{equation} 
\item compatibility of $i_\boxtimes$ and $m_\boxtimes$
\begin{equation}\label{eq:i_m}\xymatrix{(\boxtimes X_i)\odot U_{\boxtimes C_i}\ar[r]\ar[d]^{\id\odot i_\boxtimes}&\boxtimes X_i
\\
(\boxtimes X_i)\odot (\boxtimes U_{C_i})\ar[r]^-{m_\boxtimes}&\boxtimes (X_i\odot  U_{C_i})\ar[u]
}\end{equation} 
\end{itemize}
In \eqref{eq:i_m} the unlabeled arrows are unit isomorphisms.

  The dotted and dashed arrows are defined to be the composites of the remaining arrows bounding the relevant region.  For the two remaining regions: 
\begin{enumerate}
\renewcommand{\theenumi}{\arabic{enumi}}

\item\label{item:functoriality_odot_m_box} This square commutes by applying $\sh{-}$ to a square that commutes by the functoriality of $\odot$. 
\item\label{item:shadow_iso_naturality} This square commutes by the naturality of the shadow isomorphism.
\end{enumerate}
\end{proof}

\begin{figure}
\resizebox{\textwidth}{!}
{
\begin{tikzpicture}
\def\y{1.25}
\def\s{22}
\node at (19,-9*\y+\s*\y) (tbq){$\sh{T_{A_i}, \boxtimes Q_i}$};
\node at (23.5,-9*\y+\s*\y) (q){$\sh{Q_1,\ldots,Q_n}$};
\draw [->](tbq)--(q);

\node at (11,-9*\y+\s*\y) (tbqta) {$\sh{T_{A_i}, \boxtimes Q_i, U_{ \boxtimes A_i}}$};
\node at (15,-10*\y+\s*\y) (tbqbu) {$\sh{T_{A_i},  \boxtimes Q_i,  \boxtimes U_{A_i}}$};
\node at (19,-10*\y+\s*\y) (tbqu) {$\sh{T_{A_i}, \boxtimes (Q_i\odot U_{A_i})}$};
\node at (23.5,-10*\y+\s*\y) (qu){$\sh{Q_1,U_{A_1},\ldots,Q_n,U_{A_n}}$};
\draw[->](tbqta)--(tbqbu);

\node at (11,-12*\y+\s*\y) (tbqbmbn){$\sh{T_{A_i},  \boxtimes Q_i, \boxtimes M_i, \boxtimes N_i}$};
\node at (15,-11*\y+\s*\y) (tbqbmn){$\sh{T_{A_i}, \boxtimes Q_i, \boxtimes (M_i\odot N_i)}$};
\node at (19, -12*\y+\s*\y) (tbqmn){$\sh{T_{A_i}, \boxtimes (Q_i\odot M_i\odot N_i)}$};
\node at (23.5, -11*\y+\s*\y) (qmn){$\sh{Q_1, M_1, N_1, \ldots, Q_n, M_n, N_n}$};

\node at (15, -13*\y+\s*\y) (tbqmbn) {$\sh{T_{A_i}, \boxtimes(Q_i\odot M_i), \boxtimes N_i }$};
\draw[->](tbqbu)--(tbqbmn);
\draw[->] (tbqbmbn)--(tbqbmn);
\draw [->](tbqmn)--(qmn);
\node at (11, -15*\y+\s*\y) (tbmbpbn) {$\sh{T_{A_i}, \boxtimes M_{i-1}, \boxtimes P_i, \boxtimes N_i}$};
\node at (15, -14*\y+\s*\y) (tbmpbn){$\sh{T_{A_i}, \boxtimes (M_{i-1}\odot P_i), \boxtimes N_i}$};
\node at (19,-15*\y+\s*\y) (tbmpn){$\sh{T_{A_i}, \boxtimes (M_{i-1}\odot P_i\odot N_i)}$};
\node at (23.5,-16*\y+\s*\y) (mpn){$\sh{M_n,P_1,N_1,M_1,\ldots,P_n,N_n}$};

\node at (15,-16*\y+\s*\y) (tbmbpn){$\sh{T_{A_i}, \boxtimes M_{i-1}, \boxtimes (P_i\odot N_i)}$};
\draw [->](tbmpn)--(mpn);
\draw [->](tbmbpbn)--(tbmbpn);

\node at (11, -17*\y+\s*\y)(bmtbpbn){$\sh{\boxtimes M_{i}, T_{B_i},  \boxtimes P_i, \boxtimes N_i}$};
\node at (15.5,-17*\y+\s*\y) (bmtbpn){$\sh{\boxtimes M_{i}, T_{B_i}, \boxtimes (P_i\odot N_i)}$};
\draw[->](bmtbpbn)--(bmtbpn);
\draw [->](tbmbpn)--(bmtbpn);
\node at (15,-18*\y+\s*\y) (tbpnbm){$\sh{T_{B_i}, \boxtimes (P_i\odot N_i), \boxtimes M_{i}}$};
\node at (19,-19*\y+\s*\y) (tbpnm){$\sh{T_{B_i}, \boxtimes (P_i\odot N_i\odot M_i)}$};
\node at (23.5, -18*\y+\s*\y) (pnm){$\sh{P_1,N_1,M_1,\ldots,P_n,N_n,M_n}$};

\node at (11, -19*\y+\s*\y) (tbpbnbm){$\sh{T_{B_i}, \boxtimes P_i, \boxtimes N_i, \boxtimes M_i}$};
\node at (15, -20*\y+\s*\y) (tbpbnm){$\sh{T_{B_i}, \boxtimes P_i, \boxtimes (N_i\odot M_i)}$};

\draw [->](tbpbnbm)--(tbpnbm);
\draw [->](bmtbpbn) --(tbpbnbm);
\draw[->](tbpbnbm)--(tbpnbm);

\node at (11,-22*\y+\s*\y)(tbptu){$\sh{T_{B_i}, \boxtimes P_i, U_{ \boxtimes B_i} }$};
\node at (15,-21*\y+\s*\y)(tbpbu){$\sh{T_{B_i}, \boxtimes P_i, \boxtimes U_{B_i}}$};
\draw [->]  (tbptu)--(tbpbu);
\node at (19,-21*\y+\s*\y)(tbpu){$\sh{T_{B_i}, \boxtimes (P_i\odot U_{B_i})}$};
\node at (23.5, -21*\y+\s*\y)(pu){$\sh{P_1, U_{B_1},\ldots,P_n,U_{B_n}}$};

\node at (19,-22*\y+\s*\y) (tbp) {$\sh{T_{B_i}, \boxtimes P_i}$};
\node  at (23.5,-22*\y+\s*\y) (p) {$\sh{P_1,\ldots,P_n}$};

\draw[->] (tbp)--(p);
\draw[->] (pu)--(p);
\draw[->] (pnm)--(pu);
\draw[->] (mpn)--(pnm);
\draw[->] (qmn)--(mpn);
\draw[->] (qu)--(qmn);
\draw[<-] (q)--(qu);
\draw [->,dashed] (tbqbmbn)to[out = -150, in = 150](bmtbpbn);
\draw[->] (tbqbmbn)--(tbqmbn);
\draw[->] (tbqmbn)--(tbmpbn);
\draw [->] (tbmbpbn)--(tbmpbn);
\draw [->] (tbmbpbn)--(bmtbpbn);
\draw [->,dotted] (tbpbnbm)--(tbptu);
\draw [->] (bmtbpbn)--(tbpbnbm);
\draw[->,dotted] (tbqta)--(tbqbmbn);
\draw [->] (tbpbnbm)--(tbpbnm);
\draw [->] (tbpbnm)--(tbpbu);
\draw [->] (tbqta) --(tbqbu);
\draw [->] (tbqta)-- (tbq);
\draw [->] (tbqbu)--(tbqbmn);
\draw [->] (tbqbmbn)--(tbqbmn);
\draw [->] (tbqbu)--(tbqu);
\draw [->] (tbqu)--(qu);
\draw [->] (tbqu)--(tbq);
\draw [->] (tbptu)--(tbp);
\draw [->] (tbpbu)--(tbpu);
\draw [->] (tbpu)--(tbp);
\draw [->] (tbpu)--(pu);
\draw [->] (tbqu)--(tbqmn);
\draw [->] (tbqmn)--(qmn);
\draw [->] (tbqmn)--(tbmpn);
\draw [->](tbqbmn)--(tbqmn);
\draw [->](tbmpn)--(mpn);
\draw [->] (tbpnm)--(pnm);
\draw [->] (tbpnm)--(pnm);
\draw [->] (tbpbnm)--(tbpnm);
\draw [->] (tbpnm)--(tbpu);
\draw [->] (tbqmbn)--(tbqmn);
\draw [->](tbmpbn)--(tbmpn);
\draw [->](tbmbpbn)-- (tbmbpn);
\draw[->] (tbmbpn)--(tbmpn);
\draw[->] (bmtbpn)--(tbpnbm);
\draw[->] (tbpnbm)--(tbpnm);
\draw[->] (bmtbpbn)--(bmtbpn);

\node  at (barycentric cs:tbmpn=1,mpn=1,pnm=1,tbpnm=1,tbpnbm=1,tbmbpn=1) {(\ref{eq:fuller_defn_diagram})};

\node  at (barycentric cs:tbqta=1,tbq=1,tbqu=1,tbqbu=1) {$\sh{T_{A_i}\odot \eqref{eq:i_m} }$};
\node  at (barycentric cs:tbq=1,q=1,tbqu=1,qu=1) {(\ref{eq:tau_naturality})};

\node  at (barycentric cs:tbptu=1,tbp=1,tbpu=1,tbpbu=1) {$\sh{T_{B_i}\odot \eqref{eq:i_m} }$};

\node  at (barycentric cs:tbpu=1,p=1,tbp=1,pu=1) {(\ref{eq:tau_naturality})};
\node  at (barycentric cs:tbpnm=1,pnm=1,pu=1,tbpu=1) {(\ref{eq:tau_naturality})};
\node  at (barycentric cs:tbqmn=1,qmn=1,mpn=1,tbmpn=1) {(\ref{eq:tau_naturality})};
\node  at (barycentric cs:tbqu=2,qmn=1,qu=1,tbqmn=1) {(\ref{eq:tau_naturality})};

\node  at (barycentric cs:tbqbu=1,tbqu=2,tbqmn=1,tbqbmn=1) {$\sh{T_{A_i}\odot  \eqref{eq:m_box_naturality} }$};
\node  at (barycentric cs:tbqmbn=1,tbqmn=1,tbmpn=1,tbmpbn=1) {$\sh{T_{A_i}\odot \eqref{eq:m_box_naturality} }$};
\node  at (barycentric cs:tbpbnm=1,tbpnm=1,tbpu=2,tbpbu=1) {$\sh{T_{B_i}\odot  \eqref{eq:m_box_naturality} }$};

\node  at (barycentric cs:tbqbmn=1,tbqmn=1,tbqmbn=1,tbqbmbn=1) {$\sh{T_{A_i}\odot \eqref{eq:m_box_assoc} }$};
\node  at (barycentric cs:tbmpbn=1,tbmpn=1,tbmbpn=1,tbmbpbn=1) {$\sh{T_{A_i}\odot \eqref{eq:m_box_assoc} }$};
\node  at (barycentric cs:tbpnbm=1,tbpnm=1,tbpbnm=1,tbpbnbm=1) {$\sh{T_{B_i}\odot  \eqref{eq:m_box_assoc} }$};

\node  at (barycentric cs:tbmbpbn=1,tbmbpn=1,bmtbpn=1,bmtbpbn=2) {(\ref{item:functoriality_odot_m_box})};
\node  at (barycentric cs:tbpbnbm=1,tbpnbm=1,bmtbpn=1,bmtbpbn=2) {(\ref{item:shadow_iso_naturality})};

\end{tikzpicture}
}\caption{Proof of \cref{prop:fuller_equals_multitrace_axioms_version}}\label{fig:fuller_equals_multitrace}
\end{figure}

Together \cref{lem:nthpower,prop:fuller_equals_multitrace_axioms_version} prove a very abstract and general form of the ``unwinding'' argument, that the trace of a Fuller construction is isomorphic to the trace of the composite. To recover \cref{thm:main_first_half} from this, we have to further develop the case where the maps $\phi_i$ are canonical isomorphisms of base-change objects associated to maps $f_i$ in some 1-category $\bS$.

\section{Base change}\label{sec:base_change}
If $\sB$ is a shadowed bicategory with an $n$-Fuller structure and $\bS$ is a cartesian monoidal 1-category, a {\bf system of base-change objects for $\sB$ indexed by $\bS$} is  the following data and conditions.
\begin{enumerate}
	\item A pseudofunctor $[]\colon \bS \to \sB$.

 In particular, natural isomorphisms
	\[ m_{[]}\colon \bcr{B_{n-1}}{f_n}{B_n} \odot \ldots \odot \bcr{B_1}{f_2}{B_2} \odot \bcr{B_0}{f_1}{B_1} \cong \bcr{B_0}{f_n \circ \ldots \circ f_1}{B_n} \]
	compatible with composition. (The unit isomorphism $i_{[]}$ is not necessary.)

	\item A vertical natural isomorphism $\pi$ filling the square of pseudofunctors
	\[ \xymatrix{
		\bS^{\times n} \ar[r]^-\prod \ar[d]_-{[]} & \bS \ar[d]^-{[]} \\
		\sB^{\times n} \ar[r]^-\boxtimes & \sB
	} \]
	 where $\prod$ denotes a fixed model for the $n$-fold product in $\bS$.

This implies  $\boxtimes A_i = \prod A_i$ for a tuple of 0-cells $A_i$, and for each $n$-tuple of maps $A_i \overset{f_i}\to B_i$ there is an isomorphism of 1-cells 
	\[ \pi\colon \boxtimes \bcr{A_i}{f_i}{B_i} \overset\cong\to \bcr{\prod A_i}{\prod f_i}{\prod B_i} \]
so that for any $n$-tuple of composable maps $A_i \overset{f_i}\to B_i \overset{g_i}\to C_i$, the following pentagon of isomorphisms commutes.
	\begin{equation}\label{eq:mboxtimes_m_bracket}{
		\xymatrix{
			\left(\boxtimes \bcr{B_i}{g_i}{C_i}\right) \odot \left(\boxtimes \bcr{A_i}{f_i}{B_i}\right)\ar[dd]^-{\pi\odot \pi}_-\cong \ar[r]^-{m_{\boxtimes}}_-\cong&
			\boxtimes \left(\bcr{B_i}{g_i}{C_i} \odot \bcr{A_i}{f_i}{B_i}\right)\ar[d]^{\boxtimes m_{[]}}_-\cong
			\\&
			\boxtimes \bcr{A_i}{g_i \circ f_i}{C_i}\ar[d]^-\pi _-\cong
			\\
			\bcr{\prod B_i}{\prod g_i}{\prod C_i}\odot \bcr{\prod A_i}{\prod  f_i}{\prod B_i}\ar[r]^-{m_{[]}}_-\cong
			&
			\bcr{\prod A_i}{\prod (g_i \circ f_i)}{\prod C_i}
		}
	}\end{equation}
(Again, the corresponding map for the unit of $\boxtimes$ is not necessary.)

	\item An equality $T_{B_i} = \bcr{\prod B_i}{\cong}{\prod B_{i+1}}$ so that the following diagram relating $\vartheta$, $\pi$ and the pseudofunctor structure commutes.
	\begin{equation}\label{eq:vartheta_m_bracket} \xymatrix@C=8pt{
		\bcr{\prod B_i}{\cong}{\prod B_{i+1}} \odot \left(\boxtimes \bcr{E_i}{p_i}{B_i}\right) \ar[dd]_-{\id\odot \pi}^-\cong \ar[r]^-\vartheta_-\cong 
		&
		\left(\boxtimes \bcr{E_{i+1}}{p_{i+1}}{B_{i+1}}\right)\odot \bcr{\prod E_i}{\cong}{\prod E_{i+1}} \ar[d]^-{\pi\odot \id}_-\cong \\
&
		\bcr{\prod E_{i+1}}{\prod p_{i+1}}{\prod B_{i+1}} \odot \bcr{\prod E_i}{\cong}{\prod E_{i+1}}\ar[d]_-\cong^{m_{[]}} 
\\		
		\bcr{\prod B_i}{\cong}{\prod B_{i+1}} \odot \bcr{\prod E_i}{\prod p_i}{\prod B_i} \ar[r]_-\cong^-{m_{[]}} & \bcr{\prod E_i}{\textup{shift} \circ \prod p_i}{\prod B_{i+1}} 
	} \end{equation}
\end{enumerate}

\begin{example}
	The bicategory $\Ex$ has a system of base-change objects from the category of unbased spaces. As with the $n$-Fuller structure, this follows from \cref{thm:spectra_smbf} and the results from \cite{mp2} summarized in \cref{thm:indexed_bicat}.
\end{example}

If $\sB$ is a shadowed $n$-Fuller category with base change objects, any commuting square in $\bS$ of the form
\[\xymatrix{ E \ar[r]^{f}\ar[d]_{p} & E \ar[d]^p \\ B \ar[r]^{\overline{f}} & B }\]
gives an isomorphism of base-change objects
\[ \bcr{B}{\overline{f}}{B} \odot \bcr{E}{p}{B}
\to \bcr{E}{p}{B} \odot \bcr{E}{f}{E}. \]
If the base-change object $\bcr{E}{p}{B}$ is right-dualizable in $\sB$, then we can take the trace of this map. This is the Reidemeister trace associated to the above commuting square. Note that in $\Ex$, when $B = *$, it agrees with the definition of $R(f)$ we gave in \cref{sec:traces_in_bicat}.

If we instead have an $n$-tuple of  commuting squares 
	\[\xymatrix{ E_i \ar[r]^{f_i}\ar[d]_{p_i} & E_{i-1} \ar[d]^{p_{i-1}} \\ B_i \ar[r]^{\overline{f_i}} & B_{i-1} }\]
in $\bS$ then we can define a commuting square 
	\[\xymatrix @C=5em{ \prod E_i \ar[r]^{\ful{f_1,\ldots,f_n}{}} \ar[d]_{\prod p_i} &\prod E_i \ar[d]^{\prod p_{i}} \\ \prod B_i \ar[r]^{\ful{\overline{f_1},\ldots,\overline{f_n}}{}} & \prod B_i. }\]
The first squares define maps 	
\[ \phi_i\colon \bcr{B_i}{\overline{f}_i}{B_{i-1}} \odot \bcr{E_i}{p_i}{B_i}
\to \bcr{E_{i-1}}{p_{i-1}}{B_{i-1}} \odot \bcr{E_i}{f_i}{E_{i-1}} \]
for each $i$, and the second square defines a map 
\begin{center}\resizebox{\textwidth}{!}{ $\phi\colon \bcr{\prod B_i}{\ful{\overline f_1,\ldots , \overline f_n}{}}{\prod B_i} \odot \bcr{\prod E_i}{\prod p_i}{\prod B_i}\to \bcr {\prod E_{i}}{\prod p_{i}}{\prod B_{i}} \odot  \bcr{\prod E_i}{\ful{ f_1,\ldots ,  f_n}{}}{\prod E_i}$.}
\end{center}

\begin{prop}[Step 3 of \cref{thm:main_first_half}]
\label{prop:fuller_equals_multitrace_axioms_version_2}
	In a shadowed $n$-Fuller category $\sB$ with a system of base-change objects from $\bS$, for any $n$-tuple of commuting squares in $\bS$
	\[\xymatrix{ E_i \ar[r]^{f_i}\ar[d]_{p_i} & E_{i-1} \ar[d]^{p_{i-1}} \\ B_i \ar[r]^{\overline{f_i}} & B_{i-1} }\]
there is a commuting diagram 
\[\xymatrix{
	\Bigsh{T_{B_i} \odot \boxtimes  \bcr{B_i}{\overline{f}_i}{B_{i-1}} }\ar[r]^-{\sh{\cong}}\ar[d]_{\tr(\ful{\phi_1,\ldots,\phi_n}{})} &
	\Bigsh{\bcr{\prod B_i}{\ful{\overline f_1,\ldots , \overline f_n}{}}{\prod B_i}}\ar[d]^{\tr(\phi)} \\
	\Bigsh{T_{E_i} \odot \boxtimes \bcr{E_i}{f_i}{E_{i-1}} }\ar[r]^-{\sh{\cong}} &
	\Bigsh{ \bcr {\prod E_{i}}{\ful{ f_1,\ldots ,f_n}{}}{\prod E_{i}}}
}\]
\end{prop}

\begin{proof}  
The first step in this proof is to compare $\ful{\phi_1,\ldots,\phi_n}{}$ and  $\phi$.  For this we use the commutative diagram in \cref{fig:full_multi_base_change}.   We have already encountered all of the small regions in the diagram.  
The regions not labeled by an equation number commute by: 
\begin{enumerate}
\renewcommand{\theenumi}{\arabic{enumi}}
\item \label{item:functoriality_odot_2} The functoriality of $\odot$.
\item \label{item:natuality_of_m_bracket} Naturality of $m_{[]}$.
\end{enumerate}
The left composite in \cref{fig:full_multi_base_change}  is $\ful{\phi_1,\ldots,\phi_n}{}$ and the right composite is $\phi$. 

A straightforward diagram chase shows that for a diagram of the form below on the left, where $\beta$ is an isomorphism and $M_1$ and $M_2$ are dualizable, the corresponding diagram of traces on the right commutes.
\[\xymatrix{
		Q_1\odot M_1\ar[r]^{\alpha \odot \beta}\ar[d]^{f_1}&Q_2\odot M_2\ar[d]^{f_2}
		\\
		  M_1\odot P_1\ar[r]^{\beta\odot \gamma}&M_2\odot P_2}
\qquad
\xymatrix{
            \sh{Q_1}\ar[r]^{\sh{\alpha}}\ar[d]^{\tr(f_1)}&\sh{Q_2}\ar[d]^{\tr(f_2)}\\
             \sh{P_1}\ar[r]^{\sh{\gamma}}&\sh{P_2}}\] 
Looking only at the outside edges of \cref{fig:full_multi_base_change} we have a commutative diagram of exactly this form,
\[\resizebox{\textwidth}{!}{\xymatrix@C=50pt{
	\bcr{\prod B_{i-1}}{\cong}{\prod B_i} \odot 
		\boxtimes \bcr{B_i}{\overline{f_i}}{B_{i-1}} \odot 
		\boxtimes \bcr{E_i}{p_i}{B_i}
		\ar[d]^{\ful{\phi_1,\ldots,\phi_n}{}}
		\ar[r]^-{(m_{[]}\circ (\id \odot \pi))\odot \pi }
	&\bcr{\prod B_i}{\ful{\overline{f_1},\ldots,\overline{f_n}}{}}{\prod B_i}\odot 
		\bcr{\prod E_i}{\prod p_i}{\prod B_i}
		\ar[d]^-{\phi}
\\
      	\boxtimes \bcr{E_i}{p_i}{B_i} \odot 
		\bcr{\prod E_{i-1}}{\cong}{\prod E_i} \odot 
		\boxtimes \bcr{E_i}{f_i}{E_{i-1}}
		\ar[r]^-{\pi \odot (m_{[]}\circ (\id \odot \pi))}
	&\bcr{\prod E_i}{\prod p_i}{\prod B_i}\odot 
		\bcr{\prod E_i}{\ful{f_1,\ldots,f_n}{}}{\prod E_i}.
	}}\]
This completes the proof.
\end{proof}

\afterpage{
\begin{figure}
\resizebox{1\textwidth}{!}
{\begin{tikzpicture}
\def\x{7}
\def\y{3.5}
	\node at (0*\x,0*\y) (00) {$\begin{pmatrix} \bcr{\prod B_{i-1}}{\cong}{\prod B_i} \\
		\boxtimes \bcr{B_i}{\overline{f_i}}{B_{i-1}} \\
		\boxtimes \bcr{E_i}{p_i}{B_i}\end{pmatrix}$};
	\node at (1.5*\x,0*\y)(10){$\begin{pmatrix}\bcr{\prod B_{i-1}}{\cong}{\prod B_i} \\
		\boxtimes \bcr{B_i}{\overline{f_i}}{B_{i-1}}  \\
		\bcr{\prod E_i}{\prod p_i}{\prod B_i}\end{pmatrix} $};
	\node at (2.5*\x,0*\y)(20){$\begin{pmatrix}\bcr{\prod B_i}{\ful{\overline{f_1},\ldots,\overline{f_n}}{}}{\prod B_i}\\
		\bcr{\prod E_i}{\prod p_i}{\prod B_i}\end{pmatrix}$};
        \node at (1.5*\x,-1*\y)(01){$\begin{pmatrix}\bcr{\prod B_{i-1}}{\cong}{\prod B_i} \\
		\bcr{\prod B_i}{\prod \overline{f_i}}{\prod B_{i-1}} \\
		\bcr{\prod E_i}{\prod p_i}{\prod B_i} \end{pmatrix}$};
	\node at (0*\x, -2*\y)(02){$\begin{pmatrix}\bcr{\prod B_{i-1}}{\cong}{\prod B_i} \\
		\boxtimes \left(\bcr{B_i}{\overline{f_i}}{B_{i-1}} \odot 
		\bcr{E_i}{p_i}{B_i} \right)\end{pmatrix}$};
      	\node at (.75*\x,-1.5*\y)(12){$\begin{pmatrix}\bcr{\prod B_{i-1}}{\cong}{\prod B_i} \\
		\boxtimes \bcr{E_i}{\overline{f}_{i} \circ p_i}{B_{i-1}} \end{pmatrix}$};
     	\node at (1.5*\x,-2*\y)(22){$\begin{pmatrix}\bcr{\prod B_{i-1}}{\cong}{\prod B_i} \\ 
		\bcr{\prod E_i}{\prod (\overline{f}_{i} \circ p_i)}{\prod B_{i-1}}\end{pmatrix}$};
      	\node at (2.5*\x, -2*\y)(32){$\bcr{\prod E_{i}}{ \prod (\overline{f}_{i} \circ p_i)\circ \cong}{\prod B_i}$};
	\node at (0*\x,-3*\y)(03){$\begin{pmatrix}\bcr{\prod B_{i-1}}{\cong}{\prod B_i} \\
		\boxtimes \left(\bcr{E_{i-1}}{p_{i-1}}{B_{i-1}} \odot 
		\bcr{E_i}{f_i}{E_{i-1}}\right)\end{pmatrix}$};
	\node at (.75*\x, -2.5*\y)(13){$\begin{pmatrix}\bcr{\prod B_{i-1}}{\cong}{\prod B_i} \\
		\boxtimes \bcr{E_i}{p_{i-1} \circ f_i}{B_{i-1}} \end{pmatrix}$};
     	\node at (1.5*\x,-3*\y)(23){$\begin{pmatrix}\bcr{\prod B_{i-1}}{\cong}{\prod B_i} \\
		\bcr{\prod E_i}{\prod (p_{i-1} \circ f_i)}{\prod B_{i-1}}\end{pmatrix}$};
	\node at (2.5*\x,-3*\y)(33){$\bcr{\prod E_{i}}{\prod (p_{i-1} \circ f_i)\circ \cong }{\prod B_i} $};
	\node at (0*\x, -4*\y)(04){$\begin{pmatrix}\bcr{\prod B_{i-1}}{\cong}{\prod B_i} \\
		\boxtimes \bcr{E_{i-1}}{p_{i-1}}{B_{i-1}} \\
		\boxtimes \bcr{E_i}{f_i}{E_{i-1}} \end{pmatrix}$};
        \node at (1.5*\x,-4*\y)(14){$\begin{pmatrix}\bcr{\prod B_{i-1}}{\cong}{\prod B_i} \\
		\bcr{\prod E_{i-1}}{\prod p_{i-1}}{\prod B_{i-1}} \\
		\bcr{\prod E_i}{\prod f_i}{\prod E_{i-1}}\end{pmatrix} $};
        \node at (.75*\x, -5*\y)(05){$\begin{pmatrix}\bcr{\prod B_{i-1}}{\cong}{\prod B_i} \\
		\bcr{\prod E_{i-1}}{\prod p_{i-1}}{\prod B_{i-1}} \\ 
		\boxtimes \bcr{E_i}{f_i}{E_{i-1}}\end{pmatrix}$};
 	\node at (.75*\x, -6*\y)(06){$\begin{pmatrix} \bcr{\prod E_{i-1}}{\prod p_i \circ \cong}{\prod B_i}\\
		\boxtimes \bcr{E_i}{f_i}{E_{i-1}}\end{pmatrix}$};
         \node at (1.5*\x,-6*\y)(16){$\begin{pmatrix}\bcr{\prod E_{i-1}}{\prod p_i \circ \cong}{\prod B_i} \\ 
		\bcr{\prod E_i}{\prod f_i}{\prod E_{i-1}}\end{pmatrix}$}; 
      	\node at (1.5*\x, -7*\y)(07){$\begin{pmatrix} \bcr{\prod E_i}{\prod p_i}{\prod B_i}\\
		\bcr{\prod E_{i-1}}{\cong}{\prod E_i}\\
		\bcr{\prod E_i}{\prod f_i}{\prod E_{i-1}}\end{pmatrix} $};
     \node at (0*\x,-8*\y)(08){$\begin{pmatrix}\boxtimes \bcr{E_i}{p_i}{B_i} \\
		\bcr{\prod E_{i-1}}{\cong}{\prod E_i} \\
		\boxtimes \bcr{E_i}{f_i}{E_{i-1}}\end{pmatrix}$};
	\node at (.75*\x,-8*\y)(18){$\begin{pmatrix}\bcr{\prod E_i}{\prod p_i}{\prod B_i}\\
		\bcr{\prod E_{i-1}}{\cong}{\prod E_i}\\
		\boxtimes \bcr{E_i}{f_i}{E_{i-1}}\end{pmatrix}$};
	\node at (2.5*\x,-8*\y)(28){$\begin{pmatrix}\bcr{\prod E_i}{\prod p_i}{\prod B_i}\\
		\bcr{\prod E_i}{\ful{f_1,\ldots,f_n}{}}{\prod E_i}\end{pmatrix}$};
		\draw[->] (00)--(02) node[midway,left]{$\begin{pmatrix}\id\\ m_{\boxtimes}\end{pmatrix}$};
		\draw[->] (00)--(01)node[midway,below]{$\begin{pmatrix}\id\\ \pi\\ \pi\end{pmatrix}$};
		\draw[->] (00)--(10)node[midway,above]{$\begin{pmatrix}\id \\ \id \\ \pi\end{pmatrix}$};
		\draw[->] (10)--(01)node[midway,right] {$\begin{pmatrix}\id \\ \pi\\ \id \end{pmatrix}$};
		\draw[->] (10)--(20)node[midway,above] {$\begin{pmatrix}m_{[]}\circ (\id \odot \pi)\\ \id \end{pmatrix}$};
		\draw[->] (20)--(32)node[midway,right] {$m_{[]}$};
		\draw[->] (01)--(22)node[midway,left] {$\begin{pmatrix}\id \\ m_{[]}\end{pmatrix}$};
		\draw[->] (01)--(20)node[midway,below] {$\begin{pmatrix}m_{[]}\\ \id \end{pmatrix}$};
		\draw[->] (02)--(03)node[midway,left] {$\begin{pmatrix}\id\\ \boxtimes \phi_i\end{pmatrix}$};
		\draw[->] (02)--(12)node[midway,above] {$\begin{pmatrix}\id\\ \boxtimes m_{[]}\end{pmatrix}$};
		\draw[->] (12)--(22)node[midway,above] {$\begin{pmatrix}\id\\ \pi\end{pmatrix}$};
		\draw[double distance=2pt] (12)--(13); 
		\draw[double distance=2pt] (22)--(23); 
		\draw[->] (22)--(32)node[midway,above] {$m_{[]}$};
		\draw[double distance=2pt] (32)--(33); 
		\draw[->] (03)--(13)node[midway,above] {$\begin{pmatrix}\id \\ \boxtimes m_{[]}\end{pmatrix}$};
		\draw[->] (13)--(23)node[midway,above] {$\begin{pmatrix}\id \\ \pi\end{pmatrix}$};
		\draw[->] (23)--(33)node[midway,above] {$m_{[]}$};
		\draw[->] (04)--(08)node[midway,left] {$\begin{pmatrix}\vartheta\\ \id\end{pmatrix}$};
		\draw[->] (04)--(03)node[midway,left] {$\begin{pmatrix}\id\\ m_{\boxtimes}\end{pmatrix}$};
		\draw[->] (04)--(14)node[midway,above] {$\begin{pmatrix}\id\\ \pi\\ \pi\end{pmatrix}$};
		\draw[->] (04)--(05)node[midway,above] {$\begin{pmatrix}\id\\ \pi \\ \id\end{pmatrix}$};
		\draw[->] (14)--(23)node[midway,left] {$\begin{pmatrix}\id \\ m_{[]}\end{pmatrix}$};
		\draw[->] (14)--(16)node[midway,left] {$\begin{pmatrix}m_{[]}\\ \id \end{pmatrix}$};
		\draw[->] (05)--(14)node[midway,above] {$\begin{pmatrix}\id \\ \id\\ \pi\end{pmatrix}$};
		\draw[->] (05)--(06)node[midway,left] {$\begin{pmatrix}m_{[]}\\ \id \end{pmatrix}$};
		\draw[->] (06)--(16)node[midway,above] {$\begin{pmatrix}\id \\ \pi\end{pmatrix}$};
		\draw[->] (16)--(33)node[midway,left] {$m_{[]}$};
		\draw[->] (07)--(16)node[midway,left] {$\begin{pmatrix}m_{[]}\\ \id\end{pmatrix} $};
		\draw[->] (07)--(28)node[midway,above] {$\begin{pmatrix}\id\\ m_{[]}\end{pmatrix}$};
		\draw[->] (08)--(18)node[midway,above] {$\begin{pmatrix}\pi\\ \id\\ \id\end{pmatrix}$};
		\draw[->] (18)--(06) node[midway,left] {$\begin{pmatrix}m_{[]}\\ \id\end{pmatrix} $};
		\draw[->] (18)-- (07)node[midway,above] {$\begin{pmatrix}\id \\ \id \\ \pi\end{pmatrix}$};
		\draw[->] (18)-- (28)node[midway,above] {$\begin{pmatrix}\id \\ m_{[]}\circ (\id \odot \pi)\end{pmatrix}$};
		\draw[->] (28)--(33)node[midway,left] {$m_{[]}$};

\node  at (barycentric cs:00=1,10=1,01=1) {(\ref{item:functoriality_odot_2})};
\node  at (barycentric cs:04=1,14=1,05=1) {(\ref{item:functoriality_odot_2})};
\node  at (barycentric cs:06=1,16=2,14=1,05=1) {(\ref{item:functoriality_odot_2})};
\node  at (barycentric cs:10=1,01=1,20=1) {(\ref{item:functoriality_odot_2})};
\node  at (barycentric cs:16=1,33=1,28=1,07=1) {(\ref{eq:m_box_assoc})};
\node  at (barycentric cs:06=2,16=1,07=1,18=1) {(\ref{item:functoriality_odot_2})};
\node  at (barycentric cs:23=1,33=3,16=1,14=1) {(\ref{eq:m_box_assoc})};
\node  at (barycentric cs:20=1,32=1,22=1,01=1) {(\ref{eq:m_box_assoc})};
\node  at (barycentric cs:18=1,07=2,28=1) {(\ref{item:functoriality_odot_2})};
\node  at (barycentric cs:12=1,22=1,23=1,13=1) {(\ref{item:functoriality_odot_2})};
\node  at (barycentric cs:32=1,22=1,23=1,33=1) {(\ref{item:natuality_of_m_bracket})};
\node  at (barycentric cs:02=1,12=1,13=1,03=1) {(\ref{item:natuality_of_m_bracket})};
\node  at (barycentric cs:00=3,01=1,22=1,12=1,02=1) {(\ref{eq:mboxtimes_m_bracket})};

\node  at (barycentric cs:03=1,13=1,23=1,14=1,04=1) {(\ref{eq:mboxtimes_m_bracket})};

\node  at (barycentric cs:04=2,05=1,06=1,18=1,08=2) {(\ref{eq:vartheta_m_bracket})};

	\end{tikzpicture} }\caption{Comparing Fuller maps.   \\Stacked entries inside a single pair of large parentheses are combined with $\odot$.
}\label{fig:full_multi_base_change}
\end{figure}
\clearpage}

Combining \cref{lem:nthpower,prop:fuller_equals_multitrace_axioms_version,prop:fuller_equals_multitrace_axioms_version_2} in the setting of \cref{prop:fuller_equals_multitrace_axioms_version_2}, in other words the first three steps of \cref{thm:main_first_half}, we get a commutative diagram 
\[\xymatrix@C=50pt{
	\Bigsh{\bcr{\prod B_i}{\ful{\overline f_1,\ldots , \overline f_n}{}}{\prod B_i}}\ar[r]^{\tr(\phi)} 
	&
	\Bigsh{ \bcr {\prod E_{i}}{\ful{ f_1,\ldots ,f_n}{}}{\prod E_{i}}}
\\
		\Bigsh{T_{B_i} \odot \boxtimes  \bcr{B_i}{\overline{f}_i}{B_{i-1}} }\ar[u]^-{\sh{\cong}}\ar[r]^{\tr(\ful{\phi_1,\ldots,\phi_n}{})} \ar[d]_-\cong 
	&
	\Bigsh{T_{E_i} \odot \boxtimes \bcr{E_i}{f_i}{E_{i-1}} }\ar[u]^-{\sh{\cong}} \ar[d]_-\cong
\\
	\Bigsh{\bcr{B_1}{\overline{f}_1}{B_{n}},\ldots , \bcr{B_n}{\overline{f}_n}{B_{n-1}} }\ar[r]^{\tr(\phi_1,\ldots ,\phi_n)}\ar[d]_{\sh{\cong}}
	&
	\Bigsh{\bcr{E_1}{f_1}{E_{n}},\ldots ,\bcr{E_n}{f_n}{E_{n-1}}}\ar[d]_{\sh{\cong}}
\\
	\Bigsh{\bcr{B_1}{\overline{f}_1}{B_{n}} \odot \ldots \odot \bcr{B_n}{\overline{f}_n}{B_{n-1}} }\ar[r]^{\tr(\seqco{\phi_1}{\phi_n})}\ar[d]_{\sh{\cong}}
&	\Bigsh{\bcr{E_1}{f_1}{E_{n}} \odot \ldots \odot \bcr{E_n}{f_n}{E_{n-1}}}\ar[d]_{\sh{\cong}}
\\
\Bigsh{\bcr{B_n}{\seqco{\overline f_1}{\overline f_n}}{B_{n}} }\ar@{-->}[r]
&	\Bigsh{\bcr{E_n}{\seqco{f_1}{f_n}}{E_n}}
}
\]
relating the Reidemeister trace of the Fuller construction to the trace of the composite of base-change isomorphisms $\seqco{\phi_1}{\phi_n}$. To fill in the remaining dashed arrow, we observe that $\seqco{\phi_1}{\phi_n}$ arises by pasting the base-change isomorphisms that bring us from the lower to the upper route in the following diagram.
\[\xymatrix{
	E_n \ar[r]^{f_n} \ar[d]_{p_n} &
	E_{n-1} \ar[r]^-{f_{n-1}} \ar[d]^{p_{n-1}} &
	\cdots \ar[r]^{f_2} &
	E_1 \ar[r]^{f_1}\ar[d]^{p_1} & 
	E_n \ar[d]^{p_n} \\
	B_n \ar[r]^{\overline{f_n}} &
	B_{n-1} \ar[r]^{\overline{f_{n-1}}} &
	\cdots \ar[r]^{\overline{f_2}} &
	B_1 \ar[r]^{\overline{f_1}} &
	B_n
}\]
Using one last time the fact that $[]$ is a pseudofunctor, along the canonical maps this is identified with the isomorphism provided by $[]$ for the composite square. Therefore the dashed arrow is the Reidemeister trace for the square
\[\xymatrix{
	E_n \ar[r]^{\seqco{f_1}{f_n}} \ar[d]_{p_n} &
	E_n \ar[d]^{p_n} \\
	B_n \ar[r]^{\seqco{\overline{f_1}}{\overline{f_n}}} &
	B_n.
}\]
Taking $B$ to be the terminal object, this gives the fourth and final piece of the proof of the following.
\begin{cor}\label{cor:Fuller_trace_is_trace_of_composite}
	In a shadowed $n$-Fuller category $\sB$ with a system of base-change objects from $\bS$, for any $n$-tuple of composable maps $f_i\colon X_i \to X_{i-1}$ in $\bS$, the Reidemeister trace of the Fuller construction $\ful{f_1,\ldots,f_n}{}$ is isomorphic to the Reidemeister trace of the composite $\seqco{f_1}{f_n}$.
\end{cor}
Since the bicategory $\Ex$ has a shadowed $n$-Fuller structure and a system of base-change objects, this proves \cref{thm:main_first_half}. Our motivation for stating the proof at this level of generality is that the same argument will establish a more general result for the fiberwise Reidemeister trace and Fuller trace.  See \cref{part:fiberwise}.


\part{Varying the group $G$}\label{part:varying}

In this section we prove the first two triangles of \cref{thm:simple_main_thm} commute:

\begin{thmx}
	\label{thm:main_second_half}
	The following diagram commutes up to homotopy.
	\[\xymatrix@C=40pt@R=40pt{
		&&\Sph \ar[lld]_-{R(\ful{f}n)^{C_n}} 
		\ar[dl]^-{R(\ful{f}k)^{C_k}} 
		\ar[d]^-{R(\ful{f}k)} 
		\\
		(\Sigma^\infty_+ \Lambda^{\ful{f}n} X^n)^{C_n} \ar[r]^-{R} &
		(\Sigma^\infty_+ \Lambda^{\ful{f}k} X^k)^{C_k} 
		 \ar[r]^-{F}
		&	\Sigma^\infty_+ \Lambda^{\ful{f}k} X^k
	}\]
\end{thmx}

The essential idea is to show that the geometric fixed point functor $\Phi^H$, and the functor $\iota_H^*$ that forgets group actions, are strong shadow functors, so that they preserve Reidemeister traces by \cref{bicategory_map_preserves_traces}. In contrast to the previous part where we black-boxed all needed properties of parameterized spectra, in this part we work directly with these spectra. In \cref{sec:bifi} we recall some general theory about passing between symmetric monoidal bifibrations (smbfs) and bicategories, and in \cref{sec:spectra} we apply these ideas to the smbf of parametrized $G$-spectra.
We finish the proof of \cref{thm:main_second_half} in \cref{sec:grand_finale}. 

\section{Symmetric monoidal bifibrations}\label{sec:bifi}

Strong shadow functors such as $\Phi^H$ are difficult to construct on $G\Ex$ because the operations $\odot$ and $\sh{-}$ are composites of left and right derived functors. It is far easier to show that the constituent pieces of $\odot$ are  separately preserved by $\Phi^H$, and then assemble those pieces back together. The structure of these constituent pieces is captured formally by the idea of a symmetric monoidal bifibration (smbf). 

In this paper, a \textbf{bifibration} is a functor $\pi\colon \sC \to \bS$ from a category $\sC$ to a cartesian monoidal category $\bS$ with the following properties.
\begin{itemize}
	\item For every pair of an object $X \in \sC$ and an arrow $A \xto{f} \pi(X)$ in $\bS$, there is \textbf{cartesian arrow} $f^*X \to X$ satisfying a universal property given in shorthand in \cref{fig:cartesian_arrow}.
	\item For every pair of an object $X \in \sC$ and an arrow $\pi(X) \xto{f} A$ in $\bS$, there is a \textbf{cocartesian arrow} $X \to f_!X$ satisfying a universal property given in shorthand in \cref{fig:cocartesian_arrow}.
	\item There is a class of \textbf{Beck-Chevalley squares} in $\bS$,
	\[\xymatrix@-.5pc{A \ar[r]^f\ar[d]_h & B \ar[d]^g\\
		C\ar[r]_k & D,}\]
	 such that in each one the natural transformation of functors $\sC^C \to \sC^B$
	\begin{equation}\label{beck_chevalley}f_! h^* \to f_!h^*k^*k_! \xto{\sim} f_!f^*g^*k_! \to g^*k_!
	\end{equation}
	is an isomorphism.
	\item The class of Beck-Chevalley squares can be chosen to include the following squares.
	\begin{itemize}
		\item For any pair of composable maps $A \overset{f}\to B \overset{g}\to C$ and $A' \overset{f'}\to B'$, 
		\[ \xymatrix{
			A \times A' \ar[r]^-{1 \times f'} \ar[d]_-{f \times 1} & A \times B' \ar[d]^-{f \times 1} \\
			B \times A' \ar[r]^-{1 \times f'} & B \times B'
		}
		\qquad
		\xymatrix{
			A \ar[d]_-{(1,f)} \ar[r]^-{(1,g \circ f)} & A \times C \ar[d]^-{(1,f) \times 1} \\
			A \times B \ar[r]^-{1 \times (1,g)} & A \times B \times C
		}
		\qquad
		\xymatrix{
			A \ar[d]_-{f} \ar[r]^-{(1,g \circ f)} & A \times C \ar[d]^-{f \times 1} \\
			B \ar[r]^-{(1,g)} & B \times C.
		} \]
		\item Any square isomorphic to a Beck-Chevalley square. (This includes commuting squares with two parallel isomorphisms.)
		\item Any product of a Beck-Chevalley square and an object of $\bS$.
	\end{itemize}
\end{itemize}
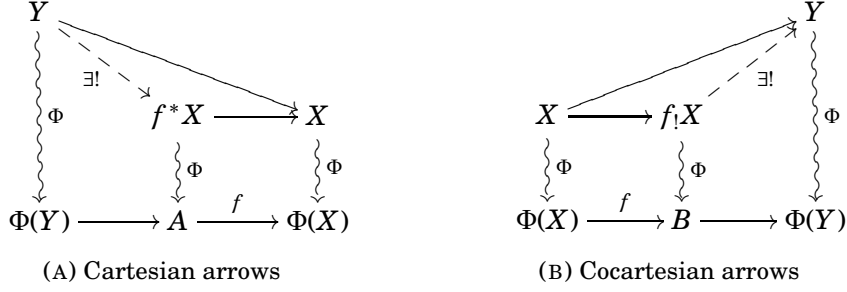
\begin{figure}
	\hspace{7em}
	\begin{subfigure}[t]{12em}
		\xymatrix{
			Y \ar@{~>}[dd]^-\Phi \ar[rrd] \ar@{-->}[rd]_-{\exists !} & & \\
			& f^*X \ar[r] \ar@{~>}[d]^-\Phi & X \ar@{~>}[d]^-\Phi \\
			\Phi(Y) \ar[r] & A \ar[r]^-f & \Phi(X)
		} 
		\caption{Cartesian arrows}\label{fig:cartesian_arrow}
	\end{subfigure}
	\hspace{5em}
	\begin{subfigure}[t]{12em}
		\xymatrix{
			& & Y \ar@{~>}[dd]^-\Phi \\
			X \ar[r] \ar[rru] \ar@{~>}[d]^-\Phi & f_!X \ar@{~>}[d]^-\Phi \ar@{-->}[ru]_-{\exists !} & \\
			\Phi(X) \ar[r]^-f & B \ar[r] & \Phi(Y)
		}
		\caption{Cocartesian arrows}\label{fig:cocartesian_arrow}
	\end{subfigure}
	\caption{Cartesian and Cocartesian arrows}
\end{figure}

A \textbf{fibration} is merely a functor $\Phi\colon \sC \to \bS$ that has cartesian arrows, while an \textbf{op-fibration} only has cocartesian arrows. A \textbf{map of bifibrations} is a strictly commuting square of functors
\[ \xymatrix{
	\sC \ar[r]^-F \ar[d]_-{\pi_\sC} & \sD \ar[d]^-{\pi_\sD} \\
	\bS \ar[r]_-{F_\flat} & \bT
} \]
such that $F$ preserves cartesian arrows and cocartesian arrows, while $F_\flat$ preserves products and Beck-Chevalley squares.

A \textbf{symmetric monoidal bifibration (smbf)} is a bifibration $\pi\colon \sC \to \bS$ and a symmetric monoidal structure on $\sC$ with monoidal product denoted $\boxtimes$ so that 
\begin{itemize}
	\item $\pi$ is a strict symmetric monoidal functor,
	\item $\boxtimes$ is a map of fibrations, i.e. a tensor of two cartesian arrows is cartesian, and 
	\item $\boxtimes$ is a map of op-fibrations, i.e. a tensor of two cocartesian arrows is cocartesian.
\end{itemize}
We think of this monoidal product as ``external'' and we denote the unit by $I$. A {\bf map of symmetric monoidal bifibrations} is a map of bifibrations together with a strong symmetric monoidal structure on the functor $F\colon \sC \to \sD$, so that $F(X) \boxtimes F(Y) \to F(X \boxtimes Y)$ lies over the canonical map $F_\flat(A) \times F_\flat(B) \cong F_\flat(A \times B)$.

Intuitively, an smbf has three operations $\boxtimes$, $f^*$, $f_!$ that ``commute'' along canonical isomorphisms. For each pair of maps $f\colon A \to B$, $g\colon A' \to B'$, there is  a canonical isomorphism
\[ f^*X \boxtimes g^*Y \cong (f \times g)^*(X \boxtimes Y), \]
of functors $\sC^{B} \times \sC^{B'} \to \sC^{A \times A'}$, and a similar canonical isomorphism for pushforwards.

\begin{example}
	 Let $\bS$ be the category of unbased spaces. The objects of $\mc U$ are the arrows $X \to A$ in $\bS$, and maps are commuting squares. The projection $\mc U \to \bS$ sends $X \to A$ to $A$. A Cartesian arrow over $A \to B$ is a pullback square of spaces.  The pushforward of $X \to A$ along $A \to B$ is  the composite $X \to B$. It satisfies the Beck-Chevalley condition for all pullback squares.
$\mc U$ is a symmetric monoidal bifibration, with tensor product given by the Cartesian product, sending $X \to A$ and $Y \to B$ to $X \times Y \to A \times B$.
\end{example}

As we have already mentioned, an smbf contains all the raw ingredients needed to form a bicategory with a system of base-change objects. We assemble the operations $\odot$, $\sh{}$, $U_B$, and $\bcr AfB$ from these more basic pieces as follows.
\begin{align*}
- \odot_B - &\colon  \sC^{A \times B} \times \sC^{B \times C} \to \sC^{A \times C}
&		\sh{-}_B &\colon  \sC^{B \times B} \to \sC^{*} \\
M\odot_B N &= (\id_A\times\pi_B\times\id_C)_!(\id_A\times \Delta_B\times\id_C)^*(M \boxtimes N) 
&	\sh{M}_B &= (\pi_B)_!(\Delta_B)^* M
\end{align*}
\begin{align*}
U_B &\colon  {*} \to \sC^{B \times B} 
&	\bcr AfB &\colon  {*} \to \sC^{B \times A} \\
U_B &= (\Delta_B)_! \pi_B^* I
&\bcr AfB &= (f,\id_A)_! \pi_A^* I
\end{align*}

\begin{thm} \label{thm:indexed_bicat} Let  $\sC \to \bS$ be a  symmetric monoidal bifibration.
	\begin{itemize}
		\item The operations $-\odot -$ and $U_B$ above and the maps $\fa, \frl, \fr$ defined in \cite[Figures 5.6 and 5.7
]{mp2} define a bicategory $\calBi CS$. \cite[14.4, 14.11]{s:framed}
		\item The operation $\sh{-}$ above and $ \theta$ defined in \cite[Figure 5.8
]{mp2} define a shadow on $\calBi CS$. \cite[5.2]{PS:indexed}
		\item There is a pseudofunctor $[]\colon \bS \to \calBi CS$ that sends each morphism to the base-change object $\bcr AfB$.  \cite[Theorems 3.4 and 3.6
]{mp2}
		\item The bicategory $\calBi CS$ has a shadowed $n$-Fuller structure and a system of base-change objects given by $[]$. \cite[Theorem 3.6
]{mp2}
		\item Each map of symmetric monoidal bifibrations $F\colon (\sC,\bS) \to (\sD,\bT)$ induces a strong shadow functor $F\colon \calBi CS \to \calBi DT$, and an isomorphism $F \circ [] \cong [] \circ F_\flat$. \cite[Theorem 14.1
]{mp2}
	\end{itemize}
\end{thm}
The last bullet point in particular reduces the problem of building strong shadow functors $\Phi^H$ and $\iota_H^*$ to the problem of building maps of symmetric monoidal bifibrations.

Finally we discuss how to invert weak equivalences in a bifibration. Suppose $\pi\colon \sC \to \bS$ is a fibration, each fiber category $\sC^A$ has a subcategory of weak equivalences, and $\ho\sC$ is the category formally obtained by inverting these equivalences. By the universal property of $\ho\sC$ there is a functor $\ho\sC \to \bS$, that is in general not a fibration.

We say that $\pi$ is a \textbf{right-deformable fibration} if for each $\sC^A$ there is
\begin{itemize}
	\item a full subcategory $\sF^A$ ,
	\item a functor $R_A\colon \sC^A \to \sC^A$ with image in $\sF^A$, and 
	\item a weak equivalence $r_A\colon \id_{\sC^A} \xto{\sim} R_A$,
\end{itemize}
 such that 
\begin{itemize}
	\item$f^*\colon \sC^B \to \sC^A$ preserves weak equivalences on $\sF^B$, and 
	\item $f^*(\sF^B) \subseteq \sF^A$.
\end{itemize}
The following two results are proven by an elementary but tedious diagram-chase, that compares $\ho\sC$ to the Grothendieck construction formed from the right-derived pullback functors $f^*R_B\colon \ho(\sC^B) \to \ho(\sC^A)$. The full proof appears in \cite{spectra_notes}.

\begin{thm}\label{thm:homotopy_category_of_fibration}
	If $\pi$ is a right-deformable fibration then $\ho\sC \to \bS$ is a fibration, and the canonical maps $\ho(\sC^A) \to (\ho\sC)^A$ are isomorphisms of categories. Dually, if $\pi$ a left-deformable op-fibration then $\ho\sC \to \bS$ is an op-fibration.
\end{thm}

We call the cartesian arrows in $\ho\sC$ \textbf{homotopy cartesian} when we want to distinguish from the cartesian arrows in $\sC$.

\begin{prop}\label{prop:cartesian_in_homotopy_category}
	Suppose $\sC$ is a right deformable fibration. Then an arrow in $\ho\sC$ is homotopy cartesian \tiff it is isomorphic to a cartesian arrow in $\sC$ with fibrant target.
The dual statement applies to cocartesian arrows in a left-deformable op-fibration.
\end{prop}

\section{Parametrized $G$-spectra and fixed point functors}\label{sec:spectra}

By the last bullet point of \cref{thm:indexed_bicat}, it now remains to construct the smbf of parametrized $G$-spectra, and to prove that $\iota_H^*$ and $\Phi^H$ give smbf maps.

\subsection{On the nose}

Fix a finite group $G$ and an unbased left $G$-space $B$. Recall from \cite{ms} that there is a category $G\mc S(B)$ of \textbf{orthogonal $G$-spectra over $B$}, or equivalently \textbf{$\mathscr J_G$-spaces over $B$}. The objects are sequences that assign to each integer $n \geq 0$ a retractive $G \times O(n)$ space $X_n$ over $B$, together with $G$-equivariant structure maps $\Sigma_B X_n \to X_{1+n}$, satisfying the condition that the composite map $\Sigma^p_B X_q \to X_{p+q}$ is $O(p) \times O(q)$-equivariant. We always assume the base space $B$ is compactly generated weak Hausdorff, while $X_n$ only has to be compactly generated.

For each $G$-equivariant map of base spaces $f\colon A \to B$,  a map of orthogonal $G$-spectra over $f$ consists of commuting diagrams
\[ \xymatrix @R=1em @C=3em{
	A \ar[d] \ar[r]^-f & B \ar[d] \\
	X_n \ar[d] \ar[r]^-{\phi_n} & Y_n \ar[d] \\
	A \ar[r]^-f & B
} \]
in which $\phi_n$ is $G \times O(n)$-equivariant and commutes with the structure maps of $X$ and $Y$. This defines a larger category $G\mc S$ of all orthogonal $G$-spectra over all base spaces, whose fiber category over $B$ is $G\mc S(B)$. The projection functor to the category $G\bS$ of unbased $G$-spaces is a bifibration, with Beck-Chevalley along strict pullback squares \cite[11.4.8]{ms}. We therefore have adjoint pullback and pushforward functors
\begin{align*}
f^*\colon G\mc S(B) &\to G\mc S(A), \\
f_!\colon G\mc S(A) &\to G\mc S(B).
\end{align*}
The pullback $f^*$ is also a left adjoint, and therefore preserves all colimits.

There is an \textbf{external smash product} functor
\begin{align*}
\barsmash \colon G\mc S(A) \times G\mc S(B) &\to G\mc S(A \times B)
\end{align*}
defined for retractive $G$-spaces by the formula
\begin{equation}\label{external_smash}
\xymatrix{
	(X \times B) \cup_{A \times B} (A \times Y) \ar[r] \ar[d] & X \times Y \ar[d] \\
	A \times B \ar[r] & X \barsmash Y }
\end{equation}
and then extended to parametrized $G$-spectra using the Day convolution along $\mathscr J_G$. This can be regarded as a functor on the entire category of $G$-spectra, $G\mc S \times G\mc S \to G\mc S$. It preserves cartesian and cocartesian arrows, and extends to a symmetric monoidal structure, hence it makes the category $G\mc S$ a symmetric monoidal bifibration. The unit of $\barsmash$ is the sphere spectrum, regarded as a parametrized spectrum over the one-point space $*$.

Let $F_V K$ denote the free parametrized $\mathscr J_G$-space on a retractive $G$-space $K$ over $B$. Concretely, this is the $\mathscr J_G$-space whose value at $W$ is the external smash product $\mathscr J_G(V,W) \barsmash K$, where $\mathscr J_G(V,W)$ is regarded as a retractive space over $*$. Since pullback and pushforward commute with $\barsmash$, they also commute with free spectra (\cite[11.4.7]{ms}):
\[ f^*F_V K \cong F_V f^*K, \qquad f_!F_V K \cong F_V f_!K. \]

We will frequently use the class of ``freely $f$-cofibrant'' orthogonal spectra over $B$. A map of retractive spaces is a (closed, equivariant) \textbf{$f$-cofibration} if it is closed and has the fiberwise, unbased, equivariant version of the homotopy extension property. A spectrum is \textbf{freely $f$-cofibrant} if it is isomorphic to a cell complex spectrum built from maps of the form $F_V K \to F_V L$, where $K \to L$ is a (closed, equivariant) $f$-cofibration of $f$-cofibrant spaces over $B$. By ``cell complex'' we mean a sequential colimit of pushouts of arbitrary coproducts of such maps.

\begin{lem}\label{lem:pullback_cofibration}
	Freely $f$-cofibrant spectra are preserved by the pullback functor $f^*$.
\end{lem}

\begin{proof}
	This follows because $f^*$ preserves colimits, free spectra, and $f$-cofibrations of retractive spaces.
\end{proof}

Next we define three functors that change the group $G$. First, for each subgroup $H \leq G$, we can  forget the $G$-action and remember the action of $H$. This gives the \textbf{forgetful functor to $H$}. Second, if $H$ has Weyl group $WH = NH/H$, the $H$-fixed point subspaces $X_n^H$ form a $WH$-equivariant spectrum over the fixed point subspace $B^H$. This defines the \textbf{categorical fixed points} functor.  Finally, we define the \textbf{geometric fixed points} functor by the following coequalizer of $WH$-spectra. 
\[ \bigvee_{V,W} F_{W^{H}} S^0 \barsmash \mathscr J_G^{H}(V,W) \barsmash X(V)^{H} \rightrightarrows \bigvee_V F_{V^{H}} S^0 \barsmash X(V)^{H}  \ra \Phi^{H} X \]
In total, this gives three functors
\begin{align*}
\iota_H^*\colon G\mc S(B) &\to H\mc S(B) \\
(-)^H\colon G\mc S(B) &\to WH\mc S(B^H) \\
\Phi^H\colon G\mc S(B) &\to WH\mc S(B^H).
\end{align*}
The \textbf{restriction map} $r\colon X^H \to \Phi^H X$ assigns each level $X_n^H$ to the $\R^n$-term 
on the right-hand side of the coequalizer system. (A little diagram-chasing shows this gives a well-defined map of spectra.)

Each of these definitions extends to maps in $G\mc S$, giving commuting squares of functors
\[ \xymatrix{
	G\mc S \ar[d] \ar[r]^-{\iota_H^*} & H\mc S \ar[d] \\
	G\bS \ar[r]^-{\iota_H^*} & H\bS
} \qquad
\xymatrix{
	G\mc S \ar[d] \ar[r]^-{(-)^H} & WH\mc S \ar[d] \\
	G\bS \ar[r]^-{(-)^H} & WH\bS
} \qquad
\xymatrix{
	G\mc S \ar[d] \ar[r]^-{\Phi^H} & WH\mc S \ar[d] \\
	G\bS \ar[r]^-{(-)^H} & WH\bS
} \]
and $r$ lives over the identity transformation of $(-)^H\colon G\bS \to WH\bS$.

\begin{prop}
	Each of these functors is a map of bifibrations, if we restrict to freely $f$-cofibrant spectra. We can give each one a lax symmetric monoidal structure, commuting with the same structure on the other three functors in its square. Furthermore $\iota_H^*$ is strong symmetric monoidal, and $\Phi^H$ is strong on the subcategory of freely $f$-cofibrant spectra.
\end{prop}

\begin{proof}
It is elementary to check that the first two functors preserve cartesian and cocartesian arrows. For $\Phi^H$ this reduces to the same claim for external smash products, and the fact that $f^*$ preserves all colimits \cite[11.4.1]{ms}.

The symmetric monoidal structure on $\iota_H^*$ is given by the identity map on the underlying spectra, and the coherences are obviously satisfied.
For $(-)^H$ the symmetric monoidal structure map $X^H \barsmash Y^H \to (X \barsmash Y)^H$ is given by noticing that the inclusion to $X \barsmash Y$ lands in the $H$-fixed points, and $(\Sph_{(G)})^H \cong \Sph_{(WH)}$ is the unique isomorphism. The coherences are also straightforward. For $\Phi^H$ the map commuting it with smash product is constructed by the method of \cite[4.7]{mandell_may}, applied verbatim with smash products replaced by external smash products. We also use the same argument to prove this map is an isomorphism on free spectra, and therefore on freely $f$-cofibrant spectra by an induction up the skeleton of the cell complex. To check the coherence of this symmetric monoidal structure, it suffices to restrict attention to one fiber. Then it follows immediately from \cite[1.2]{malkiewich_thh_dx}.
\end{proof}

\subsection{The homotopy category}
A map of orthogonal $G$-spectra $X \to Y$ over $B$ is a \textbf{level equivalence} if each map $X(V) \to Y(V)$ is an equivalence on the $H$-fixed points for all subgroups $H \leq G$. There is a \textbf{level fibrant replacement} functor $R^{lv}$ that replaces each $X$ by a level equivalent spectrum $X \overset\sim\to R^{lv}X$  so that $(R^{lv}X(V))^H \to B^H$ is a quasifibration \cite[6.5.1 and 12.1.7]{ms}. 
A map $X \to Y$ is a \textbf{stable equivalence} if on each fiber over $b \in B$ the map $R^{lv} X \to R^{lv} Y$ is  an isomorphism on the $H$-equivariant stable homotopy groups for every $H \leq \textup{stab}_b \leq G$. This definition is independent of the choice of functor $R^{lv}$.

\begin{thm}\cite[12.3.10]{ms}\label{thm:stable_model_structure}
	There is a model structure on $G\mc S(B)$ where the weak equivalences are stable equivalences.  
\end{thm}

\begin{prop}\cite[12.6.7]{ms}
	The pullback functor $f^*$ is a Quillen right adjoint, and a Quillen equivalence if $f$ is a weak equivalence of $G$-spaces (i.e. $A^H \overset\sim\to B^H$ is a weak equivalence for all $H \leq G$). In fact, $f^*$ preserves all stable equivalences between spectra whose levels $X(V)^H$ are quasifibrations over $B^H$.
\end{prop}

This is the \textbf{stable $qf$-model structure}. The generating cofibrations are the free spectra on the {\bf $qf$-cells}, i.e. those maps $G/H \times (S^{k-1} \to D^k)$ over $B$ that are $f$-cofibrations. The generating acyclic cofibrations are the free spectra on maps of the form $G/H \times (D^{k-1} \to D^k)$ with a different cofibration condition, and also the $\barsmash$-pushout-products of generating cofibrations over $B$ and the maps $k_{V,W}$ over $*$ from \cite[III.4.6]{mandell_may}. We will not spell out the condition on $D^{k-1} \to D^k$ because it will not matter; it only matters that we fix the definition once and for all.

In particular, every $qf$-cofibrant spectrum is also freely $f$-cofibrant. The next two lemmas therefore show that $\Phi^H$ and $\barsmash$ preserve stable equivalences between pullbacks of such spectra. This was already done nonequivariantly for $\barsmash$ in \cite{coassembly}, but here we give a different argument that is easier to make equivariant.

\begin{lem}\label{lem:geo_fp_free_spectra}
	$\Phi^H$ preserves cofibrations, acyclic cofibrations, and stable equivalences between freely $f$-cofibrant spectra.
\end{lem}

\begin{proof}
	The proof that it preserves cofibrations and acyclic cofibrations is identical to the proof in the non-parametrized case \cite{mandell_may}, so we focus on the last claim.
	
	We freely use the fact that a pushout-product of $f$-cofibrations of retractive spaces is again an $f$-cofibration, and that the external smash product $K \barsmash K'$ of $f$-cofibrant spaces preserves weak equivalences. This implies that $F_V K \to F_V L$ is an $f$-cofibration on each spectrum level when $K \to L$ is an $f$-cofibration, and also that $F_V K \to F_V K'$ is a level equivalence when $K \to K'$ is an equivalence of $f$-cofibrant spaces.
	
	It suffices to show that for a freely $f$-cofibrant spectrum $X$, there is some $qf$-cofibrant spectrum $X'$ and stable equivalence $X' \to X$ such that $\Phi^H X' \to \Phi^H X$ is an equivalence. Let $X^{(n)}$ denote the $n$-skeleton of $X$, meaning the target of the $n$th map in the sequential colimit system that defines $X$. By induction on $n$, we build two cofibrant spectra $X^{[n-1/2]}$ and $X^{[n]}$, fitting into a diagram
	\[ \xymatrix{
		X^{[n-1]} \ar[r]^-\sim \ar[d]^-\sim & X^{[n-1/2]} \ar[r] \ar[d]^-\sim & X^{[n]} \ar[d]^-\sim \\
		X^{(n-1)} \ar@{=}[r] & X^{(n-1)} \ar[r] & X^{(n)}
	} \]
	where the $\sim$ maps are level equivalences and the top row consists of $qf$-cofibrations of $qf$-cofibrant spectra. The colimit over $n$ is a homotopy colimit on each spectrum level and therefore $\colim_n X^{[n]} \to X$ is a level equivalence of spectra. Then we prove that $\Phi^H X^{[n]} \to \Phi^H X^{(n)}$ is an equivalence. Since $\Phi^H$ preserves free $f$-cofibrations, and pushouts and sequential colimits along such, this implies that $\Phi^H \colim_n X^{[n]} \to \Phi^H X$ is an equivalence, as desired.

	Now we build these spectra. For each of the maps $K \to L$ appearing at stage $n$ of the colimit system for $X$, factor $B \to K$ into a $qf$-cell complex $B \to K' \xto\sim K$, then factor $K' \to L$ into another $qf$-cell complex $K' \to L' \xto\sim L$. Because $K'$ is a cell complex relative to $B$ and $X^{[n-1]}_V \to X^{(n-1)}_V$ is a weak equivalence of $G$-spaces, the map $K' \to K \to X^{(n-1)}_V$ can be modified by a homotopy rel $B$ to a map that lifts to $X^{[n-1]}_V$. This data together gives a map from the mapping cylinder of $K' \to L'$ rel $B$ into $X^{(n-1)}_V$ for which the front end lifts to a map $K' \to X^{[n-1]}_V$. We define a projection map from the mapping cylinder back to $B$,  by composing this map with the projection $X^{(n-1)}_V \to B$ so that we have a map of retractive spaces over $B$.
	
	Form $X^{[n-1/2]}$ by attaching the cylinder part of this mapping cylinder to $X^{[n-1]}$, and $X^{[n]}$ by attaching the entire mapping cylinder. The maps in the diagram above are then clear. The levels of $X^{[n-1/2]}$  deformation retract onto $X^{[n-1]}$. Using the fact that external smash product preserves equivalences of $f$-cofibrant spaces, $X^{[n]} \to X^{(n)}$ is also a level equivalence. 
	
	Now apply $\Phi^H$ to the construction of $X^{[n]}$. We get the same equivalences as before, except possibly the one all the way on the right. Before $\Phi^H$, it is a map of pushouts of the form
	\[ \xymatrix{
		X^{[n-1/2]} \ar[d]^-\sim & \ar[l] \coprod F_{V_\alpha} K_\alpha' \ar[d]^-\sim \ar[r] & \coprod F_{V_\alpha} L_\alpha' \ar[d]^\sim \\
		X^{(n-1)} & \ar[l] \coprod F_{V_\alpha} K_\alpha \ar[r] & \coprod F_{V_\alpha} L_\alpha.
	} \]
	After $\Phi^H$, it is a map of pushouts of the form
	\[ \xymatrix{
		\Phi^H X^{[n-1/2]} \ar[d]^-\sim & \ar[l] \coprod F_{V_\alpha^H} (K_\alpha')^H \ar[d]^-\sim \ar[r] & \coprod F_{V_\alpha^H} (L_\alpha')^H \ar[d]^\sim \\
		\Phi^H X^{(n-1)} & \ar[l] \coprod F_{V_\alpha^H} K_\alpha^H \ar[r] & \coprod F_{V_\alpha^H} L_\alpha^H.
	} \]
	The left-hand vertical map is an equivalence by inductive hypothesis. The middle vertical map is an equivalence because $(K_\alpha')^H \to K_\alpha^H$ is an equivalence of equivariantly $f$-cofibrant $WH$-spaces, and similarly for the right-hand vertical map. The horizontal maps on the right-hand side are also $f$-cofibrations on each spectrum level, because $K_\alpha^H \to L_\alpha^H$ and $(K_\alpha')^H \to (L_\alpha')^H$ are both $WH$-equivariant $f$-cofibrations. Therefore the map of pushouts $\Phi^H X^{[n]} \to \Phi^H X^{(n)}$ is a level equivalence of $WH$-spectra, completing the induction.
\end{proof}

\begin{lem}\label{lem:external_smash_product_freely_f_cofibrant}
	$\barsmash$ preserves stable equivalences between freely $f$-cofibrant spectra.
\end{lem}

\begin{proof}
	The proof is essentially the same as the previous lemma. It suffices to take two freely $f$-cofibrant spectra $X$ and $Y$, build the spectra $X^{[n-1/2]}$, $X^{[n]}$ as in that argument, and to show that $\left(\colim_n X^{[n]}\right) \barsmash Y \to X \barsmash Y$ is a level equivalence. Then we could do the same with the roles of $X$ and $Y$ swapped, and conclude that $QX \barsmash QY \to X \barsmash Y$ is an equivalence.
	
	We first observe that pushout-products of spectra constructed with $\barsmash$ preserve free $f$-cofibrations; this follows from the same statement for spaces and the formal fact that smash products of free spectra are free. We already used in the previous proof that free $f$-cofibrations are level $f$-cofibrations. Using this, we can prove that if $K' \to K$ is an equivalence of $f$-cofibrant spaces then $F_{V_\alpha} K' \sma Y \to F_{V_\alpha} K \sma Y$ is an equivalence. We observe that $F_{V_\alpha} K \sma -$ turns the cell complex structure of $Y$ into a new cell complex structure in which the representations all have $V_\alpha$ added to them, and the spaces all have $K$ smashed into them. By the pushout-product property, the pushout squares are all along free $f$-cofibrations of spectra, which are $f$-cofibrations on each level. For each space $A$ occurring in the cell complex structure of $Y$, the equivalence of $f$-cofibrant spaces $K' \barsmash A \to K \barsmash A$ gives a level equivalence of free spectra. Hence every one of the pushout squares is changed by a level equivalence when we pass from $K'$ to $K$; hence
	\[ F_{V_\alpha} K' \sma Y \xto{\sim} F_{V_\alpha} K \sma Y \]
	is an equivalence.

	Since free $f$-cofibrations are level $f$-cofibrations, $\colim_n (X^{[n]} \barsmash Y) \cong \left(\colim_n X^{[n]}\right) \barsmash Y$ is a homotopy colimit. It therefore suffices to prove by induction that $X^{[n]} \barsmash Y \to X^{(n)} \barsmash Y$ is an equivalence. For the inductive step we have the diagram
	\[ \xymatrix{
		X^{[n-1]} \sma Y \ar[r]^-\sim \ar[d]^-\sim & X^{[n-1/2]} \sma Y \ar[r] \ar[d]^-\sim & X^{[n]} \sma Y \ar[d] \\
		X^{(n-1)} \sma Y \ar@{=}[r] & X^{(n-1)} \sma Y \ar[r] & X^{(n)} \sma Y
	} \]
	where the marked $\sim$ on the top is the deformation retract of $X^{[n-1/2]}$ onto $X^{[n-1]}$. We just need to see that the vertical on the right is an equivalence. Before smashing with $Y$, it is a map of pushouts of the form
	\[ \xymatrix{
		X^{[n-1/2]} \ar[d]^-\sim & \ar[l] \coprod F_{V_\alpha} K_\alpha' \ar[d]^-\sim \ar[r] & \coprod F_{V_\alpha} L_\alpha' \ar[d]^\sim \\
		X^{(n-1)} & \ar[l] \coprod F_{V_\alpha} K_\alpha \ar[r] & \coprod F_{V_\alpha} L_\alpha.
	} \]
	After $\barsmash Y$, the horizontal maps of the  right-hand square are level cofibrations, by the pushout-product property for free $f$-cofibrations. The vertical maps are equivalences by inductive hypothesis and the intermediate lemma we established earlier in the proof. Therefore the map of pushouts is an equivalence, and the induction is complete.
\end{proof}

Now we pass to the homotopy category by inverting all the stable equivalences in $G\mc S$. By \cref{thm:homotopy_category_of_fibration} the resulting category $\ho G\mc S$ is a fibration and op-fibration whose base category is the category $G\bS$ of $G$-spaces. By \cref{prop:cartesian_in_homotopy_category}, an arrow in $\ho G\mc S$ is homotopy cocartesian \tiff it is isomorphic to a cocartesian arrow $X \to f_!X$ in which $X$ is cofibrant. An arrow is homotopy cartesian \tiff it is isomorphic to a cartesian arrow $f^*Y \to Y$ with $Y$ fibrant.

\begin{rmk}
	This homotopy category is the homotopy category of the ``integral model structure'' of \cite{harpaz_prasma}, but taking the weak equivalences in the base category to be the isomorphisms, rather than the weak homotopy equivalences. In other words, it retains information about the base space up to homeomorphism, but the fiber spectra are remembered only up to stable equivalence.
\end{rmk}

Since every cell and acyclic cell in $G\mc S(A)$ pushes forward along $f$ to a cell or acyclic cell in $G \mc S(B)$, the cofibrant replacements $Q_A$ built using the small-object argument assemble into a single functor $Q\colon G \mc S \to G\mc S$. Using $Q$, we can left-derive the external smash product functor
\[ \barsmash\colon G\mc S \times G\mc S \to G\mc S. \]
This makes $\ho G\mc S$ into a symmetric monoidal category. More concretely, the tensor product is $\barsmash^{\bL} = (Q-) \barsmash (Q-)$, with associator, unitor, and symmetry isomorphism given by deleting all copies of $Q$ that are not applied to the inputs (for instance one that is applied to the output of $\barsmash$), applying the corresponding isomorphism for $\barsmash$, and then re-inserting the extra copies of $Q$. Since any two left-derivations of a functor are canonically isomorphic, we can be assured that if we had chosen a different model structure we would get an isomorphic symmetric monoidal category.

\begin{thm}\label{thm:spectra_smbf}
	This symmetric monoidal structure makes $\ho G\mc S$ into a symmetric monoidal bifibration, with Beck-Chevalley for every homotopy pullback square of $G$-spaces.
\end{thm}

\begin{proof}
	The projection to the base category $G\bS$ is still strict symmetric monoidal because the map $QX \xto{\sim} X$ lies over the identity of $G\bS$. The Beck-Chevalley property for pullback squares with one leg a fibration is \cite[9.9]{shulman2011comparing}, building on \cite[Thm 13.7.7]{ms}; see also \cite{spectra_notes}. 
	For a commuting square of spaces where two of the parallel sides are weak equivalences, we also get the Beck-Chevalley property because each component of the Beck-Chevalley map is an isomorphism as functors of homotopy categories. We then deduce the Beck-Chevalley property for an arbitrary homotopy pullback square using the usual pasting lemma.
	
	It remains to show that $\barsmash^\bL$ preserves cocartesian arrows and cartesian arrows. In principle, this should be citable away to \cite{ms}, but it is difficult to work directly with their construction of the symmetric monoidal structure on the pullback functors $f^*$. We instead start with the ``canonical'' one defined just above.\footnote{It is clear that this has the expected behavior on suspension spectra, which is all we need for the applications anyway.}
	
	Take any two homotopy cocartesian arrows in $\ho G\mc S$. Up to isomorphism, they are cocartesian arrows $X \to f_!X$ and $Y \to g_!Y$ in the point-set category $G\mc S$ with $X$ and $Y$ cofibrant. On these inputs we have an equivalence $(Q-)\barsmash(Q-) \simeq - \barsmash -$, so the derived product $\barsmash^\bL$ of these arrows in the homotopy category is isomorphic to their actual product $\barsmash$, which is cocartesian in $G\mc S$ and still has a cofibrant source, hence is homotopy cocartesian. Therefore $\barsmash^\bL$ preserves homotopy cocartesian arrows.
	
	Now take any two homotopy cartesian arrows in $\ho G\mc S$. Up to isomorphism, they are cartesian arrows in the point-set category $G\mc S$ whose targets are both cofibrant and fibrant. Let us call them $f^*X \to X$ and $g^*Y \to Y$. Form the following commuting diagram in $G\mc S$.
	\[ \xymatrix @R=2em @C=2em{
		Qf^*X \barsmash Qg^*Y \ar[r] \ar[d]_-\sim^-{(1)} & QX \barsmash QY \ar[d]^-\sim \\
		f^*X \barsmash g^*Y \ar[r] \ar[d]_-\sim^-{(5)} & X \barsmash Y \ar[d]^-\sim_-{(2)} \\
		f^*PX \barsmash g^*PY \ar[r] \ar@{<->}[d]_-\cong & PX \barsmash PY \ar@{<->}[d]^-\cong \\
		(f \times g)^*P(X \barsmash Y) \ar[r] \ar[d]_-\sim^-{(3)} & P(X \barsmash Y) \ar[d]^-\sim \\
		(f \times g)^*PR(X \barsmash Y) \ar[r] \ar@{<-}[d]_-\sim^-{(4)} & PR(X \barsmash Y) \ar@{<-}[d]^-\sim \\
		(f \times g)^*R(X \barsmash Y) \ar[r] & R(X \barsmash Y)
	} \]
	Here $P$ is the functor from \cite{coassembly}; it pulls back $X$ and then pushes it forward along the two evaluation maps $B^I \rightrightarrows B$. The maps $(1)$ and $(2)$ are equivalences by \cref{lem:external_smash_product_freely_f_cofibrant} and the fact that $f^*$ and $P$ preserve freely $f$-cofibrant spectra. The equivalences $(3)$, $(4)$ are because $f^*$ preserves the stable equivalences between spectra whose levels are quasifibrant, and this class includes both $RX$ and $PX$ when $X$ is cofibrant. The equivalence $(5)$ uses all of these facts together. Therefore $\barsmash^\bL$ preserves homotopy cartesian arrows.
\end{proof}

\begin{prop}\label{lem:forget_symmetric_monoidal}
	The functor $\iota_H^*\colon \ho G\mc S(B) \to \ho H\mc S(B)$ extends to a map of symmetric monoidal bifibrations $\ho G\mc S \to \ho H\mc S$ over the functor $\iota_H^*\colon G\bS \to H\bS$.
\end{prop}

\begin{proof}
	The functor $\iota_H^*$ clearly makes sense on all of $G\mc S$ and preserves all equivalences, therefore directly passes to a functor $\ho G\mc S \to \ho H\mc S$. The point-set functor preserves cartesian arrows, cocartesian arrows, cofibrant objects, and fibrant objects (because it is right Quillen). Therefore when viewed as a derived functor, it preserves homotopy cartesian arrows and homotopy cocartesian arrows. Since $\iota_H^*$ strictly commutes with $\barsmash$, it commutes with $\barsmash^\bL$ up to isomorphism by deleting the $Q$s, applying the commutation, then re-inserting the $Q$s. (Notice we have to do this because cofibrant replacement for $H$-spectra may not be $\iota_H^*$ of the cofibrant replacement functor for $G$-spectra.) Again since $QX \xto{\sim} X$ is over the identity in the base, this isomorphism lies over the corresponding isomorphism of spaces $\iota_H^*(A) \times \iota_H^*(A') \cong \iota_H^*(A \times A')$ in the base category. This gives $\iota_H^*$ the structure of a symmetric monoidal functor of homotopy categories, whose coherences follow from the same coherences on the point-set level. 
\end{proof}

Although $\Phi^H$ is not a left adjoint, by \cref{lem:geo_fp_free_spectra} it preserves cofibrations and acyclic cofibrations. It therefore also has a left-derived functor $\bL\Phi^H = \Phi^H Q$.

\begin{prop}\label{lem:geo_fp_symmetric_monoidal}
	The functor $\bL\Phi^H\colon \ho G\mc S(B) \to \ho WH\mc S(B)$ extends to a map of symmetric monoidal bifibrations $\ho G\mc S \to \ho WH\mc S$ over the functor $(-)^H\colon G\bS \to WH\bS$.
\end{prop}

\begin{proof}
	As above, the structure of $\bL\Phi^H$ as a symmetric monoidal functor is obtained by deleting all extraneous copies of $Q$, applying the same structure for $\Phi^H$, then re-inserting $Q$. Again, this gives a symmetric monoidal structure that lies over the canonical one on $(-)^H\colon G\bS \to WH\bS$, and its coherences follow from the same coherences in the point-set category for $\Phi^H$.
	
	The functor $\Phi^H$ preserves both cocartesian and cartesian arrows because the same is true for the smash products, fixed points, and colimits that make up its definition.\footnote{This is under the convention that parametrized spectra are built from compactly generated spaces ($k$-spaces) that are not necessarily weak Hausdorff. The argument still works if we work entirely in weak Hausdorff spaces, but it takes longer to argue that $f^*$ preserves the colimits that make up $\Phi^H X$ when $X$ is freely $f$-cofibrant.} Since it also preserves cofibrant objects, and $\bL\Phi^H \cong \Phi^H$ on the homotopy category of cofibrant objects, this implies $\bL\Phi^H$ preserves homotopy cocartesian arrows.
	
	However $\Phi^H$ does not preserve fibrant objects, so for homotopy cartesian arrows we have to work a little harder. Start with a point-set cartesian arrow $f^*X \to X$ in which $X$ is cofibrant and fibrant. 
	\[ \xymatrix @R=1.5em @C=2em{
		\Phi^H Qf^*X \ar[r] \ar[d]_-\sim & \Phi^H QX \ar[d]^-\sim \\
		\Phi^H f^*X \ar[r] \ar[d]_-\sim & \Phi^H X \ar[d]^-\sim \\
		\Phi^H f^*PX \ar[r] \ar@{<->}[d]_-\cong & \Phi^H PX \ar@{<->}[d]^-\cong \\
		f^* P\Phi^H X \ar[r] \ar[d]_-\sim & P\Phi^H X \ar[d]^-\sim \\
		f^* PR\Phi^H X \ar[r] \ar@{<-}[d]_-\sim & PR\Phi^H X \ar@{<-}[d]^-\sim \\
		f^* R\Phi^H X \ar[r] & R\Phi^H X
	} \]
	The weak equivalences in the top half follow from \cref{lem:geo_fp_free_spectra}, because the class of cofibrant spectra described in that lemma is preserved by pullback and by $P$. The isomorphisms in the middle follow because $P$ is a composition of a pullback and a pushforward. The arrow at the bottom is homotopy cartesian, hence so is the arrow at the top. Therefore $\bL\Phi^H$ preserves homotopy cartesian arrows.
\end{proof}

\section{Change of groups for the Reidemeister trace}\label{sec:grand_finale}
Combining \cref{thm:indexed_bicat,lem:forget_symmetric_monoidal,lem:geo_fp_symmetric_monoidal} gives the following result.  

\begin{thm}\label{thm:forget_fixed_strong_shadow} If $H$ is a subgroup of a finite group $G$, 
$\iota_H^*$ and $\Phi^H$ are strong shadow functors on $G\Ex$.
\end{thm}

\cref{bicategory_map_preserves_traces,thm:forget_fixed_strong_shadow}  imply that 	if $X$ is any finitely dominated $G$-CW complex and $f: X \to X$ any $G$-equivariant self-map, there are isomorphisms in the homotopy category
\begin{align*}
	\iota_H^* R_G(f) &\cong R_H(\iota_H^* f) \\
	\Phi^H R_G(f) &\cong R_{WH}(f^H).
\end{align*}
Tracing through the constructions shows that these come about through familiar isomorphisms on the source and target, for instance the isomorphism $\Phi^H \Sigma^\infty_+ \Lambda^f X \cong \Sigma^\infty_+ \Lambda^{f^H} X^H$.

\begin{cor}\label{cor:they_preserve_reidemeister_traces}
If $G$ is a finite group, $H$ is a normal subgroup of $G$, $Y$ is a $G$-space and $\phi\colon Y\to Y$ is a $G$-equivariant map, then 
\[\xymatrix@C=40pt@R=40pt{
		&&\Sph \ar@/_14pt/[dll]_-{R_G(\phi)^{G}} 
		\ar@/_7pt/[dl]_(.6){R_{WH}(\phi^H)^{G/H}} 
		\ar[d]^-{R(\phi^H)} 
		\\
		(\Sigma^\infty_+ \Lambda^{\phi} Y)^{G} \ar[r]^-{R} &
		(\Sigma^\infty_+ \Lambda^{\phi^H} (Y^H))^{G/H} 
		 \ar[r]^-{F}
		&	\Sigma^\infty_+ \Lambda^{\phi^H} (Y^H)
}\]
commutes up to homotopy.  
\end{cor}

\begin{proof}
We describe the argument for the left triangle.  The right triangle is similar and more straightforward.

Let $E$ and $E'$ be fibrant replacements of $\Sigma^\infty_+ \Lambda^{\phi} Y$ and $\Sigma^\infty_+ \Lambda^{\phi^H} Y^H$, respectively, in orthogonal $G$-spectra.   
The underived versions of $(-)^H$ and $\Phi^H$ define a diagram of orthogonal spectra
\[ \xymatrix@C=50pt{
	\Sph \ar[d]_-\cong \ar[rr]^-\cong && \Phi^{H} \Sph \ar[d]^-\cong 
		\\
	\Sph^{G} \ar[d]_-{\left(R_{G}(\phi)\right)^{G}} \ar@{=}[r] 
	& \left(\Sph^{H}\right)^{G/H} \ar[d]^-{\left(\left(R_{G}(\phi)\right)^{H}\right)^{G/H}} \ar[r]^-r_-\cong 
	& \left(\Phi^{H} \Sph\right)^{G/H} \ar[d]^-{\left(\Phi^{H} R_{G}(\phi)\right)^{G/H}} \ar[r]^-\cong
	& \left(\Sph\right)^{G/H} \ar[d]^-{\left(R_{WH}(\phi^H)\right)^{G/H}}
		\\
	E^{G} \ar@{=}[r] \ar@/_2em/[rrr]_-R
	& \left(E^{H}\right)^{G/H} \ar[r]^-r 
	& \left(\Phi^{H} E\right)^{G/H} \ar[r]^-\sim
	& \left(E'\right)^{G/H}
} \]
The unlabeled  $\cong$s exist and the top region commutes because $\Sph$ has a unique automorphism. The right-hand region is the agreement of $\Phi^{H} R_{G}(\phi)$ with $R_{WH}(\phi^H)$ along $\Phi^H E \simeq E'$ lying under $\Sigma^\infty_+ \Lambda^{\phi} Y \cong \Sigma^\infty_+ \Lambda^{\phi^H} Y^H$ in the homotopy category. The bottom region is the definition of $R$ for equivariant suspension spectra, cf. \cite[\S 6]{dotto2017comparing}, \cite[\S 2.5]{madsen_survey}.
\end{proof}

\cref{thm:main_second_half} follows by 
taking $G=C_n$, $H=C_k$, $Y=X^n$ and $\phi=\ful{f}n$.

\part{The fiberwise generalization}\label{part:fiberwise}
One of the primary strengths of our approach to \cref{thm:main_first_half} and \cref{thm:main_second_half} is that it applies in a range of categories. We will illustrate this by extending the results to the fiberwise setting.

\section{Spectra over fibrations over $B$}
Fix an unbased space $B$ and let $\bS_B$ be the category whose objects are Hurewicz fibrations $A \to B$ and whose maps are maps of spaces over $B$. This has a forgetful functor to spaces $\bS_B \to \bS$ that forgets the map to $B$.

We construct a new symmetric monoidal bifibration $\ho\mc S_{(B)}$ by pulling back $\ho\mc S$ along this functor $\bS_B \to \bS$. It is standard that this gives a bifibration, in which an arrow is (co)cartesian \tiff its image in $\ho\mc S$ is (co)cartesian. The symmetric monoidal structure is a little more subtle, but follows from the following lemma.

\begin{lem}
	\label{lem:get_rid_of_ex_b}
	Suppose $F\colon \bT \to \bS$ is a functor of cartesian monoidal categories, and that $\bS$ and $\bT$ are endowed with a class of Beck-Chevalley squares, preserved by $F$, such that for any two maps $A \to A'$, $B \to B'$ in $\bT$ the square
	\[ \xymatrix{
		F(A \times B) \ar[d] \ar[r] & F(A) \times F(B) \ar[d] \\
		F(A' \times B') \ar[r] & F(A') \times F(B')
	} \]
	is Beck-Chevalley in $\bS$. Then for any smbf $\mc A$ over $\bS$, the pullback category $F^*\mc A$ can be naturally given the structure of an smbf. 
\end{lem}

\begin{proof}
	This is essentially a generalization of the proof of \cite[12.8]{s:framed}: the product $\otimes$ in $F^*\mc A$ is defined as a pullback of the product $\boxtimes$ in $\mc A$ along the canonical map
	\[ F(A \times B) \to F(A) \times F(B). \]
	The proof that this preserves cocartesian arrows reduces to the Beck-Chevalley condition in the statement of the lemma, and the proof that it preserves cartesian arrows is easier. We produce the rest of the symmetric monoidal structure for $\otimes$ 
	by lifting the same structure from $\boxtimes$, using the universal property of cartesian arrows. A more explicit treatment appears in \cite{spectra_notes}.
\end{proof}

\begin{example}
	The functor $\bS_B \to \bS$ satisfies the statement of the lemma because for any two maps $A \to A'$ and $E \to E'$ of fibrations over $B$, the following square is homotopy pullback.
	\[ \xymatrix{
		A \times_B E \ar[d] \ar[r] & A \times E \ar[d] \\
		A' \times_B E' \ar[r] & A' \times E'
	} \]
	\begin{itemize}
		\item Pulling back $\ho\mc U$ along $\bS_B \to \bS$ gives an smbf $\ho\mc U_{(B)}$ whose objects are pairs of maps $X \to A \to B$ where $A \to B$ is a fibration, and morphisms are maps over $B$. The product with $Y \to A' \to B$ is the fiber product $X \times_B Y \to A \times_B A' \to B$.
		\item Pulling back $\ho\mc S$ along $\bS_B \to \bS$ gives an smbf $\ho\mc S_{(B)}$ whose  objects are pairs of a fibrations $A \to B$ and a spectrum $X$ over $A$.  Morphisms are a map $A \to A'$ over $B$ and  a map of spectra $X \to Y$ over $A \to A'$. The pullback and pushforward are defined as in $\ho\mc S$, and the smash product is the relative external smash product, given by pulling back $X \barsmash^\bL Y$ from $A \times A'$ to the fiber product $A \times_B A'$.
		\item Both of these generalize to $G$-spaces and $G$-spectra, giving $\ho G\mc U_{(B)}$ and $\ho G\mc S_{(B)}$. We always assume that $A \to B$ is a Hurewicz fibration whose path-lifting function is $G$-equivariant.
	\end{itemize}
\end{example}

The bicategory $\Ex_B^{\fibra}$ of spectra over fibrations over $B$, is the bicategory associated to $\ho\mc S_{(B)}$, compare \cite[19.2.6, 19.3.4]{ms}. Performing the same operation for $G$-equivariant Hurewicz fibrations $A \to B$ and $G$-equivariant spectra gives another bicategory $G\Ex_B^{\fibra}$.
The following is a corollary of \cref{lem:forget_symmetric_monoidal,lem:geo_fp_symmetric_monoidal,lem:get_rid_of_ex_b,thm:indexed_bicat}.
\begin{cor}\label{thm:fiberwise_forget_fixed_strong_shadow} If $H$ is a subgroup of a finite group $G$, 
$\iota_H^*$ and $\Phi^H$ are strong shadow functors on $G\Ex_B^{\fibra}$.  
\end{cor}

The bicategory $\Ex_B^{\fibra}$ has an $n$-Fuller structure and a system of base-change objects by \cref{thm:spectra_smbf,lem:get_rid_of_ex_b,thm:indexed_bicat}. In particular there is a pseudofunctor $[]_B\colon {\bS_B}\to  \bicat {U_B}{\bS_B}$
and 
coherent isomorphisms
\begin{align}\label{eq:compositions_of_fiberwise_base_change_in_spectra}
      m_{[]}\colon \bcr{Y}{g}{Z}_B \odot \bcr{X}{f}{Y}_B &\xto{\sim} \bcr{X}{g \circ f}{Z}_B, \qquad
      i_{[]}\colon U_X \xto{\sim} \bcr{X}{\id}{X}_B.
\end{align}
The same applies with $G$-equivariant spaces as well.

If $p\colon E\to B$ is a perfect fibration, i.e. a (Hurewicz) fibration with finitely dominated fibers, 
$\bcr EpB_B$ is right dualizable as a 1-cell in $\Ex_B^{\fibra}$. The same is true equivariantly if $B$ has a trivial action and the fibers of $p$ are equivariantly finitely dominated. Therefore for each commuting square
\[\xymatrix{
	E\ar[d]_p\ar[r]^f
	&E\ar[d]^p
	\\
	B\ar@{=}[r]
	&B
}\]
we can define fiberwise versions of the traces from \cref{sec:traces_in_bicat}.
\begin{itemize}
	\item The {\bf fiberwise Lefschetz number} $L_B(f)$ is the trace of $f$ as a map in the symmetric monoidal category of spectra over $B$, in other words $\Ex_B^{\fibra}(B,B)$.	It is a self-map of the fiberwise sphere spectrum $\Sph_B = \Sigma^\infty_{+B} B$ in the homotopy category of spectra over $B$.
	\item The \textbf{pretransfer} is the trace of $f \times_B \id\colon E \to E \times_B E$, which is a slight refinement of $L_B(f)$. This gives a map $\Sph_B \to \Sigma^\infty_{+B} E$. When $f = \id$, this is the Becker-Gottlieb pretransfer \cite[\S 5]{becker1976transfer}.
	
	\item The {\bf fiberwise Reidemeister trace} $R_B(f)$ is the trace of the canonical isomorphism in $\Ex_B^{\fibra}$
	\begin{align}\label{fiberwise_foverBE}
		\bcr EpB_B&\overset\sim\to \bcr EpB_B\odot \bcr EfE_B.
	\end{align}
	It gives a map in the homotopy category of spectra over $B$,
		\[ R_B(f)\colon \Sph_B \simeq \Sigma^\infty_{+B} \Bigsh{\bcr {B}{=}{B}_B} \to \Sigma^\infty_{+B} \Bigsh{\bcr {E}{f}{E}_B} \simeq \Sigma^\infty_{+B} \Lambda^{f}_B E \]
	which is $R(f_b)$ on each fiber.

The fiberwise Reidemeister trace $R_B(f)$ is a complete obstruction to the removal of fixed points from a family of maps $f$, provided $B$ is a cell complex of dimension $d$, and $p$ is a fiber bundle whose fibers $X$ are compact manifolds of dimension at least $d+3$ \cite{kw}. 
	
	\item The {\bf fiberwise $n$th Fuller trace} $R_{B, C_n}(\fful fBn)$ is the trace of the map 
		\[ \bcr {\fib EBn}{\fib pBn}{B}_B \xto{\sim} \bcr {\fib EBn}{\fib pBn}{B}_B \odot \bcr {\fib EBn}{\fful fBn}{\fib EBn}_B \] in $C_n\Ex_B^{\fibra}$
	arising from the commuting square
	 \[\xymatrix{
	 	\fib EBn\ar[r]^{\fful fBn}\ar[d]^{\fib pBn}
	 	&\fib EBn\ar[d]^{\fib pBn}
	 	\\
	 	B\ar@{=}[r]
	 	& B
	 }\]
	It is a map in the homotopy category of $C_n$-equivariant spectra over $B$
		\[ R_{B, C_n}(\fful fBn)\colon \Sph_B \simeq \Sigma^\infty_{+B} \Bigsh{\bcr {B}{=}{B}_B} \to \Sigma^\infty_{+B} \Bigsh{\bcr {E^n}{\fful fBn}{E^n}_B} \simeq \Sigma^\infty_{+B} \Lambda_B^{\fful fBn} \fib EBn \]
	which is $R_{C_n}(\ful {f_b}n)$ on each fiber.
\end{itemize}

We can now state the promised fiberwise version of \cref{thm:simple_main_thm}.
\begin{thm}[Fiberwise version of \cref{thm:simple_main_thm}]\label{thm:fiberwise_simple_main_thm}  
	The following diagram commutes up to fiberwise homotopy.
\[\xymatrix@C=25pt@R=40pt{
	&&\Sph_B\ar[dll]_-{R_B(\fful{f}{B}n)^{C_n}} 
		\ar[dl]^-{R_B(\fful{f}{B}k)^{C_k}} 
		\ar[d]^(.6){R_B(\fful{f}{B}k)} 
		 \ar[rd]^{R_B(f^k)} 
	\\
	(\Sigma^\infty_{+B} \Lambda_B^{\fful{f}{B}n} \fib EBn)^{C_n} \ar[r]^-R
	&(\Sigma^\infty_{+B} \Lambda_B^{\fful{f}{B}k} \fib EBk)^{C_k} \ar[r]^-{F} 
	&\Sigma^\infty_{+B} \Lambda_B^{\fful{f}{B}k} \fib EBk \ar[r]^-{\simeq} 
	&\Sigma^\infty_{+B} \Lambda_B^{f^k} E
}\]
\end{thm}

Note these are all maps of fibrations over $B$ that on each fiber capture the simpler maps we constructed earlier.

\begin{proof}
The right-hand triangle is just \cref{cor:Fuller_trace_is_trace_of_composite} applied to the bicategory $\Ex_B^{\fibra}$. The remaining two triangles are proven by restating the proof of \cref{cor:they_preserve_reidemeister_traces} in the category of $G$-spectra over $B$, and then taking $G=C_n$, $H=C_k$, $Y=X^{\times_Bn}$ and $\phi=\fful{f}{B}n$.
\end{proof}

Our list of fixed point invariants that can be identified using this approach is far from exhaustive. We leave the adaptation of this theorem to the remaining generalizations of $L(f)$ and $R(f)$ to the interested reader.

\bibliographystyle{amsalpha2}
\bibliography{references}
\newpage

\end{document}